\documentclass[a4paper, 11pt, abstract=true]{scrartcl}

\usepackage{amsmath}
\usepackage{amsthm}
\usepackage{amsfonts}
\usepackage{amssymb}
\usepackage{bbm}
\usepackage[pdftex]{graphicx}
\usepackage{enumitem}

\usepackage[backend=bibtex, style=authoryear]{biblatex}
\bibliography{literatur}
\usepackage[sc]{mathpazo}
\usepackage[T1]{fontenc}
\usepackage{microtype}
\usepackage{pdflscape}
\usepackage{afterpage}
\usepackage{booktabs}
\usepackage{tikz}
\usetikzlibrary{backgrounds,calc,trees,shapes}

\newtheorem{thm}{Theorem}[section]
\newtheorem{lem}[thm]{Lemma}
\newtheorem{prop}[thm]{Proposition}
\newtheorem{cor}[thm]{Corollary}

\theoremstyle{definition}
\newtheorem{dfn}[thm]{Definition}
\newtheorem{pro}{Problem}
\newtheorem{alg}[thm]{Algorithm}
\newtheorem{obs}[thm]{Observation}

\newcommand{\tr}{\operatorname{tr}}
\newcommand{\id}{\mathbbm{1}}
\newcommand{\supp}{\operatorname{supp}}
\newcommand{\kernel}{\operatorname{ker}}
\newcommand{\ran}{\operatorname{ran}}
\newcommand{\diag}{\operatorname{diag}}
\newcommand{\rank}{\operatorname{rank}}
\newcommand{\detm}{\operatorname{det}}
\newcommand{\R}{\mathbb{R}}
\newcommand{\C}{\mathbb{C}}
\newcommand{\Q}{\mathbb{Q}}
\newcommand{\N}{\mathbb{N}}

\usepackage{abstract}

\usepackage{hyperref}
\definecolor{linkblue}{rgb}{0.1,0.2,.7}
\definecolor{citegreen}{rgb}{0.1,0.7,.2}
\hypersetup{
    colorlinks,
    citecolor=citegreen,
    filecolor=linkblue,
    linkcolor=linkblue,
    urlcolor=linkblue
}
\allowdisplaybreaks

\title{\vspace{-15mm}%
	A review of matrix scaling and Sinkhorn's normal form for matrices and positive maps}
\author{%
	\large
	\textsc{Martin Idel\footnote{martin.idel@tum.de}} \\[2mm]
	\normalsize	Zentrum Mathematik, M5, Technische Universit\"{a}t M\"{u}nchen, 85748 Garching \\
	\vspace{-5mm}
	}
\date{}

\begin{document}

\maketitle

\begin{abstract}
Given a nonnegative matrix $A$, can you find diagonal matrices $D_1,~D_2$ such that $D_1AD_2$ is doubly stochastic? The answer to this question is known as Sinkhorn's theorem. It has been proved with a wide variety of methods, each presenting a variety of possible generalisations. Recently, generalisations such as to positive maps between matrix algebras have become more and more interesting for applications. This text gives a review of over 70 years of matrix scaling. The focus lies on the mathematical landscape surrounding the problem and its solution as well as the generalisation to positive maps and contains hardly any nontrivial unpublished results. 
\end{abstract}

\tableofcontents

\section{Introduction}
It is very common for important and accessible results in mathematics to be discovered several times. Different communities adhere to different notations and rarely read papers in other communities also because the reward does not justify the effort. In addition, even within the same community, people might not be aware of important results - either because they are published in obscure journals, they are poorly presented in written or oral form or simply because the mathematician did not notice them in the surrounding sea of information. This is a problem not unique to mathematics but instead inherent in all disciplines with epistemological goals. 

The scaling of matrices is such a problem that has constantly attracted attention in various fields of pure and applied mathematics\footnote{The term ``Babylonian confusion'' to describe the history of this problem was first used in \cite{kru79}}. Recently, generalisations have been studied also in physics to explore possibilities in quantum mechanics where it turns out that a good knowledge of the vast literature on the problem can help a lot in formulating approaches. This review tries to tell the mathematical story of matrix scaling, including algorithms and pointers to applications.

As a motivation, consider the following problem: Imagine you take a poll, where you ask a subset of the population of your country what version (if any) of a certain product they buy. You distinguish several groups in the population (for instance by age, gender, etc.) and you distinguish several types of product (for instance different brands of toothbrushs). From the sales statistics, you know the number of each product sold in the country and from the country statistics you know the number of people in different groups. Given the answers of a random sample of the population, how can you extrapolate results?

Central to a solution is the following innocuous theorem:
\begin{thm}[Sinkhorn's theorem, weak form \cite{sin64}] \label{thm:introweak}
Given a matrix $A$ with positive entries, one can find matrices $D_1,D_2$ such that $D_1AD_2$ is doubly stochastic.
\end{thm}
The literature on Sinkhorn's theorem and its generalisations is vast. As we will see, there are some natural ways to attack this problem, which further explains why the different communities were often not aware of the efforts of their peers in other fields.

One of the main motivations for this review was a generalisation of Sinkhorn's theorem to the noncommutative setting of positive maps on matrix algebras:
\begin{thm}[Weak form of \cite{gur03}'s generalisation to positive maps] \label{thm:introgur}
Given a map $\mathcal{E}:\C^{n\times n}\to \C^{n\times n}$ which maps positive semidefinite matrices to positive definite matrices one can find invertible matrices $X,Y$ such that the map $\mathcal{E}^{\prime}(\cdot):=Y\mathcal{E}(X\cdot X^{\dagger})Y^{\dagger}$ is doubly stochastic, i.e.
\begin{align*}
	\mathcal{E}^{\prime}(\id)=\id, \qquad \mathcal{E}^{\prime\,*}(\id)=\id
\end{align*}
with the adjoint matrices $X^{\dagger}$ and the adjoint map $\mathcal{E}^*$. 
\end{thm}
Some results and approaches can be translated to this noncommutative setting, but many questions remain open and the noncommutativity of the problem makes progress difficult. 

Very recently, a new generalisation of Sinkhorn's theorem to a noncommutative setting has appeared in \cite{ben16}.

The goal of this review is therefore threefold:
\begin{itemize}
	\item Trace the historical developments of the problem and give credit to the many people who contributed to the problem.
	\item Illuminate the many approaches and connections between the approaches to prove Sinkhorn's theorem and its generalisations.
	\item Sketch the generalisation to positive maps and its history and highlight the questions that are yet unanswered and might be attacked using the knowledge from the classical version.
\end{itemize}
In addition, I will try to give a sketch of the algorithmic developments and pointers to the literature for applications. I will probably have forgotten and/or misrepresented contributions; comments to improve the review are therefore very welcome.

\section{Notation and Preliminaries}
Most of the concepts and notations discussed in this short section are well-known and can be found in many books. I encourage the reader to refer to this section only if some notation seems unclear.

We will mostly consider matrices $A\in \R^{n\times m}$. Such matrices are called \emph{nonnegative} (\emph{positive}) if they have only nonnegative (positive) entries. We denote by $\R^n_+$ ($\R^n_{+0}$) all vectors with only positive entries (nonnegative entries) and for any such $x\in \R^n_+$, $\diag(x)$ defines the diagonal matrix with $x$ on its main diagonal, while $1/x\in \R^n_+$ defines the vector with entries $1/x_i$ for all $i$.

An important concept for nonnegative matrices is the pattern. The \emph{support} or \emph{pattern} of a matrix $A$ is the set of entries where $A_{ij}>0$. A subpattern of the pattern of $A$ is then a pattern with fewer entries than the pattern of $A$. We write $B\prec A$ if $B$ is a subpattern of $A$, i.e.~for every $B_{ij}>0$ we have $A_{ij}>0$. 

Finally, let us introduce irreducibility and decomposability. Details and connections to other notions for nonnegative matrices are explained in Appendix \ref{app:prelimmatrices}. If $A$ is nonnegative, then $A$ is \emph{fully indecomposable} if and only if there do not exist permutations $P,Q$ such that 
\begin{align}
	PAQ=\begin{pmatrix}{}A_1 & 0 \\ A_3 & A_2\end{pmatrix} \label{eqn:fullyindecdef}
\end{align}
where neither $A_1$ nor $A_2$ contain a zero row or column and $A_3\neq 0$. The matrix is \emph{irreducible}, if no permutation $P$ can be found such that already $PAP^T$ is of form (\ref{eqn:fullyindecdef}). In particular, this implies that all fully indecomposable matrices are at least irreducible. 

For positive vectors, we will not use the notation $x>0$ to avoid confusion with the positive definite case: Especially in the second part of this review, we will be dealing mostly with \emph{positive (semi)definite} matrices $A\in \R^{n\times n}$, which are symmetric matrices with only positive (nonnegative) eigenvalues and should not be confused with positive matrices. We also introduce the partial order $\geq$ for positive semidefinite matrices, where $A\geq B$ if and only if $A-B$ is positive semidefinite and $A>B$ if $A-B$ is positive definite.

When talking about positive maps, we will also adopt the notation that $\mathcal{M}_{n,m}$ denotes the complex $n\times m$ matrices, while the shorter $\mathcal{M}_n$ is used for complex $n\times n$ square matrices. 

\section{Different approaches to equivalence scaling} \label{sec:approaches}
This section explores the historical development and current form of the mathematical landscape surrounding the following extension to Theorem \ref{thm:introweak}:
\begin{thm} \label{thm:sink}
Let $A\in \R^{m\times n}$ be a matrix with nonnegative entries. Then for any vectors $r\in \R^m$ and $c\in \R^n$ with nonnegative numbers there exist diagonal matrices $D_1$ and $D_2$ such that 
\begin{align*}
	D_1AD_2e&=r \\
	D_2A^TD_1e&=c
\end{align*}
if and only if there exists a matrix $B$ with $Be=r$ and $B^Te=c$ and the same pattern as $A$. Here, $e=(1,\ldots,1)^T$ which means that $r$ contains the row sums of the scaled matrix and $c$ contains the column sums.

Furthermore, if the matrix has only positive entries, $D_1$ and $D_2$ are unique up to a constant factor.
\end{thm}
In Section \ref{sec:equivscaling}, we give maximal formulations of this theorem. Some immediate questions emerge, such as: How to compute $D_1,D_2$ and the scaled matrix? Can this be generalised to arrays of higher dimension? All of these questions and many more have been answered in the literature.

Given a matrix $A\in \R^{n\times n}$ with positive entries and the task to scale $A$ to given row sums $r$ and column sums $c$, one is very naturally lead to the following approximation algorithm:
\begin{alg}[RAS method] \label{alg:ras} Given $A\in\R^{m\times n}$, do:
\begin{itemize}
	\item Multiply each row $j$ of $A$ with $r_j/\left(\sum_{i} A_{ij}\right)$ to obtain $A^{(1)}$ with row sums $r$.
	\item Multiply each column $j$ of $A^{(1)}$ with $c_j/\left(\sum_i A_{ji}\right)$ to obtain $A^{(2)}$ with column sums $c$.
	\item If the row sums of $A^{(2)}$ are very far from $r$, repeat steps one and two.
\end{itemize}
\end{alg}
If the algorithm converges, the limit $B$ will be the scaled matrix. However, there is a priori no guarantee that $D_1,D_2$ exist, in which case we can only ask for \emph{approximate scaling}, i.e.~matrices $D_1,D_2$ such that $D_1AD_2\approx B$. 

\subsection{Historical remarks}
The iterative algorithm \ref{alg:ras} is extremely natural and it is therefore not surprising that it was rediscovered several times. It is at least known as \emph{Kruithof's projection method} (\cite{kru79}) or \emph{Kruithof double-factor model} (especially in the transportation community; \cite{vis80}), the \emph{Furness (iteration) procedure} (\cite{rob74}), \emph{iterative proportional fitting procedure (IPFP)} (\cite{rus95}), the \emph{Sinkhorn-Knopp algorithm} (\cite{kni08}), the \emph{biproportional fitting procedure} (in the case of $r=c=e$; \cite{bac70}) or the \emph{RAS method} (especially in economics and accounting; \cite{fof02}). Sometimes, it is also referred to simply as \emph{matrix scaling} (\cite{rot07}), which is mostly used as the term for scalings of the form $DAD^{-1}$, or \emph{matrix balancing}, which is mostly used for scalings to equal row and column sums. The algorithm is a special case of a number of other algorithms such as Bregman's balancing method (cf. \cite{lam81}) as we will see later on. 

When was interest sparked in the RAS method and diagonal equivalence? The earliest claimed appearance of the model dates back to at least the 30s and Kruithof's use of the method in telephone forecasting (\cite{kru37}). At a similar time, according to \cite{bre67a}, the Soviet architect Sheleikhovskii considered the method. \cite{sin64} claims that when he started to evaluate the method, it had already been proposed and in use. His example is the unpublished report \cite{welsh}. \cite{bac70} acknowledges \cite{dem40} in transportation science, who popularised the RAS method in the English speaking communities. 

None of these approaches seem to have been thoroughly justified. Bacharach notes that Deming and Stephan only propose an ad-hoc justification for using their method to study their problem, which turned out to be wrong (cf. \cite{ste42}). He further claims that the first well-founded approach to use the RAS - this time in economics - was given by Richard Stone, who also coined the name ``RAS model'' (Bacharach cites \cite{sto62}, although the name RAS must have occurred earlier as it already occurs in \cite{thi61} without attribution and explanation). However, one can argue that the first justified approach occurred earlier: \cite{sch31} had already posed a question regarding models of Brownian motion when given a priori estimates, which led to a similar problem. His approach was justified, albeit the ultimate justification in terms of large deviation theory needed to wait for the development of modern probability theory (cf. \cite{geo15}). The problem boils down to solving a continuous analogue of Sinkhorn's theorem, which leads to the matrix problem using discrete distributions (essentially similar to \cite{hob65}) and was first attacked in \cite{for41} using a fixed point approach similar to Algorithm \ref{alg:perfro}.

However, none of the original papers provided a convergence proof with the possible exception of \cite{for41}\footnote{The notation and writing is very difficult to read today, so I am not entirely sure whether the proof is correct and captures the case we are interested in.}. As noted by \cite{fie70}, after Deming and Stephan provided their account, their community started to develop the ideas, but a proof was still lacking (\cite{smi47,elb55,fri61}).

Summarising the last paragraphs, the RAS method was discovered independently for different reasons in the 30s to 40s, although none of the authors provided a proof (with the possible exception of Fortet). A more theoretical analysis developed in the 60s after Stone's results in economics (e.g.~\cite{sto62}) and \cite{sin64} in statistics and algebra. Since then, a large number of papers has been published analysing or applying Theorem \ref{thm:sink}. Every decade since the sixties contains papers where proving the theorem or an extension thereof is among the main results (examples are \cite{sin64,mac77,pre80,bor98,puk09,geo15}). 

Many authors are aware of at least some other attempts, but only a few try to give an overview.\footnote{This suggests once again that the problem had a very complicated history which also makes it difficult to find out whether a problem has already been solved in the past. Several authors have attempted more complete historical overviews such as \cite{fie70, mac77, sch90b, bro93, kal96a, kal08, puk09}. In \cite{rot89a}, the authors claims that a colleague collected more than 400 papers on the topic of matrix scaling.} The situation is further complicated by a the fact that the technical answer to the question of scalability is tightly linked with the question of patterns, which has a rich history in itself, probably starting with Fr{\'e}chet (overview of a long line of work in \cite{fre60}). 

The last point is particularly interesting: In fact, one could summarise matrix scaling matrix scaling as follows: Given a nonnegative matrix $A$ it is scalable to a matrix $B$ fulfiling some constraints (mostly linear but some nonlinear constraints are allowed), a matrix is scalable with diagonal matrices (in different ways, mostly $D_1AD_2$ where $D_1$ and $D_2$ need not be independent) if and only if there exists a matrix $C$ with the same pattern as $A$ fulfiling the constraints.

Today, proofs and generalisations of Theorem \ref{thm:sink} and similar questions about scaling matrices in the form $DAD$ or $DAD^{-1}$ form a knot of largely interconnected techniques. We will now try to give an overview of these results and highlight their connections. A graphical overview is presented in Figure \ref{fig:overview}.

\begin{figure}[htbp] \label{fig:overview}
\resizebox{.95\textwidth}{!}{%
\begin{tikzpicture}[font=\sffamily]
  \node[text width=3.3cm,align=center] at (6.5,13) (A) {log-barrier function \\ {\color{red} coordinate descent}};
  \node[text width=3cm] at (3.5,11) (B) {logarithmic single-variable potential};
  \node[text width=3cm] at (3.5,9) (C) {single-variable potential};
  \node[text width=3cm] at (10.5,12) (D) {homogeneous \\ potential \\ {\color{red} total gradient}};
	
	\node[text width=2.6cm,align=center] at (0,4) (E) {relative entropy \\ {\color{red} iterative I-projection}};
	\node at (-1.3,4) (Einv) {};
	\node[text width=3cm,align=center] at (2,6)(L) {log-linear models};
	
	\node[text width=3cm,align=center] at (10.5,6) (F) {Nonlinear Perron-Frobenius Theory \\ {\color{red} f.p. iteration}};
	\node at (9,5.5) (Finv) {};
	\node at (11.5,4) (G) {combined};
	\node[text width=2cm] at (10.5,2.3) (H) {Kakutani fixed point theorem};
	\node[text width=2.3cm] at (12.5,2) (I) {Kronecker \\ index theorem};
	
	\node[text width=3cm,align=center] at (4,0) (J) { {\color{red} convex opt.} \\convex log-barrier \\ {\color{red} coordinate ascent}};
	\node at (2.5,-0.4) (Jinv) {};
	\node[text width=3cm] at (7,1.5) (K) {single-variable convex potential};

  \begin{pgfonlayer}{background}
    \foreach \nodename in {A,B,C,D,E,F,G,H,I,J,K,L} {
      \coordinate (\nodename') at (\nodename);
    }
		\node[color=blue!30] at (10,14.5) (nlino) {Potential optimisation};
		\node[color=red!30,text width=2cm] at (0,0) (conv) {Convex \\ optimisation};
		\node[color=green!30,text width=3cm] at (11.5,8.5) (top) {Fixed point \\ approaches};
		\node[color=orange!40] at (-0.7,6.6) (entr) {Entropy};
    \path[fill=blue!15,draw=blue!15,line width=3.4cm, line cap=round, line join=round] 
    (A') to (B') 
         to (C') 
         to (D') 
         to (A') -- cycle;
		\path[fill=red!15,draw=red!15,line width=3.4cm, line cap=round, line join=round]
		(E') to (J')
		     to (K')
				 to (E') -- cycle;
		\path[fill=green!15,draw=green!15,line width=3.4cm, line cap=round, line join=round]
		(F') to (G')
		     to (I')
				 to (H')
				 to (F') -- cycle;
		\path[fill=yellow!15,draw=yellow!15,line width=3cm, line cap=round, line join=round]
		(E') to (L')
				 to (E') -- cycle;
  \end{pgfonlayer}
	\path[<->] (A) edge node[right]{\ref{obs:logperfro}} (F);
	\path[<->,dotted] (A) edge (B);
	\path[<->] (B) edge node[right]{\ref{obs:djok}} (C);
	\path[<->] (C) edge node[above]{\ref{lem:london}} (D);
	\path[<->] (D) edge node[above]{\ref{obs:logtohom}} (A);
	\path[<->] (A) edge node[right]{\ref{lem:convprog}} (J);	
	\path[->] (H) edge node[right]{} (G);
	\path[->] (I) edge node[right]{} (G);
	\path[<->,dotted] (H) edge node[above, near start]{\ref{obs:topentrop}} (E);
	\path[->] (E) edge node[right]{\ref{obs:wolfedual}} (J);
	\path[<->] (J) edge node[below, near end]{\hspace{.2cm} \ref{lem:london}} (K);
	\path[->] (C) edge node[right, very near start]{\ref{obs:topentrop}} (K);
	\path[->] (B) edge[bend right] node[right]{\ref{obs:logentrop}} (E);
	\path[<->] (E) edge node[right]{\hspace{.1cm} \ref{lem:loglinentrop}} (L);
	\path[<-] (L) edge node[above]{applies to} (G);
	
	\path[<->] (A) edge[bend right=30,red] node[right]{\ref{obs:rasperfro}} (Finv); 
	\path[<->] (A) edge[bend right=50,red] node[right]{\ref{obs:rasentrop}} (Einv);
	\path[<->] (E) edge[bend right=10,red] node[right]{\ref{obs:ascent}} (Jinv);
	\path[->,dotted] (J) edge[red] node[right]{\color{red} similar} (D);
\end{tikzpicture}
}
\caption{Connected approaches to prove Theorem \ref{thm:sink} and their relationships. Red arrows and text denote natural algorithms and their connections.}
\end{figure}

\subsection{The logarithmic barrier function} \label{sec:logbarrier}
Potentials and barrier functions have been important in the study of matrix scaling since at least the unpublished results of \cite{gor63}. Here, we largely follow \cite{kal96a}, who give a very lucid account about the interconnections between different barrier function formulations for $g$.

Let $A\in \R^{n\times n}$ be a matrix with nonnegative entries and $r,c\in\R^n_+$. Define the \emph{logarithmic barrier function}
\begin{align}
	g(x,y)=y^{T}Ax-\sum_{i=1}^n c_i\ln x_i - \sum_{i=1}^n r_i\ln y_i \label{eqn:logbarrier}
\end{align}
If we take partial derivatives, we obtain
\begin{align}
	\begin{split}
		\partial_{y_i} g(x,y)=Ax- r_i/y_i \\
		\partial_{x_i} g(x,y)=y^TA- c_i/x_i \label{eqn:partialderiv}
	\end{split}
\end{align}
which implies that for any stationary point we have
\begin{align*}
	\sum_{j} A_{ij}x_jy_i=r_i \qquad \sum_{j} A_{ji}x_iy_j=c_i
\end{align*}
and setting $D_1=\diag(y)$ and $D_2=\diag(x)$ solves the scaling problem. Conversely, any scaling gives a stationary point of the logarithmic barrier function. In summary:
\begin{lem} \label{lem:logbar}
Given $A\in \R^{n\times n}$ nonnegative and two vectors $r,c\in \R^n_+$, then the matrix can be diagonally scaled to a matrix $B$ with row sums $r$ and column sums $c$ if and only if the corresponding logarithmic barrier function (\ref{eqn:logbarrier}) has a stationary point.
\end{lem}
According to \cite{mac77}, this observation was first made by \cite{gor63} who also gave the first complete and correct proof. However, the paper only circulated privately. Gorman apparently did not consider this scaling function directly but used an approach similar or identical to the ones considered in convex geometry described in Section \ref{sec:convexprog}. 

The potential barrier function can also be seen from the perspective of Lagrangian multipliers:
\begin{lem}[\cite{mar68}] \label{lem:marolk}
Given $A\in \R^{n\times n}$ nonnegative and two vectors $r,c\in \R^n_+$, then the matrix can be diagonally scaled to a matrix $B$ with row sums $r$ and column sums $c$ if and only if on the region
\begin{align}
 \Omega:=\left\{(x,y)\middle|\prod_{i=1}^m x_i^{c_i}=\prod_{i=1}^m y_i^{r_i}=1, x_i>0,y_i>0\right\}
\end{align}
the function $y^TAx$ is bounded away from zero and is unbounded whenever $\|x\|_{\infty}+\|y\|_{\infty}\to \infty$. The function $y^TAx$ then attains a minimum defining $D_1$ and $D_2$.
\end{lem}
This was used to prove our Theorem \ref{thm:sink} in \cite{mar68}. We observe:
\begin{obs}
Lemma \ref{lem:marolk} and \ref{lem:logbar} are equivalent: The logarithmic barrier function is the Lagrange function of the optimisation problem in Lemma \ref{lem:marolk}.
\end{obs}

Now consider $g(x,y)$ for a fixed $x$. Since $(-\ln)$ is a convex function and $x^TAy$ is linear in $y$, $g$ is convex in $y$. The same holds for a fixed $y$, i.e.~$g$ is convex in both direction. It is then natural to consider the \emph{coordinate descent algorithm} (for an introduction and overview see \cite{wri15}):
\begin{alg} \label{alg:coorddesc}
Given a nonnegative matrix $A$, take a starting point for $g$, e.g.~$x_0=y_0=e$ and iterate:
\begin{enumerate}
	\item For fixed $y_n$, find $x_{n+1}$ by searching for the minimum of $g(x,y_n)$.
	\item For fixed $x_{n+1}$, find $y_{n+1}$ by searching for the minimum of $g(x_{n+1},y)$. 
	\item Repeat until convergence.
\end{enumerate}
\end{alg}
It is possible to solve $\min_x g(x,y)$ or $\min_y g(x,y)$ analytically:
\begin{align*}
	x_{n+1}=p/(Ay_n), \qquad y_{n+1}=q/(Ax_{n+1}).
\end{align*}
This leads to the following observation:
\begin{obs}[\cite{kal96a}]
Algorithm \ref{alg:coorddesc} and \ref{alg:ras} are the same.
\end{obs}
\begin{proof}
Define $D^{(1)}_n:=\diag(y_n)$ and $D^{(2)}_n:=\diag(x_n)$. Then we have $D^{(1)}_{n+1}AD^{(2)}_ne=r$ and $e^TD^{(1)}_nAD^{(2)}_n=c^T$, which implies that we perform successive row- and column normalisations as in the RAS method.
\end{proof}
Using the fact that the algorithm is a coordinate descend method, one can obtain a convergence proof including a discussion of convergence speed of this algorithm and a dual algorithm (\cite{luo92}). See also Observation \ref{obs:rasentrop} for a discussion of coordinate ascent methods.

However, $g$ is not jointly convex. For a purely (jointly) convex reformulation, consider the minimum for $t$ along any line $g(tx,ty)$, where $g$ is convex. If we define
\begin{align}
	k(x,y):=\min_{t>0} g(tx,ty)
\end{align}
minimising $k(x,y)$ is still equivalent to minimising $g(x,y)$. The corresponding $k$ will be homogeneous and the domain for minimisation will in fact be compact. 
\begin{obs}[\cite{kal96a}] \label{obs:logtohom}
We obtain:
\begin{align}
	k(x,y)&= \min_{t>0} \left( t^2 y^TAx -2n\ln t-\sum_{i=1}^n c_i\ln x_i-\sum_{i=1}^n r_i\ln y_i\right) \\
		&= \ln \left(\frac{(y^TAx)^n}{\prod_{i=1}^n x_i^{c_i}\prod_{j=1}^n y_j^{r_j}}\right)+n-n\ln(n)
\end{align}
hence minimising $g$ is equivalent to minimising $k$.
\end{obs}
This proves the following lemma:
\begin{lem}
Given a nonnegative matrix $A\in \R^{n\times m}$, it can be scaled to a matrix with row sums $r$ and column sums $c$ if and only if the minimum of $k(x,y)$ exists and is positive. The corresponding minima $(x,y)$ define the diagonal matrices to achieve the scaling.
\end{lem}
The function $k$ is also similar to Karmakar's potential function for linear programming and Algorithm \ref{alg:ras} is the coordinate descent method for this function (\cite{kal96a,kal96b}).

Setting $y(x)=(Ax)^{-1}$, we arrive at another formulation of the problem. In the doubly stochastic case, this formulation is due to \cite{djo70,lon71} and was later adapted to arbitrary column and row sums in \cite{sin74}\footnote{later studied in \cite{kru79}, who used an entropic approach for the generalised problem and in \cite{ber79}, who used a direct convergence approach reminiscent of Sinkhorn and others.}:
\begin{lem} \label{lem:london}
Let $A\in\mathbb{R}^{n\times n}$ be a nonnegative matrix. There exists a scaling to a matrix with row sums $r$ and column sums $c$ iff the infimum 
\begin{align}
	\inf\left\{\prod_{i=1}^n \left(\sum_{j=1}^n A_{ij}x_j\right)^{r_i} \middle|\prod_{i=1}^n x_i^{c_i}=1\right\} \label{eqn:london}
\end{align}
is attained on $x,y\in\R_+$.
\end{lem}
\begin{obs} \label{obs:djok}
Note that the infimum is attained iff the infimum 
\begin{align*}
	\inf\left\{\sum_{i=1}^n r_i\ln \left( \sum_{j=1}^n A_{ij}x_j\right) \middle|\sum_{i=1}^n c_i\ln x_i=0\right\}
\end{align*}
is attained. This is the formulation in Lemma \ref{lem:marolk}.
\end{obs}
Finally, let us sketch a proof using potential methods. 
\begin{proof}[Sketch of proof of Theorem \ref{thm:sink} (Potential version)]
We sketch a proof for arbitrary row and column sums based on the short proof of \cite{djo70} for doubly stochastic scaling: First assume that $A\in\R^{m\times n}$ is a positive matrix. Starting with equation (\ref{eqn:london}) we define the function
\begin{align*}
	f(x_1,\ldots,x_n):=\frac{\prod_{i=1}^m \left(\sum_{j=1}^n A_{ij}x_j\right)^{r_i}}{\prod_{i=1}^n x_i^{c_i}}
\end{align*}
on the set of $x_i$ with $x_i>0$ and $\sum_i x_i=1$. Consider an arbitrary point $b$ on the boundary (i.e.~$b_i=0$ for at least one $i\in 1,\ldots n$). For $x_i\to b_i$, since $\prod_i x_i=0$ and $\sum_j A_{ij}x_j \neq 0$ always, we have that $f(x_1,\ldots,x_n)\to \infty$. Hence the function takes its minimum in the interior. At the minimum, the partial derivatives must vanish and we obtain:
\begin{align*}
	0\stackrel{!}{=}\partial_{x_l}f &=\prod_{i=1}^m \left(\frac{\left(\sum_{j=1}^n A_{ij}x_j\right)^{r_i}}{x_i^{c_i}}\right)
		\left(\sum_{k=1}^m \left(\frac{x_k^{c_k}}{\left(\sum_{p=1}^n A_{kj}x_j\right)^{r_k}}\right)\cdot \right.\\
		&~~~~\left.\left(\frac{r_k\left(\sum_{p=1}^n A_{kj}x_j\right)^{r_k-1}A_{kl}}{x_k^{c_k}}-\frac{c_l\left(\sum_{p=1}^n A_{kj}x_j\right)^{r_k}}{x_l^{c_l+1}}\delta_{kl}\right)\right) \\
		&= \prod_{i=1}^m \left(\frac{\left(\sum_{j=1}^n A_{ij}x_j\right)^{r_i}}{x_i^{c_i}}\right)\left(\sum_{k=1}^m A_{kl}r_k \left(\sum_{p=1}^n A_{kj}x_j\right)^{-1}-\sum_{k=1}^m \frac{c_l x_k^{c_k}}{x_l^{c_l+1}}\delta_{kl} \right).
\end{align*}
If we take all conditions for $l=1,\ldots n$, then this is equivalent to the condition
\begin{align*}
	A^T(r/(Ax))=c/x
\end{align*}
which boils down to equations (\ref{eqn:partialderiv}).

The more technical part for nonnegative matrices is a more careful analysis of what happens for nonnegative matrices that are not positive. For doubly stochastic matrices, we can use the fact that fully indecomposable matrices have a positive diagonal, which implies once again that $\prod_i\sum_j A_{ij}b_j\neq 0$. A similar argument can be made for arbitrary patterns, but we leave it out in this sketch.
\end{proof}

\subsection{Nonlinear Perron-Frobenius theory} \label{sec:perfro}
Another early approach uses nonlinear Perron-Frobenius theory which is essentially a very general approach to tackle fixed point problems for (sub)homogeneous maps on cones. A short overview is given in appendix \ref{app:perfro}. The basic idea is given by:
\begin{lem}[\cite{bru66}] \label{lem:nonlinearperron}
Given a nonnegative matrix $A\in \R^{n\times n}$, there exists a scaling of $A$ to a matrix with row sums $r$ and column sums $c$ if and only if the following map has a fixed point $x>0$:
\begin{align}
	\begin{split}
		\mathbf{T}:\R^n\to\R^n \\
		\mathbf{T}(x)=c/(A^T(r/(Ax)))
	\end{split}
\end{align}
\end{lem}
This also suggests another simple algorithm:
\begin{alg} \label{alg:perfro}
Given a nonnegative matrix $A$. Let $x_0=e$. Iterate until convergence: 
\begin{align}
	x_{n+1}=\mathbf{T}(x_n).
\end{align}
\end{alg}
The development of this idea that started with \cite{men67} and was used to provide a full proof of Theorem \ref{thm:sink} in \cite{bru66} for doubly stochastic matrices. \cite{men69} consider arbitrary row- and column sums and give a complete study of the spectrum of the Menon-operator. Some contraction properties were used to give a direct proof of convergence of the RAS algorithm in \cite{ber79}. The connection to Hilbert's projective metric, and therefore to ``Nonlinear Perron-Frobenius theory'' (cf. \cite{lem12}), became clear later on and allowed to give upper bounds on the covergence speed of the RAS (\cite{fra89,geo15}).

However, Menon was not the first to define the operator $\mathbf{T}$: Looking closely at the arguments given in \cite{for41}, one can see the continuous version of $\mathbf{T}$, which lead to an independent rediscovery of $\mathbf{T}$ and its connection to the Hilbert metric in \cite{geo15}. Probably, Menon was not even the first to define the discrete version of the operator and to note that the existence of a fixed point can be seen by invoking Brouwer's fixed point theorem. This dates back to \cite{thi63,thi64}, building on work about matrix patterns (\cite{thi61}). According to \cite{cau65}, \cite{thi64} was also the first paper to conjecture the necessary and sufficient conditions for scalability\footnote{He also notes that the early history around \cite{dem40} is a little bit curious, since the authors claim to have a convergence proof but never publish it.}. The ideas where rediscovered another time in \cite{bal89b}, where the authors used the fixed point argument to prove that a scaling exists and fulfils their axiomatic approach.

Let us connect the approach to Section \ref{sec:logbarrier}. First note that the algorithm is nothing else but a slight variation of the RAS method:
\begin{obs} \label{obs:rasperfro}
Setting $y_{n+1}:=r/(Ax_{n})$ and $x_{n+1}:=c/(A^Ty_{n+1})$ we can immediately see that one iteration of Algorithm \ref{alg:perfro} is one complete iteration of the RAS method \ref{alg:ras}.
\end{obs}
The connection with the logarithmic barrier method is also very close:
\begin{obs} \label{obs:logperfro}
Any fixed point of the Menon operator defines a stationary point of the logarithmic barrier function (\ref{eqn:logbarrier}) and vice versa.
\end{obs}
\begin{proof}
Let $A$ be a nonnegative matrix. The derivative conditions for the stationary points of (\ref{eqn:logbarrier}) are given in equation (\ref{eqn:partialderiv}), which are equivalent to:
\begin{align}
	Ax=r/y \quad y^TA=c/x \label{eqn:menon}
\end{align}
This implies immediately that $x=p/(A^T(q/Ax))$, hence $x$ is a fixed point of $\mathbf{T}$. Similarly, any positive fixed point immediately gives a scaling as a minimum of the logarithmic barrier function.
\end{proof}
This also proves Lemma \ref{lem:nonlinearperron}.

\begin{proof}[Sketch of proof of Theorem \ref{thm:sink} (Nonlinear Perron-Frobenius theory version)]
Let us first assume $A$ has only positive entries. Then $\mathbf{T}$ sends all vectors $x\in \R^n_{+\,0}$ to $\R^n_+$ hence using Brouwer's fixed point theorem $\mathbf{T}$ has a positive fixed point. Note that in order to apply Brouwer's fixed point theorem, we need to have a compact set. To achieve this, consider the operator $\tilde{\mathbf{T}}(x)=\mathbf{T}(x)/\sum_{i=1}^n \mathbf{T}(x)_i$. 

For general nonnegative matrices $A$, one can extend $\mathbf{T}$ to be a map from $x\in \R^n_{+\,0}$ into itself (see also Appendix \ref{app:perfro}) either by a general argument (see Theorem \ref{thm:context}) or by defining $\infty\cdot 0=0$ and $\infty\cdot c=\infty$ for all positive $c$. One can easily see that $\mathbf{T}$ will not send any entry to $\infty$. 

However, it is not immediately clear when the fixed point is positive if $A$ contains zero-entries. This is the main technical difficulty for a complete proof. \cite{bru66} show that if $A$ is fully indecomposable, $\mathbf{T}(x)$ has at least $k+1$ entries which are nonzero if $x$ has exactly $k$ entries which are nonzero, which immediately proves that the fixed point must be positive. 

Upon closer observation, the map is contractive under Hilbert's metric and Banach's fixed point theorem immediately provides existence and uniqueness of the scaled matrix. The fixed point itself provides the diagonal of $D_2$.  
\end{proof}

\subsection{Entropy optimisation} \label{sec:entrop}
Another approach, which underlies many justifications for applications, considers entropy minimisations under linear constraints. An overview of entropy minimisation and its relation to diagonal equivalence can be found in \cite{bro93}, a broader overview about the relation of the RAS algorithm to entropy scaling with a focus on economic settings can be found in \cite{mcd99}. 

To formulate the problem, we define the \emph{Kullback-Leibler divergence}, \emph{I-divergence} or \emph{relative entropy}, which was first described in \cite{kul51} (see also \cite{kul59}) for two vectors $x,y\in \R^{n}_{+\,0}$:
\begin{align}
	D(x\|y):=\sum_{j=1}^n x_j \ln\left(\frac{x_j}{y_j}\right)
\end{align}
where we use the convention that the summand is zero if $x_j=y_j=0$ and infinity if $x_j>0,y_j=0$. The relative entropy is nonnegative and zero if and only if $x=y$ and it is therefore similar to a distance measure. Given a set, what is the smallest ``distance'' of a point to this set in relative entropy? This is known as \emph{I-projection} (cf. \cite{csi75}). 

Let $A$ be a nonnegative matrix and define
\begin{align*}
	\Pi_1&:=\{B|Be=r\} \\
	\Pi_2&:=\{B|e^TB=c^T\}. 
\end{align*}
We ask for the I-projection of $A$ onto the set $\Pi_1\cap \Pi_2$, i.e.~we want to find $A^*$ such that
\begin{align}
	D(A^*\|A)=\inf_{B\in \Pi_1\cap\Pi_2} D(B\|A). \label{eqn:Iproj}
\end{align}
The connection to scaling was probably first used in \cite{bro59}, where the RAS method is used to improve an estimate for positive probability distributions of dimensions $2\times 2\times \ldots \times 2$ in the relative entropy measure (Brown cites \cite{lew59} as a justification for his approach, where the relative entropy is justified as a ``closeness'' measure). According to \cite{fie70}, this approach was later generalised to all multidimensional tables in \cite{bis67} based on some duality of optimisation by \cite{goo65}\footnote{Both references were not available to me.}. Another early use of relative entropy occurs in \cite{uri66} (see also \cite{the67}), where it was noted without proof that the results were the same as the RAS.

A very natural approach to obtain $A^*$ would be to try an iterative I-projection:
\begin{alg} \label{alg:ipro}
Let $A$ be nonnegative. 
\begin{itemize}
	\item Let $A^{(0)}=A$. 
	\item If $n$ is even, find $A^{(n+1)}$ such that
		\begin{align*}
			D(A^{(n+1)}\|A^{(n)}):=\inf_{B\in \Pi_1} D(B\|A).
		\end{align*}
	\item If $n$ is odd, find $A^{(n+1)}$ such that
		\begin{align*}
			D(A^{(n+1)}\|A^{(n)}):=\inf_{B\in \Pi_2} D(B\|A)
		\end{align*}
	\item Repeat the steps until convergence.
\end{itemize}
\end{alg}
\begin{obs}[cf. \cite{csi75,csi89}] \label{obs:rasentrop}
The algorithms \ref{alg:ipro} and \ref{alg:ras} are the same.
\end{obs}
\begin{proof}
This was first shown in \cite{ire68}. We give a short argument based on Lagrangian multipliers restricted to column normalisation. Given $A\in \R^{n\times n}$, the Lagrangian for the problem is
\begin{align*}
	L(B,\lambda):=D(B\|A)+\lambda_j \left(\sum_i A_{ij}-q_j\right).
\end{align*}
Partial derivatives $\partial_{B_{ij}}L=0$ and $\partial_{\lambda_j}L=0$ lead to the system of equations:
\begin{align*}
	\ln\left( \frac{B_{ij}}{A_{ij}}\right)+1+\lambda_j=0 \qquad i,j=1,\ldots,n \\
	\sum_i A_{ij}-c_j=0 \qquad j=1,\ldots,n.
\end{align*}
A solution is easily seen to be
\begin{align*}
	B_{ij}=A_{ij}\frac{c_j}{\sum_k A_{kj}} \qquad i,j=1,\ldots,n
\end{align*}
The latter is the column renormalisation as in the RAS (Alg. \ref{alg:ras}).
\end{proof}

This implies that if the iterated I-projection converges to the I-projection of (\ref{eqn:Iproj}), then matrix scalability solves equation (\ref{eqn:Iproj}). This was supposedly proved in \cite{ire68}\footnote{In \cite{fie70}, it is pointed out that a simplified version of this proof appeared in \cite{dem69}, which however is unavailable to me.} and \cite{kul68}, but the proofs contain an error as pointed out in \cite{csi75} (see also \cite{bro93}). A corrected proof appeared in \cite{csi75}, however for some of the theorems it is not immediately clear whether more assumptions are needed as noted in \cite{bor94}. 

In addition, the proof in \cite{aar05} for positive matrices proves that the RAS converges using relative entropy as a ``progress measure''. He shows that it decreases under RAS steps to a unique stationary point. Another direct proof appeared in \cite{fra89}. 

At this point, let us make the following observation:
\begin{obs}[\cite{cot86}] \label{obs:ascent}
The RAS method can also be seen as the coordinate ascent method to the dual problem of entropy minimisation.
\end{obs}
This is justified as follows: When deriving the I-projections of each step of the algorithm, we set up the Lagrangian
\begin{align*}
	L(B,\lambda):=D(B\|A)+\lambda_j \left(\sum_i A_{ij}-c_j\right)
\end{align*}
and calculate its solution. This consists in explicitly solving the resulting equations for the Lagrangian multipliers $\lambda_j$. In this sense, the algorithm is not really a primal problem. This is also consistent with the nomenclature above: In Section \ref{sec:convexprog} we see that the dual problem of entropy minimisation is a convex program that is basically just the (negative) logarithmic barrier function above. Since the RAS is the coordinate descent algorithm of this problem, it is the coordinate ascent method of the dual problem of entropy minimisation. 

In other word, the justification of this observation is due to:

\begin{obs}[\cite{geo15,gur04}] \label{obs:logentrop}
Given a matrix $A\in\mathbb{R}^{n\times n}$ with nonnegative entries. Suppose there exist positive diagonal matrices such that $D_1AD_2$ has row sums $r$ and column sums $c$, then
\begin{align}
	-\ln\left(\inf\left\{(\prod_{i=1}^n r_i\sum_{j=1}^n A_{ij}x_j) \middle|\prod_{i=1}^n x_i^{c_i}=1	\right\}\right) 
		=\inf\{D(B\|A)|Be=r,B^Te=c\}
\end{align}
and in particular, the minimum is the scaled matrix. 
\end{obs}
The proof of this observation will essentially follow from the results in Section \ref{sec:convexprog}.

Let us finish this section by giving another proof sketch of Sinkhorn's theorem:

\begin{proof}[Sketch of proof of Theorem \ref{thm:sink} (Entropic version)]
We sketch the proof given in \cite{csi75} restricted to our scenario, which is similar to the proof in \cite{dar72} (see also \cite{csi89} for a comment on the connection). We prove convergence of Algorithm \ref{alg:ipro}, essentially by showing that the relative entropy of two successive iterations decreases to zero. 

Given a nonnegative matrix $A$, assume that there exists a matrix $B\prec A$ with required row- and column sums. Otherwise, the relative entropy will always be infinite and the problem has no solution.

The crucial observation is that if $A^{\prime}$ is the I-projection of $A$ onto $\Pi$, then for any $B\in \Pi$ we have
\begin{align}
	D(B\|A)=D(B\|A^{\prime})+D(A^{\prime}\|A). \label{eqn:pyth}
\end{align}
This ``Pythagorean identity'' usually only holds with $\geq$. The equality case is a special case of the ``minimum discrimination principle'' (\cite{kul59,kul66}) and it is proven for constraints $\Pi_i$ in \cite{csi75}. This equality leads to a very useful transitivity result (see also \cite{kuk68}) stating that if $A$ has I-projection $B$ on $\Pi_i$ for some $i$ and I-projection $B^{\prime}$ on $\Pi$, then $B$ has I-projection $B^{\prime}$ on $\Pi$. This is not necessarily true in the general case. 

Let $A^{\prime}$ be the I-projection of $A$ onto $\Pi$. Denoting by $A^{(n)}$ the repeated I-projection as defined in Algorithm \ref{alg:ipro}, repeated application of equation (\ref{eqn:pyth}) shows 
\begin{align*}
	D(A^{\prime}\|A)=D(A^{\prime}\|A^{(n)})+\sum_{i=1}^n D(A^{(n)}\|A^{(n-1)})
\end{align*}
Therefore, the sequence $A^{(n)}$ lies in a bounded set and hence contains a convergent subsequence by compactness. However, we also have that $D(A^{(n)}\|A^{(n-1)})\to 0$ for $n\to \infty$, which implies $\|A^{(n)}-A^{(n-1)}\|_{\infty}\to 0$ for $n\to \infty$, hence $A^{(n)}$ converges to some matrix $A^{\prime \prime}$. Clearly, $A^{\prime \prime}\in \Pi$, since $A^{(2n)}\in \Pi_1$ and $A^{(2n+1)}\in \Pi_2$ for every $n\in \N$. Using the transitivity of the I-projection, $A^{\prime \prime}$ is the I-projection of $A^{(n)}$ for all $n$ and equation (\ref{eqn:pyth}) holds in the form
\begin{align*}
	D(A^{\prime \prime} \|A^{(n)})=D(A^{\prime \prime}\|A^{\prime})+D(A^{\prime}\|A^{(n)})
\end{align*}
Since the first and last term converge to zero, $D(A^{\prime \prime}\|A^{\prime})=0$ and the I-projection $A^{\prime}$ is indeed the limit of Algorithm \ref{alg:ipro}.
\end{proof}
A similar proof can be found in \cite{bro93}.

\subsection{Convex programming and dual problems} \label{sec:convexprog}
Recall the logarithmic barrier function $g$ in equation (\ref{eqn:logbarrier}) and that it is not jointly convex. However, it is very beneficial to make $g$ convex for several reasons:
\begin{enumerate}
	\item Convex programming is efficient in the complexity theoretic sense (\cite{boy04}).
	\item The duality theory for convex programming is very well developed and can lead to new algorithms (see \cite{lit14} for a heuristic introduction and \cite{roc97,boy04} for a more careful analysis). 
	\item Uniqueness proofs can become simpler: A convex function has a unique minimum iff it is strictly convex at the minimum.
\end{enumerate}
To obtain a convex program, one simply needs to substitute $x=(e^{\xi_1},e^{\xi_2},\ldots,e^{\xi_n})$ and $y=(e^{\eta_1},e^{\eta_2},\ldots e^{\eta_n})$ into $g$ to obtain (\cite{mac77,kal96a}):
\begin{lem} \label{lem:convprog}
Given a nonnegative matrix $A\in \R^{n\times n}$, one can find diagonal matrices to scale $A$ to a matrix with row-sum $r$ and column sum $c$ if and only if the function
\begin{align}
	f(\xi,\eta):=\sum_{ij=1}^n A_{ij}e^{\eta_i+\xi_j}-\sum_{i=1}^n r_i\xi_i-\sum_{j=1}^n c_j\eta_j \label{eqn:gormanconvex}
\end{align}
attains its minimum on $\xi,\eta\in \R^n_{-\,0}$. 
\end{lem}
A proof based on this approach can be found in \cite{bac79}.\footnote{In \cite{bac70} it is also noted that the function is used in the approach by \cite{gor63} later to be simplified by \cite{bin65}. Both papers are unavailable to me.} We have already seen:
\begin{obs}
The convex programming formulation in Lemma \ref{lem:convprog} is equivalent to the logarithmic barrier function approach in Lemma \ref{lem:logbar}.
\end{obs}
Likewise, it can be shown:
\begin{obs} \label{obs:wolfedual}
The convex programming formulation \ref{lem:convprog} is the Wolfe dual (\cite{mac77,kru79}) or Lagrangian dual (\cite{bal04}) of the entropy minimisation approach.
\end{obs}
\begin{proof}
The entropy minimisation problem was given as:
\begin{align*}
		\inf_{B_{ij}} & \sum_{ij} B_{ij} \ln (B_{ij}/A_{ij}) \\
		\mathrm{s.t.} &\sum_{i} B_{ij}=p_j \quad \sum_{j} B_{ij}=q_i
\end{align*}
This implies that the Wolfe dual is given by
\begin{align*}
	\sup_{B_{ij}} & \sum_{ij} B_{ij}\ln(B_{ij}/A_{ij})+\sum_{j} u_j\left(\sum_i B_{ij}-p_j\right)+\sum_i v_i\left(\sum_j B_{ij}-q_i\right) \\
	\mathrm{s.t.} & \ln(A_{ij}/B_{ij})+1+u_j+v_i = 0 \quad \forall i,j \\
	&u,v\geq 0
\end{align*}
The constrained can be rewritten as
\begin{align*}
	B_{ij}=A_{ij}\exp(-1-u_j-v_i)
\end{align*}
and inserting this into the Wolfe dual function (see e.g.~\cite{bot10}) we obtain:
\begin{align*}
	\sup_{B_{ij}} \sum_{ij} B_{ij}\ln(B_{ij}/A_{ij})+\sum_{j} u_j\left(\sum_i B_{ij}-p_j\right)+\sum_i v_i\left(\sum_j B_{ij}-q_i\right) \\
		= \sup_{u,v} -\left(\sum_{ij} A_{ij} \exp(-1-u_j-v_i)-\sum_j u_jp_j-\sum_i v_iq_i\right)
\end{align*}
which is (up to the constant $\sum_{ij}A_{ij}/e$) the optimisation problem in \ref{lem:convprog}. The calculation for the Lagrangian dual is similar (see \cite{bal04}).
\end{proof}
Another connection to the barrier function is to use \emph{geometric programming:} 
\begin{obs} \label{obs:geometric}
Minimisation of the logarithmic barrier function $g$ is equivalent to
\begin{align*}
	\min~& y^TAx \\
	\mathrm{s.t.} & \prod_{i=1}^n x_i^{c_i}=1, \prod_{i=1}^n y_i^{r_i}=1
\end{align*}
This is in standard form of a geometric program, which implies that a substitution $\xi=\ln(x),\eta=\ln(y)$ gives a convex program (\cite{boy04}, Section 4.5.3).
\end{obs}
This observation was made in \cite{rot89a}, which also gives necessary and sufficient conditions for a matrix to be scalable or approximately scalable.

As described in \cite{kal96a}, one can also reduce the problem to an unconstrained optimisation problem for only a single variable by taking the formulation of Lemma \ref{lem:london} and substituting $x=\exp(\xi)$ as above to obtain the minimising function
\begin{lem} \label{lem:convprogsimpler}
Given a nonnegative matrix $A\in \R^{n\times n}$, one can find diagonal matrices to scale $A$ to a matrix with row sums $r$ and column sums $c$ if and only if the function
\begin{align}
	f(\xi)=\sum_{j=1}^n r_j\ln\left(\sum_{i=1}^n A_{ij}e^{\xi_i}\right)-\sum_{i=1}^n c_i\xi_i.
\end{align}
attains its minimum on $\xi\in \R^n_{-\,0}$. 
\end{lem}

Finally, let us return to entropy minimisation: Relative entropy is jointly convex and therefore a convex program. In fact, relative entropy is a special case of a broader class of functions called \emph{Bregman divergences} which we will sketch in Section \ref{sec:generalentrop}. 

A proof of Theorem \ref{thm:sink} using convex programming is often similar to the approach in Section \ref{sec:logbarrier}. The advantage is that any critical point is automatically a minimum and one does not need to consider the boundary.

\subsection{Topological (non-constructive) approaches} \label{sec:topidea}
In the proof of Theorem \ref{thm:sink} in Section \ref{sec:perfro}, the result was achieved by Brouwer's fixed point theorem, but it is only one of many topological proofs.

For every nonnegative matrix $A$ with a given pattern, we want to decide whether a scaling with prespecified row- and column sums exists. Assume that we also know the set of possible row- and column sums for a given pattern. In a sense, we therefore just have to prove that the map $\phi:\R^n\times \R^n\to \R^{n\times n}$ defined via $(D_1,D_2)\mapsto D_1AD_2$ hits all row- and column sums, or else: we need to see that the map
\begin{align}
	\begin{split}
		\phi^{\prime}: &\R^n_+\times \R^n_+\to \R^n_+\times \R^n_+ \\
		&(D_1,D_2)\mapsto (p,q):~p_i=\sum_{j} (D_1AD_2)_{ij},q_j=\sum_i (D_1AD_2)_{ij} \label{eqn:topmap1}
	\end{split}
\end{align}
is onto. This is somewhat problematic, because the spaces involved are not compact, but by normalising both diagonal matrices and row- and column-sums, one can consider the map as a map from a compact space into itself. This approach was taken in \cite{bap82} for positive matrices (based on his thesis) and the map was shown to be surjective using a topological theorem, which Bapat claims is sometimes known as Kronecker's index theorem.\footnote{I could not find any other instance of where the theorem is given that name. The theorem simply states that for any map $f:D^{n+1}\to D^{n+1}$, if $f$ maps $\partial D^{n+1}$ into itself and is of nonzero degree, then it must be surjective.} The case for general nonnegative matrices could only be covered by combining the approach with \cite{rag84} (see \cite{bap89a}).

\cite{rag84} uses yet another fixed point theorem (Kakutani's fixed point theorem of set-valued maps). Defining the set $K$ of all matrices in $\R^{n\times m}$ with prescribed marginals and zero (sub)pattern of the a priori matrix $A$, he considers the map 
\begin{align}
	\phi(H)=\{Z| Z\in K, \max_{Z^{\prime}} \langle C(H),Z^{\prime}\rangle = \langle C(H),Z\rangle
\end{align}
where $C(H)_{ij}=\log(A_{ij}/H_{ij})$ (if $A_{ij}>0$, and $0$ else), we take all matrices as vectors in $\R^{nm}$ and the usual scalar product. The fixed point theorem then implies that there exists $H$ such that
\begin{align*}
	\max_{Z^{\prime}\in K} \langle C(H),Z\rangle=\langle C(H),H\rangle
\end{align*}
and using the dual of this maximisation, one can show that it scales the matrix. 
\begin{obs} \label{obs:topentrop}
There is a simple connection to entropy minimalization, since $\langle C(H),H\rangle=D(H\|A)$.
\end{obs} 

However, we can also take the converse road: Instead of exploring the possibilites for every $A$, we can start with the set of matrices with prescribed row- and column sums and matrix pattern $\mathcal{X}$ (call the set $\mathcal{M}(p,q,\mathcal{X})$) and map it to the set of all nonnegative matrices of pattern $\mathcal{X}$ (call it $\mathcal{M}(\mathcal{X})$) by diagonal equivalence, i.e.~consider the map:
\begin{align}
	\begin{split}
		\psi: &\R^n_+\times \R^n_+\times \mathcal{M}(p,q,\mathcal{X})\to \mathcal{M}(\mathcal{X}) \\
		&(D_1,D_2,B)\mapsto D_1BD_2 \label{eqn:topmap2}
	\end{split}
\end{align}	
Again, it would be enough to show surjectivity. As such, it cannot be injective, because we can obviously shift a scalar from $D_1$ to $D_2$, hence we would at least have to restrict the first coordinate of $D_1$ to be $1$. The resulting map $\psi^{\prime}$ is indeed a homeomorphism as shown in an overlooked paper of \cite{tve76}.

Another topological proof has recently been proposed in \cite{fri16}. To describe the approach, note that the following two statements are equivalent:
\begin{enumerate}
	\item There exist $D_1,D_2$ such that $D_1AD_2$ has row sums $r$ and column sums $c$.
	\item There exist $D_1^{\prime},D_2^{\prime}$ such that $D_1^{\prime}AD_2^{\prime}$ is a stochastic matrix with $D_1^{\prime}AD_2^{\prime}c=r$. 
\end{enumerate}
The proof is trivial, in fact $D_1^{\prime}=D_1$ and $D_2^{\prime}=D_2\diag(1/c)$. In \cite{fri16}, the author therefore restricts to stochastic matrices. To do this, he defines the following map:
\begin{align}
	\Phi_A: &\R^n_+\to \R^{n\times n}_+; \quad \Phi_A(x)=\diag(x)A/\diag(A^Tx)
\end{align}
A quick calculation shows that $\Phi_A^Te=e$, hence the matrix is always stochastic. Hence given a nonnegative matrix $A$ and row sums $r$ and columns sums $c$, the question of scalability is equivalent to the question whether there exists an $x\in \R^n$ such that $\Phi_A(x)c=r$. 

For positive matrices $A$ and any $c\in \R^n_+$, \cite{fri16} now proves scalability by proving that the map $\Phi_{A,c}:x\to \Phi_A(x)c$ is continuous as a set from $\R^n_{+\,0}\cap \{v|\sum_i v_i=1\}$ onto itself and a diffeomorphism from $\R^n_{+}\cap \{v|\sum_i v_i=1\}$ onto itself. The result is achieved using degree theory similar to \cite{bap89a}.

\subsection{Other ideas}
A very general approach to prove Theorem \ref{thm:sink} was provided in \cite{let74}, where matrix theorems are derived as a consequence of the following theorem:
\begin{thm}[\cite{let74}]
Let $X$ be a finite set, $(\mu(x))_{x\in X}$ strictly positive numbers and $\mathcal{H}$ a fixed linear subspace of $\R^X$. Then there exists a unique (nonlinear) map from $\R^X\to \mathcal{H}$ denoted $f\mapsto h_f$ such that
\begin{align*}
	\sum_{x\in X}[\exp(f(x))-\exp(h_f(x))]g(x)\mu(x)=0
\end{align*}
for all $g$ in $H$.
\end{thm} 
Sinkhorn's theorem follows as an easy corollary:
\begin{proof}[Sketch of proof of Theorem \ref{thm:sink}, \cite{let74}]
First, let $X\subset \{1,\ldots,m\}\times \{1,\ldots n\}$, then we first define the following maps:
\begin{align*}
	a: (\R^m,\R^n)\to \R^{m\times n}, &\quad (\xi,\eta)\to (\xi_i+\eta_j)_{ij} \\
	\pi: (\R^{m\times n})\to \R^X, & \quad (A_{ij})_{i=1,j=1}^{i=m,j=n} \to (A_{ij})_{(i,j)\in X}
\end{align*}
The second is just the natural projection from $\R^{m\times n}$ to $\R^X$. 

Now let $A\in \R^{m\times n}$ be a nonnegative matrix and let $X:=\{(i,j)|A_{ij}>0\}$ be its pattern. We already know that the pattern is a necessary condition for scalability, hence we know that there exists a $B\in \R^{m\times n}$ with row sums $r$ and column sums $c$. Given the pattern, we define $\mathcal{H}=\ran(\pi\circ a)$ the range of the composition of $\pi$ and $a$. 

Now let $F\in \R^{X}$ be the matrix with entries $F_{ij}:=\log(B_{ij}/A_{ij})$ and apply the theorem to $F$, i.e.~there exists a unique matrix $H\in \mathcal{H}$ such that
\begin{align*}
	\sum_{ij}\exp(F_{ij})A_{ij}G_{ij}=\sum_{ij} \exp(H_{ij})A_{ij}G_{ij} \quad \forall G\in \mathcal{H}
\end{align*}
But since $\exp(F_{ij})A_{ij}=B_{ij}$ and $H_{ij}=\xi_i+\eta_j$ for some $\xi\in \R^m$, $\eta\in \R^n$ by definition of $\mathcal{H}$, we have 
\begin{align*}
		\sum_{ij}B_{ij}G_{ij}=\sum_{ij} \exp(\xi_i)A_{ij}\exp(\eta_j)G_{ij} \quad \forall G\in \mathcal{H}
\end{align*}
which implies that $\exp(\xi_i)A_{ij}\exp(\eta_j)$ has row sums $r$ by taking $G=\pi\circ a(e_i,0)$ and column sums $c$ by taking $G=\pi \circ a (0,e_j)$ for the unit vectors $e_i\in \R^m,~e_j\in \R^n$.

Clearly, the choice of $(\xi,\eta)$ is unique up to $\ker(\pi\circ a)$, which can be made explicit and leads to the usual conditions.
\end{proof}

\subsubsection{Geometric proofs}
In principle, we have already two geometric interpretations of the RAS: First, the RAS is akin to iterated I-projections and second, the RAS is the application of a contractive mapping on a cone with a projective metric. Two other ``geometric'' proofs are known: 

\cite{fie70} shows that the RAS is a contractive mapping in the Euclidean metric using that the RAS preserves \emph{cross-ratios} of a matrix. Given a matrix, the products 
\begin{align}
	\alpha_{ijkl}:=\frac{A_{ij}A_{kl}}{A_{il}A_{kj}}
\end{align}
remain invariant. This was first observed in \cite{mos68}, where it was used to justify the use of the RAS in statistical settings (see Section \ref{sec:applic}). Fienberg then follows that if one associates any positive matrix to a point of the simplex
\begin{align*}
	S_{rc}=\{(A_{11},\ldots,A_{1c},\ldots,A_{r1},\ldots,A_{rc})| \sum_{ij}A_{ij}=1\}
\end{align*}
by normalising the matrix, then any point reachable by diagonal equivalence scaling lies on a certain type of manifold inside the simplex. Using some structural knowledge of these manifolds he then shows that each full cycle of the RAS corresponds to a contraction mapping with respect to the Euclidean metric. The result is general enough to cover multidimensional tables, but in this simplicity handles only positive matrices.

There is an interesting connection: While the cross-ratios within the matrix remain constant, Hilbert's metric is also closely connected to cross-ratios. In fact, the contraction ratio is connected to the largest cross-ratio within the matrix and it is not finite if the matrix contains zeros. In that case, the matrix does not easily define a contraction in Hilbert's metric. The same holds true for Fienberg's proof. 

\cite{bor98} consider the column space of scaled matrices $AS$ and notes that $RAS$ is doubly stochastic if the columns are included in the convex hull of the columns of $R^{-1}$ and the barycentre of the sets of the two columns coincide. The observation of the barycentre then leads them to a proof involving Brouwer's fixed point theorem once again. By some continuity argument, the proof can be extended to nonnegative matrices.

\subsubsection{Other direct convergence proofs}
Many papers contain direct convergence proofs, not least the original approach in \cite{sin64} and the proof of the full result \cite{sin67a} (another proof based on this approach is given in \cite{pre80}). The idea is to show that some seemingly unrelated quantity always converges. Often, this quantity turns out to be very much related to some potential barrier function or entropy and we already cited the approach in the corresponding section. 

One different proof is the short convergence proof of \cite{mac77} establishing that $\sum_j A_{ij}^{(n)}/\sum_j A_{ij}^{(n-1)}\to 1$ and similarly $\sum_i A_{ij}^{(n)}/\sum_i A_{ij}^{(n-1)}\to 1$ for every $i,j$. This proof is in some sense derived from Bacharach's approach (\cite{bac65,bac70}, see also \cite{sen06}) and is very straightforward.\footnote{Macgill also mentions yet another work that contains a proof of Theorem \ref{thm:sink}, namely \cite{her73}, however no details are given beyond the fact that it contains also approximate scaling.} In parallel to Bacharach's earlier work \cite{bac65} but not cited in his later \cite{bac70}, \cite{cau65} proved the convergence of the RAS method in the general case of multidimensional matrices via the same idea which he attributes to \cite{thi64} (see the appendix of \cite{cau65}).

A second direct proof of convergence in \cite{sin67b}, uses a norm difference as convergence measure. More precisely, he considers the map
\begin{align*}
	\phi(x,y)= \max_i \left(r_i^{-1} \sum_j x_iA_{ij}y_j\right) - \min_i \left(r_i^{-1} \sum_j x_iA_{ij}y_j\right)
\end{align*}
on the set of all $(x,y)\in \R^n_+\times \R^n_+$ with some boundedness condition on their entries and proves that $\phi(x,y)=0$ is achieved by two positive vectors. 

A third proof of direct and approximate scaling is given in \cite{puk09} by combining the approach of Bacharach with a simple $L^1$-error function borrowed from \cite{bal89b}.

\section{Equivalence scaling} \label{sec:equivscaling}
Let us now collect maximal results. A similar but scarcely referenced collection of results was provided in \cite{kru79}. We follow the cleaner presentation style of \cite{rot89a}. Starting with equivalence scaling, we have:
\begin{thm}
Let $A\in \R^{n\times m}$ be a nonnegative matrix and $r\in \R^n_+,c\in \R^m_+$. Then the following are equivalent:
\begin{enumerate}
	\item There exist positive diagonal matrices $D_1,D_2$ such that $D_1AD_2$ has row sums $r$ and column sums $c$.
	\item There exists a matrix $B$ with row sums $r$ and column sums $c$ with the same pattern as $A$ (\cite{men68,bru68}).
	\item There exists no pair of vectors $(u,v)\in \R^n\times \R^m$ such that (\cite{rot89a})
		\begin{align*}
			u_i+v_j\geq 0 \quad \forall (i,j)\in \supp(A) \\
			r^Tu+c^Tv\leq 0 \\
			\mathrm{either}~u_i+v_j>0~\mathrm{for~some~}(i,j)\in \supp(A)\mathrm{~or~} r^Tu+c^Tv<0 
		\end{align*}
	\item For every $I\subset \{1,\ldots,m\},J\subset \{1,\ldots,n\}$ such that $A_{I^cJ}=0$ we have that 
		\begin{align*}
			\sum_{i\in I}r_i\geq \sum_{j\in J}c_j
		\end{align*}
		and equality holds if and only if $A_{IJ^c}=0$ (\cite{men69}).
	\item The RAS method converges and the product of the diagonal matrices of the iteration also converges to positive diagonal matrices (\cite{sin67a}).
\end{enumerate}
\end{thm}
The equivalence of the first two items was essentially established in the proof sketches in section (\ref{sec:approaches}). The equivalence to the fourth item follows from the characterisation of matrix patterns (see appendix \ref{app:prelimmatrices}) and the third follows from studying the geometric program \ref{obs:geometric}.

For doubly stochastic scaling, using the classification of doubly stochastic patterns, we then know that scalability is equivalent to having total support (cf. \cite{csi72}). The scaling matrices $D_1,D_2$ are unique up to scalar multiplication if and only if $A$ is fully indecomposable.

For approximate equivalence scaling, the results are similar. The only difference is that certain elements of $A$ can become zero in the limit (which implies that elements of $D_i$ must become zero and others infinite, hence the diagonal matrices cannot exist):
\begin{thm}
Let $A\in \R^{n\times m}$ be a nonnegative matrix and $r\in \R^n_+,c\in \R^m_+$. Then the following are equivalent:
\begin{enumerate}
	\item For every $\varepsilon>0$ there exist diagonal matrices $D_1,~D_2$ such that $B=D_1AD_2$ satisfies
		\begin{align*}
			\|Be-r\|<\varepsilon, \|B^Te-c\|<\varepsilon
		\end{align*}
	\item There exists a matrix $A^{\prime}\prec A$ such that $A^{\prime}$ is scalable to a matrix $B$ with row sums $r$ and column sums $c$.
	\item There exists a matrix $B\prec A$ with row sums $r$ and column sums $c$ (\cite{sch80}).
	\item There exists no pair of vectors $(u,v)\in \R^{n\times m}$ such that (\cite{rot89a})
		\begin{align*}
			u_i+v_j\geq 0 \quad \forall (i,j)\in \supp(A) \\
			r^Tu+c^Tv<0
		\end{align*}
	\item For every $I\subset \{1,\ldots,m\},J\subset \{1,\ldots,n\}$ such that $A_{I^cJ}=0$ we have that 
		\begin{align*}
			\sum_{i\in I}r_i\geq \sum_{j\in J}c_j
		\end{align*}
	\item The RAS method converges (\cite{sin67a} for the d.s. case).
\end{enumerate}
\end{thm}

For doubly stochastic scaling, using the classification of doubly stochastic patterns, we have that approximate scalability is equivalent to $A$ having support. Using \cite{sch80}, this is a trivial consequence of the fact that a matrix has total support if and only if it has doubly stochastic pattern and Proposition \ref{prop:fullyindec}\footnote{One recent observation of this is in \cite{bra10}. The observation has however already been made before such as in \cite{ach93}}.

The uniqueness conditions are also simple enough to state:
\begin{thm}
Let $A\in \R^{n\times m}$ be a nonnegative matrix and $r\in \R^m_+$, $c\in \R^n_+$. Then, $A$ has at most one scaling.

Furthermore, if there exist no permutations $P,Q$ such that $PAQ$ is a direct sum of block matrices, then $D_1,D_2$ are unique up to scalar multiples. Otherwise, the scaled matrices $D_1,D_2$ are only unique up to a scalar multiple in each block. 
\end{thm}
For doubly-stochastic scaling, this result already appears in \cite{bru66}. For the case of general marginals, it occurs in \cite{men68} and for general matrices and marginals in \cite{her88,men69}. The tools can also be applied to prove that the approximately scaled matrix is unique.

Let us now have a closer look at the difference between approximate scaling and equivalence scaling. What can be said about the convergence of the RAS?
\begin{thm}[\cite{pre80}, Theorem 1] \label{thm:pretzel}
Let $A\in\mathbb{R}^{n\times n}$ be a matrix that is approximately scalable to a matrix with row sums $r$ and column sums $c$. Let $B$ be a matrix with row sums $r$ and column sums $c$ with maximal subpattern of $A$ (i.e.~the number of entries $(i,j)$ such that $B_{ij}=0$ and $A_{ij}>0$ is minimal). 

Then $A$ converges to a matrix $C\prec B$ and the same result holds for $A^{\prime}$ with $A^{\prime}_{ij}=A_{ij}$ if $B_{ij}>0$ and $A^{\prime}_{ij}=0$ else.
\end{thm}

The continuity of the scaling can also be studied:
\begin{thm}
Let $A$ be nonnegative and $r,c$ be prescribed row- and column sums. Then the limit of the Sinkhorn iteration procedure is a continuous function of $A$ on the space of matrices with $r,c$-pattern.

When the scaling matrices are unique up to a scalar multiple, this also implies that the scaling is continuous in $D_1,D_2$.
\end{thm}
The first proof of this result limited to the doubly-stochastic case was given in \cite{sin72}. The full result follows directly from the homeomorphism properties of the map (\ref{eqn:topmap2}) from \cite{tve76}. A discussion is also presented in \cite{kru79}. Furthermore, the continuity can be achieved using arguments of Section \ref{sec:convstab}.

Finally, let us mention another characterisation of equivalence scaling using \emph{transportation graphs}. 

Following \cite{sch90b}, let $A\in \R^{m\times n}$ be a nonnegative matrix. Let $M=\{1,\ldots,m\},N=\{1,\ldots,n\}$ and consider the bipartite graph with the bipartition given by the vertices $M$ and $N$ and the edges defined via $E=\{(i,j):A_{ij}>0\}$, directed from $i\in M$ to $j\in N$. Now we define a source $S_1$ that connects to each vertex in $M$, where the edges have capacity $r_i$ (corresponding to the edge from $S_1$ to $i\in M$) and we define a sink $S_2$ that is connected from every vertex in $N$, where the edges have capacities $c_j$ (see Fig. \ref{fig:transportationgraph} for an example).
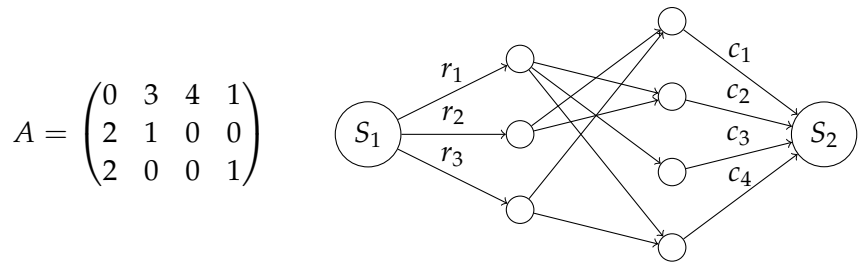
\begin{figure}[htbp]
\begin{center}
\begin{tikzpicture}

	\node at (0,-2) (mat) {$A=\begin{pmatrix}{} 0 & 3 & 4 & 1 \\ 2 & 1 & 0 & 0 \\ 2 & 0 & 0 & 1\end{pmatrix}$};
	
	\node[circle,draw] at (3,-2) (S1) {$S_1$};
	\node[circle,draw] at (5,-1) (r1) {};
	\node[circle,draw] at (5,-2) (r2) {};
	\node[circle,draw] at (5,-3) (r3) {};
	\node[circle,draw] at (7,-0.5) (c1) {};
	\node[circle,draw] at (7,-1.5) (c2) {};
	\node[circle,draw] at (7,-2.5) (c3) {};
	\node[circle,draw] at (7,-3.5) (c4) {};
	\node[circle,draw] at (9,-2) (S2) {$S_2$};
	
	\path[->] (S1) edge node[above]{$r_1$} (r1);
	\path[->] (S1) edge node[above]{$r_2$} (r2);
	\path[->] (S1) edge node[above]{$r_3$} (r3);
	
	\path[->] (c1) edge node[above]{$c_1$} (S2);
	\path[->] (c2) edge node[above]{$c_2$} (S2);
	\path[->] (c3) edge node[above]{$c_3$} (S2);
	\path[->] (c4) edge node[above]{$c_4$} (S2);
		
	\path[->] (r1) edge (c2);
	\path[->] (r1) edge (c3);
	\path[->] (r1) edge (c4);
	\path[->] (r2) edge (c1);		
	\path[->] (r2) edge (c2);
	\path[->] (r3) edge (c1);
	\path[->] (r3) edge (c4);
	
\end{tikzpicture}
\end{center}
\caption{An easy example of the transportation graph for row sums $r$ and column sums $c$ corresponding to the pattern of the matrix $A$. This example is similar to an example in \cite{sch90b}}
\label{fig:transportationgraph}
\end{figure} 

Then it is easy to see that the matrix is approximately scalable if and only if the maximum flow of this network is equal to $\sum_i r_i$. The flows along the edges $E$ then define a matrix with the wanted pattern. The matrix is exactly scalable if and only if the maximum flow of this network is equal to $\sum_i r_i$ and every edge contains flow. 

\section{Other scalings}
The problem of equivalence scaling is closely connected to different forms of scalings, the most prominent ones asking for a diagonal matrix $D$ such that $DAD$ is row-stochastic or such that $DAD^{-1}$ has equal row and column-sums. 

Many modern approaches to matrix equivalence scaling are general enough to cover most of those different scalings (see Section \ref{sec:generalised}).

\subsection{Matrix balancing}
Given a matrix $A$, does there exist a matrix $D$ such that $DAD^{-1}$ has equal column- and row sums? Clearly, this is a special case of $D_1AD_2$ scaling with a different set of constraints. We have the following characterisation:

\begin{thm}
Let $A\in\R^{n\times n}$ be a nonnegative matrix. Then the following are equivalent:
\begin{enumerate}
	\item There exists a diagonal matrix $D$ such that $B=DAD^{-1}$ fulfills $\sum_{i=1}^n B_{ij}=\sum_{i=1}^n B_{ji}$.
	\item $A$ is completely reducible or equivalently, a direct sum of irreducible matrices (\cite{har71}).
	\item There exists $B$ with the same pattern as $A$ and $\sum_{i=1}^n B_{ij}=\sum_{i=1}^n B_{ji}$ (\cite{let74}).
\end{enumerate}
The scaling of $A$ is unique and $D$ is unique up to scalars for each irreducible block of $A$. 
\end{thm}

The problem was first considered in \cite{osb60} in the context of preconditioning matrices (see Section \ref{sec:applic}) by proposing an algorithm and proving its convergence (and uniqueness). \cite{gra71}, building on Osborne's results, considers the matrix balancing method and provides an algorithm and convergence proof for completely reducible matrices. Unaware of the effort of Osborne and Grad, but considering ``the analogue of [Sinkhorn's] result in terms of irreducible matrices'' \cite{har71} proves essentially the same result. His approach is based on a progress measure which is basically the maximum difference of the row- and column sums. \cite{let74} provided an interpretation in terms of patterns. The same was later proved in \cite{sch80,gol83,eav85} by yet different means.

Similar to the RAS method, one can propose a simple iterative approximation algorithm:
\begin{alg}[\cite{sch90b}]
Let $A\in\R^{n\times n}$ be nonnegative. Let $A^0:=A$. For $k=0,1,\ldots$ we define the steps
\begin{enumerate}
	\item For $i=1,\ldots,n$, let $u_i=\sum_{j=1}^n A_{ij}^k$ be the row sum and similarly $v_i$ be the column sum. Then define $p$ as the minimum index such that $|u_p-v_p|$ is maximal among $|u_i-v_i|$. 
	\item Define $\alpha_k$ such that $\alpha_k u_p=1/\alpha_k v_p$.
	\item Let $D=\diag(1,\ldots,1,\alpha_k,1,\ldots,1)$ with $\alpha_k$ at the $p$-th position. Then define $A^{k+1}=DA^+D^{-1}$ and iterate.
\end{enumerate}
\end{alg}
According to \cite{sch90b}, this algorithm is also similar to the proposed scheme in \cite{osb60}. At any step, the $p$-th row is already correctly scaled, while all other rows change their scaling a bit. Note that unlike in the RAS method, the selection of the row and column to be scaled are done using norm differences. Given the results of \cite{bro93} that the RAS converges regardless of the order of column and row sum normalisations, a similar condition might also accelerate RAS convergence.

We have the following observation:
\begin{prop}[e.g.~\cite{sch90b}]
The algorithm converges to a balanced matrix $B$. This matrix is also the unique minimiser of the function
\begin{align}
	\sum_{i,j=1}^n \left(B_{ij}\ln\left(\frac{B_{ij}}{A_{ij}}\right)-B_{ij}\right)
\end{align}
subject to the balancing conditions.
\end{prop}
\begin{proof}[Sketch of proof]
The fact that the balanced matrix minimises the entropy functional can be seen by direct calculation (the minimiser must be a scaling of the original matrix and the balancing conditions ensure that the scaling is of the form $DAD^{-1}$). 

A proof is similar to observation \ref{obs:rasentrop}: Each step of the algorithm is an I-projection onto the set of matrices with only one row/column balancing constraint. Since the conditions are linear, the repeated projection will converge. 

It remains to see that the order of the projections does not matter as long as all directions are chosen arbitrarily often.
\end{proof}

As with equivalence scaling, a graph version of this problem exists, this time using \emph{transshipment graphs}. A nice description can be found in \cite{sch90b} (see also Figure \ref{fig:transshipmentgraph}): Given a nonnegative matrix $A\in \R^{n\times n}$, let $V=\{1,\ldots,n\}$ and define the set of edges of the transshipment graph $(V,E)$ by $E=\{(i,j)|A_{ij}>0,i\neq j\}$. We can then add weights $A_{ij}$ to any edge $(i,j)$. A matrix is then balanced, if and only if the incoming flow at each vertex equals the outgoing flow.

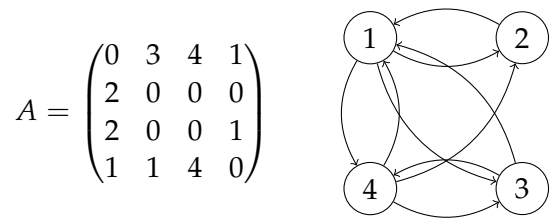
\begin{figure}[htbp]
\begin{center}
\begin{tikzpicture}
	\node at (0,-2) (mat) {$A=\begin{pmatrix}{} 0 & 3 & 4 & 1 \\ 2 & 0 & 0 & 0 \\ 2 & 0 & 0 & 1 \\ 1 & 1 & 4 & 0 \end{pmatrix}$};
	
	\node[circle,draw] at (3,-1) (1) {$1$};
	\node[circle,draw] at (5,-1) (2) {$2$};
	\node[circle,draw] at (5,-3) (3) {$3$};
	\node[circle,draw] at (3,-3) (4) {$4$};
	
	\path[->] (1) edge[bend right] (2);
	\path[->] (1) edge[bend right] (3);
	\path[->] (1) edge[bend right] (4);
	
	\path[->] (2) edge[bend right] (1);

	\path[->] (3) edge[bend right] (1);
	\path[->] (3) edge[bend right] (4);
	
	\path[->] (4) edge[bend right] (1);
	\path[->] (4) edge[bend right] (2);
	\path[->] (4) edge[bend right] (3);			
\end{tikzpicture}
\end{center}
\caption{An easy example of the transshipment graph corresponding to the pattern of the matrix $A$ similar to the example in \cite{sch90b}.}
\label{fig:transshipmentgraph}
\end{figure}

\subsection{DAD scaling}
Another closely related problem is the question, whether given a nonnegative matrix $A$, there exists a single diagonal matrix $D$ such that $DAD$ has prespecified row- or column sums. A short but quite good overview is given in \cite{joh09}. 

\paragraph*{Symmetric nonnegative matrices}

Let us first focus on the case where $A$ is symmetric. It seems natural that this follows directly from Sinkhorn's theorem: If $D_1AD_2$ has equal row-sums and $A$ is symmetric, so does $D_2AD_1$. By uniqueness of $D_i$ up to scaling, this implies that one can choose $D_1=D_2$. This was noted for example in \cite{sin64}. 

The first discussion of the case of symmetric $A$ can be traced back to the announcements \cite{mar61,max62}\footnote{This is covered in many papers, for instance \cite{mar68}.}. A first proof for the case of positive matrices and doubly stochastic scaling was given in \cite{sin64}. Shortly later, \cite{bru66} consider the case of doubly stochastic scaling for nonnegative matrices with positive main diagonal, while \cite{csi72} shows that a doubly stochastic scaling exists if and only if there exists a symmetric doubly stochastic matrix with the same zero pattern if and only if the matrix has total support. This was extended in \cite{bru74} to cover the case of arbitrary row sums giving the following theorem:

\begin{thm}[\cite{bru74}] \label{thm:symnonneg}
Let $A\in \R^{n\times n}$ be a symmetric nonnegative matrix. Then the following are equivalent:
\begin{enumerate}
	\item There exists a diagonal matrix $D$ with positive entries such that $DAD$ has row sums given by $r\in \R^n_+$.
	\item There exists a symmetric nonnegative matrix $B$ with the same pattern as $A$ and row sums $r$.
	\item For all partitions $\{I,J,K\}$ of $\{1,\ldots,n\}$ such that $A(J\cup K,K)=0$, $\sum_{i\in I} r_i\geq \sum_{i\in K} r_i$ with equality if and only if $A(I,I\cup J)=0$. 
\end{enumerate}
Furthermore, the scaling is unique.
\end{thm}
The equivalence of 2. and 3. is given in \cite{bru68}. 1. follows from 2. using Sinkhorn's theorem and the reverse direction is proved via contradiction. Using the uniqueness in Sinkhorn's theorem then provides uniqueness for the scaling. 

Note that the following observation gives a very simple proof of Theorem \ref{thm:sink}:
\begin{obs}\label{obs:symvsequiv}
Let $A\in \R^{m\times n}$ be a matrix and $r\in \R^m_+,c\in \R^n_+$ be two prescribed vectors. Then $A$ has an equivalence scaling if and only if the following symmetric matrix $A^{\prime}$
\begin{align}
	A^{\prime}=\begin{pmatrix}{} 0 & A \\ A^T & 0 \end{pmatrix} \label{eqn:symmetricspecial}
\end{align}
has a row-sum symmetric scaling to $(r^{\prime})^T=(r^T,c^T)$. 
\end{obs}
\begin{proof}
First assume that there exist $D_1,D_2$ positive diagonal such that $D_1AD_2$ fulfills
\begin{align*}
	D_1AD_2e=r, \qquad D_2A^TD_1e=c.
\end{align*} 
Then setting $D^{\prime}:=\diag(D_1,D_2)$ we have 
\begin{align*}
	D^{\prime}A^{\prime}D^{\prime}=\begin{pmatrix}{} 0 & D_1AD_2 \\ D_2A^TD_1 & 0 \end{pmatrix}
\end{align*}
and clearly $D^{\prime}A^{\prime}D^{\prime}e=(r^T,c^T)^T$.

Conversely, if $A^{\prime}$ has a row-sum symmetric scaling $D^{\prime}$, by an analogous argument $A$ will have an equivalence scaling with row sums $r$ and column sums $c$.
\end{proof}
This was already known in the 70s, maybe even earlier; explicit formulations include \cite{rot94,kni08,kni12}. Note that the observation can easily be extended to not just row- and column sums, but all $p$-norms for $0<p\leq \infty$ as considered in Section \ref{sec:normscaling}. It can also be extended to approximate scalings with the same proof. This implies:
\begin{obs}
Results from symmetric scaling for symmetric nonnegative matrices $A$ can always be translated to cover equivalence scaling for arbitrary nonnegative matrices. 
\end{obs}
The other direction is not true, since clearly not all symmetric matrices are of the special form (\ref{eqn:symmetricspecial}). However, it can still be beneficial to study equivalence scaling on its own, as many algorithms (e.g.~the RAS) do not preserve symmetry.

\paragraph*{Arbitrary symmetric matrices}
Theorem \ref{thm:symnonneg} can be generalised to cover matrices that are not necessarily nonnegative:
\begin{thm} \label{thm:dadscalings}
Let $A\in \R^{n\times n}$ be symmetric and $\lambda\in\R^n_+$ prescribed column sums. Then:
\begin{enumerate}
	\item If $A$ is positive semidefinite, then $A$ is scalable if and only if $A$ is strictly copositive (\cite{kal90,kal96b}).
	\item Any principal submatrix of $A$ (including $A$) is scalable if and only if $A$ is strictly copositive (\cite{joh09}).
\end{enumerate}
In general, at least one of the following two propositions is true (\cite{kal96b}):
\begin{enumerate}
	\item The following set is not empty:
		\begin{align}
			\{x\in \R^n| x^TAx=0,x\geq 0,x\neq 0\} \label{eqn:setdad}
		\end{align}
	\item For all $\lambda\in \R^n_+$ with $\lambda>0$ there exists a positive diagonal matrix $D$ such that $DADe=\lambda$. In other words, for any set of prescribed row sums, there exists a scaling.
\end{enumerate}
\end{thm}
More general conditions for scalability of arbitrary symmetric $A$ can be found in \cite{joh09}. We make a number of remarks concerning the results:
\begin{enumerate}
	\item Another necessary condition for scalability (the matrix must be \emph{diluted}) is provided in \cite{liv04}.
	\item The question of equivalent conditions for the scalability of matrices remains open. However, these conditions might not have a very useful description, since scalability of arbitrary symmetric matrices is NP-hard (\cite{kha96}\footnote{This was conjectured also in \cite{joh09}, who noted that deciding whether a matrix is (strictly) copositive is NP-complete according to \cite{mur87}. The authors seemed to have been unaware of the paper by Khachiyan. The alternative in Theorem \ref{thm:dadscalings} is also not very useful computationally, because deciding the emptiness of the set (\ref{eqn:setdad}) is also NP-hard (\cite{kal90}, according to \cite{kal96b}).}). 
 \item The second result implies in particular that if a matrix is strictly copositive, it is scalable, which was first proved in \cite{mar68}. Note that positive definite matrices are in particular strictly copositive, which means that this result encompasses the claimed proofs of scalability of completely positive matrices in \cite{max62}. An elementary proof for matrices with strictly positive entries has recently appeared in \cite{joh09} based on an iterative procedure. 	
	\item For doubly stochastic scaling, the alternative conditions of \cite{kal96b} can also be derived using linear programming duality and/or the hyperplane separation theorem using extremely general methods of duality in self-concordant cones (\cite{kal98,kal99,kal05}).
	\item Scaling of the special class of Euclidean predistance matrices has been considered in \cite{joh05}. It turns out that all such matrices are scalable.
	\item Note that the equivalence conditions for positive semidefinite matrices can be strengthened. If a matrix is scalable and positive semidefinite, 
	\begin{align*}
		\mu:=\min\{x^TAx|x\geq 0,\|x\|_2=1\}
	\end{align*}
can be bounded in terms of the matrix dimension (cf. \cite{kha92}, where it is also noted that the scaling problem is related to linear programming). 
\end{enumerate}

Uniqueness of matrix scaling has also been studied:
\begin{prop}
Let $A\in \R^{n\times n}$ be symmetric and $\lambda\in \R^n_+$ prescribed row sums. Then
\begin{enumerate}
	\item If $A$ has two or more distinct scalings, then there exists a matrix $D$ such that $DAD$ has eigenvalues $+1$ and $-1$ (\cite{joh09}).
	\item For scalable positive definite matrices $A$ there exist $2^n$ diagonal matrices $D$ such that $DADe=\lambda$, one for each sign pattern of $D$ (\cite{ole03}). In particular, scaling by positive diagonal matrices is unique.
	\item If $A$ is positive semidefinite, then if $A$ is scalable to row sums $r$, the positive diagonal matrix is unique (\cite{mar68}).
\end{enumerate}
For the scaling of positive semidefinite matrices, upper and lower bounds on $\|D\|$ were derived in \cite{kha92,ole03}.
\end{prop}
\cite{joh09} also note that for nonnegative matrices uniqueness holds in particular if $A$ is primitive (including the case of positive matrices already covered in \cite{sin64}) or if $A$ is irreducible and there does not exist a permutation $P$ such that 
\begin{align*}
	PAP^T=\begin{pmatrix}{} 0 & B \\ B^T & 0 \end{pmatrix}.
\end{align*}

It is also very simple to give an algorithm of RAS type for this problem, using the observation that a $DAD$ scaling to row sums $\lambda$ exists if and only if $ADe=r/(De)$. This implies that any scaling is a fixed point of the map $\mathbf{T}_{\mathrm{sym}}:\R^n\to \R^n$ with $\mathbf{T}_{\mathrm{sym}}(x)=r/(Ax)$.
\begin{alg}[\cite{kni08}]
Let $A\in \R^{n\times n}$ be nonnegative and symmetric. For the algorithm, set $x_0=e$ and iterate
\begin{align}
	x_{n+1}=\mathbf{T}_{\mathrm{sym}}(x_n)
\end{align}
\end{alg}

\paragraph*{Nonsymmetric matrices}
If we do not restrict to symmetric matrices we can only hope to scale $A$ to a matrix with given row-sums. The only notable result seems to be:
\begin{prop}(\cite{sin66})
Let $A\in \R^{n\times n}$ be a positive matrix. Then there exists $D$ such that $DAD$ is stochastic.
\end{prop}
The theorem can be extended to cover arbitrary row sums. The first proof occurred in \cite{sin66}. Likewise, the proof in \cite{joh09} does not need symmetry of $A$. 

\subsection{Matrix Apportionment}
Another scaling problem which is interesting particularly for its applications, is asking for an equivalence scaling, but with the added constraint that the resulting matrix have integer entries. This is important for instance when attributing votes to seats in a parliament and has been applied as early as 1997 (\cite{bal97}, see also \cite{puk04} for one of many explicit accounts for actual changes). 

This problem, which is often called \emph{matrix apportionment} has first been studied in \cite{bal89a,bal89b}. Algorithms akin to the RAS method exist and others based on network flows can be obtained from \cite{rot07}; an overview and many references can be found in \cite{puk09}.

\subsection{More general matrix scalings} \label{sec:genscaling}
This review has so far largely been concerned with nonnegative matrix scaling, with the exception of symmetric $DAD$ scaling. This is understandable, as most of the applications concern nonnegative matrices. However, in view of completeness, let us mention a few of the (mostly quite recent) other cases of matrix scaling. 

\paragraph*{Arbitrary equivalence scaling}
While arbitrary $D_1AD_2$ scaling is interesting for real symmetric matrices, scalings of general real matrices have never sparked a similar amount of interest. It is merely known that the question whether or not a matrix is scalable is NP-hard (\cite{kha96}) - a question that has also been considered for matrices over the algebraic numbers in \cite{kal97}. Since the problem of nonnegative matrix scaling turns out to be equivalent to the existence of matrices with given pattern, it seems natural to ask whether the $(+,-,0)$-pattern of matrices with prescribed row- and column sums play a similar role. For positive diagonal scaling the sign pattern of the matrix cannot change and it is a necessary condition for scalability, which is not sufficient as shown in \cite{joh01}. Nevertheless, the authors achieve a characterisation of general matrix patterns (generalising \cite{bru68}, see also \cite{joh00,eis02c}).

\paragraph*{Complex matrices}
Let us first start with the definition
\begin{dfn}
Let $A\in \C^{n\times m}$ be a complex matrix, then $A$ is \emph{doubly quasistochastic} if all sums and columns sum to one. 
\end{dfn}
Note that in case all entries are nonnegative the matrix is doubly stochastic. For the rest of this section, let us restrict to square matrices. Quasistochasticity is interesting, because if $A$ is quasistochastic, then $F_n^*AF_ne_1=e_1$, where $e_1=(1,0,\ldots,0)^T$ and $F_n$ is the $n\times n$ discrete Fourier transformation. This is true since $F_ne_1=e$ and $e$ is an eigenvector of $A$ by quasistochasticity. A doubly quasistochastic matrix $A$ therefore satisfies that $F_n^*AF_n$ has $e_1$ as its first row and column. Repeating diagonal scalings and Fourier transform can then lead to new matrix decompositions.

The natural generalisation of $DAD$ scaling would be $D^*AD$-scalings for positive semidefinite matrices. These were first studied in \cite{per03} and later in \cite{per14}. Observing that the proof of \cite{mar68} extends to complex entries, the authors obtain already part of the following partial results:
\begin{thm}[\cite{per14}]
Let $A\in \C^{n\times n}$ be positive definite. Then there exist diagonal matrices $D_1,D_2$ such that $D_1AD_2$ is doubly quasistochastic. 

Neither $D_1,D_2$ nor the scaled matrices are necessarily unique. However, there exists at most one scaling with positive matrices $D_1,D_2$.
\end{thm}
The authors suggested that such scalings can be applied to generate highly entangled symmetric states. They furthermore conjectured that the number of such scalings would be upper-bounded, but this was disproved recently in \cite{hut16} by giving counterexamples for $n\geq 4$, which have infinitely many scalings. For $n=3$, there exist at most four scalings. An RAS type algorithm can be obtained from the fact that an equivalent version of Observation \ref{obs:rasperfro} also holds in the complex case.

\paragraph*{Unitary matrices}
For the subclass of unitaries, we proved the following theorem:
\begin{thm}[\cite{ide15}] \label{thm:unitsink}
For every unitary matrix $U\in U(n)$ there exist diagonal unitary matrices $D_1,D_2$ such that $D_1UD_2$ is doubly quasistochastic. Neither $D_1,D_2$ nor $D_1UD_2$ are generally unique, in fact in some cases there may even be a continuous group of scalings.
\end{thm}
An algorithm how to obtain $D_1,D_2$ similar to the RAS method is given and studied in \cite{vos14a}, however its convergence is unknown.

The theorem was conjectured in \cite{vos14a} and used later (\cite{vos14b,ide15}) to prove that any unitary matrix can be considered as a product of diagonal unitary matrices and Fourier transforms on principal submatrices. Recently, it has also been applied to prove an analogue of the famous Birkhoff theorem for doubly-stochastic matrices (\cite{vos16}). 

The proof of Theorem (\ref{thm:unitsink}) boils down to noticing that a scaling exists if and only if there exists a vector $x$ with $Ux=y$ and $|x_i|=|y_i|=1$ for all $i=1,\ldots,n$. This is a problem of symplectic topology in disguise and can be solved using a theorem in \cite{bir04}. When we published the theorem in \cite{ide15} we were unaware of the fact that this proof had in principle already been found, since the equation $Ux=y$ with $|x_i|=|y_i|=1$, which defines so called \emph{biunimodular vectors} (see for instance \cite{fuh15}), also pops up in several other places. In this context, essentially the same proof was described in \cite{lis11}. A first formal publication containing this proof was probably \cite{kor14} applying it to error-disturbance relations in quantum mechanics.

\section{Generalised approaches} \label{sec:generalised}
All of the approaches above can be generalised to some extend. Many can then incorporate also different scalings. With an eye towards matrix equivalence, we will attempt to see the different ways of generalisations and what can be gained. A quick summary can be found in Table \ref{tab:generalise}.

\afterpage{
	\clearpage
	\begin{landscape}
		\begin{table} \label{tab:generalise}
		\centering 
			\resizebox{1.35\textwidth}{!}{%

			\begin{tabular}{p{2cm}p{3.8cm}p{3.3cm}p{1cm}p{1cm}p{1cm}p{1cm}p{4.5cm}}
			\toprule

			\textbf{Type} & \textbf{Base case} & $D_1AD_2$ & $DAD$ & $DAD^{-1}$ & \textbf{multi\-dim.} & \textbf{conti\-nuous} & \textbf{additional generalisation} \\  \toprule

			Algorithmic approaches & RAS-type algorithms & with row or column norm constraints (also max-sum) & $\checkmark$ & $\checkmark$ & - & - & - \\ \midrule

			Axiomatic approach & $D_1AD_2$ scaling with inequality constraints & $\checkmark$ & - & - & - & - & - \\ \midrule

			Convex optimisation & $D_1AD_2$, $DAD$ scaling & $\checkmark$ & $\checkmark$ & - & - & - & also copositive matrices, complex scaling \\ \midrule

			Entropies & minimising relative entropy minimisation + (non)linear constraints & $\checkmark$ & ($\checkmark$) & ($\checkmark$) & $\checkmark$ & $\checkmark$ & special case of Bregman divergences; cross entropies; justifications \\ \midrule

			Letac's approach & no scaling: existence of some function & $\checkmark$ & $\checkmark$ & $\checkmark$ & - & - & completely different applications \\ \midrule

			Log linear models & scaling of prob. distributions $w_i=x_i\prod_j d_j^{C_{ij}}$ given $C,x$ with constraints $Cw=b$ & $\checkmark$ & $\checkmark$ & $\checkmark$ & $\checkmark$ & - & Different scalings defined via $C$ \\ \midrule

			N-L Perron-Frobenius Theory & fixed points of homogeneous maps on cones & $\checkmark$ & $\checkmark$ & - & - & $\checkmark$ & Maps in different vector spaces (such as positive maps); infinite matrices \\ \midrule

			Truncated matrix scaling & $X=\Lambda DAD^{-1}$ balancing with $L\leq X\leq U$ entrywise & also inequalities & $\checkmark$ & $\checkmark$ & - & - & - \\

			\bottomrule
			\end{tabular}
			}
			\caption{This table gives an overview about possible approaches to matrix scaling, their application to various different scalings discussed in Section \ref{sec:generalised} and additional possible applications.}
		\end{table}
	\end{landscape}
}

\subsection{Direct multidimensional scaling}
Especially in transportation planning, equivalence scaling of arrays with three indices has been important from the beginning. Except for nonlinear Perron-Frobenius theory, the approaches can be readily generalised to this case. As already pointed out, \cite{bro59} was the first to consider multidimensional scaling. According to \cite{eva74} (see also \cite{eva70}), Furness pointed out iterative scaling as a possible solution to certain transportation planning problems in the unpublished paper \cite{fur62}. Evans and Kirby themselves proved convergence in a limited scenario by extending the convex programming approach of equation (\ref{eqn:gormanconvex}) and proofs have been provided or pointed out in several other papers such as \cite{fie70,kru79}. The case of approximate multidimensional scaling is discussed in \cite{bro93}.

For multidimensional exact or approximate scaling, the convergence results of \cite{pre80} reflected in Theorem \ref{thm:pretzel} still hold. In addition, the order in which we normalise any of the indices of the multidimensional array is irrelevant:
\begin{thm}[\cite{bro93} and comment in \cite{bro59}]
Let $A$ be an array with $m$ indices (or dimensions) and let $i_k$ be the dimension of the array that is scaled in the $k$-th step. If each element of $\{1,\ldots,m\}$ appears in the sequence $\{i_1,i_2,\ldots\}$ infinitely often, then the scaling converges to the limit of the cyclic RAS method, the I-projection of $A$.
\end{thm}

\subsection{Log-linear models and matrices as vectors}
Most of the ideas above use matrices as matrices, as sets of numbers with two indices. One can likewise consider just vectors of numbers and define columns and rows by defining partitions of the vectors. This approach has the advantage that the generalisation to multidimensional matrices is immediate. It was probably pioneered by \cite{dar72}, although \cite{lam81} credit Murchland, who circulated his results later (\cite{mur77,mur78}\footnote{The papers were not available to me}). The approach was then taken on in \cite{bap89a} (see also \cite{bap97}, Chapter 6 for an overview and a more lucid presentation of their ideas). While \cite{dar72} used an entropic approach, \cite{bap89a} is based on a combination of optimisation and topological approaches as discussed in Section \ref{sec:topidea}. The same theorem is also proved in \cite{fra89} in a very elementary fashion and in \cite{rot89b} using optimisation techniques. 

The original goal of \cite{dar72} was not to study matrix scaling but rather obtaining probability distributions using so called log-linear models. Given a positive (sub)probability distribution $\pi$ over some finite index set $I$, a log-linear model is a probability distribution $p$ such that
\begin{align}
	p_i=\pi_i D \prod_{s=1}^d D_s^{C_{si}} \label{eqn:loglinearmod}
\end{align}
which satisfies some constraints $\sum_{i\in I} C_{si}p_i=k_s$. Here, $D$ and $D_s$ have to be determined while $C$ is given from the problem. The name derives from the fact that the solution is an exponential family of probability distributions.

Depending on the choice of $C$, one can write matrix balancing, equivalence scaling or $DAD$ scaling as finding a log-linear model.

To achieve equivalence scaling with row-sums $r$ and column sums $s$, consider for simplicity the case of a $2\times 3$ matrix. Then $C$ and $b$ are given by
\begin{align*}
	C=\begin{pmatrix}{} 1 & 1 & 1 & 0 & 0 & 0 \\ 0 & 0 & 0 & 1 & 1 & 1 \\ 1 & 0 & 0 & 1 & 0 & 0 \\ 0 & 1 & 0 & 0 & 1 & 0 \\ 0 & 0 & 1 & 0 & 0 & 1 \end{pmatrix}, \qquad b=\begin{pmatrix}{} r_1 \\ r_2 \\ s_1 \\ s_2 \\ s_3 \end{pmatrix}
\end{align*}
and we define $y_1=A_{11}, y_2=A_{12}, \ldots, y_5=A_{22}, y_6=A_{23}$ (example from \cite{bap89a,rot89b}). 

To achieve matrix balancing with row-sums equaling column sums, consider for simplicity the case $3\times 3$, then $C$ is given by
\begin{align}
	C=\begin{pmatrix}{} 0 & 1 & 1 & -1 & 0 & 0 & -1 & 0 & 0 \\ 0 & -1 & 0 & 1 & 0 & 1 & 0 & -1 & 0 \\ 0 & 0 & -1 & 0 & 0 & -1 & 1 & 1 & 0 \end{pmatrix} \label{eqn:balancing}
\end{align}
and $b=0$ and we order $x,y$ again as before (example from \cite{rot89b}).

We have the following theorem:
\begin{thm}[\cite{bap89a}] \label{thm:bapatraghavan}
Let $C\in \R^{m\times n}$ and $b\in \R^m_{+\,0}$. Let $K=\{v|Cv=b,v\geq 0\}$ be bounded. Let $x\in \R^n_{+\,0}$. Then there exists a $w\in K$ such that for some $D\in \R^n_+$ we have
\begin{align*}
	w_j=x_j\prod_{i=1}^m D^{C_{ij}},\quad j=1,\ldots,n
\end{align*}
if and only if there exists a vector $y\in \R^n_{+\,0}$ with $y\in K$ and the same zero pattern as $x$. 
\end{thm}
Note that this is a major generalisation of scaling as the matrix $C$ can contain any real numbers.

The limiting factor of the theorem is the boundedness of $K$. While the constraints in the case of matrix equivalence are bounded, the constraint set defined by (\ref{eqn:balancing}) is not necessarily bounded. \cite{rot89b} applies a completely different proof which only works for positive matrices. However we can still apply Theorem \ref{thm:bapatraghavan}: $K$ is unbounded, because the matrix entries can become unbounded since we only want equal row and column sums but do not specify them further. We fix that by using
\begin{align*}
	\tilde{C}=\begin{pmatrix}{} 0 & 1 & 1 & -1 & 0 & 0 & -1 & 0 & 0 \\ 0 & -1 & 0 & 1 & 0 & 1 & 0 & -1 & 0 \\ 0 & 0 & -1 & 0 & 0 & -1 & 1 & 1 & 0 \\ 1 & 1 & 1 & 1 & 1 & 1 & 1 & 1 & 1 \end{pmatrix}
\end{align*}
and $b_4=1$. The last row just implies that the sum of all matrix entries should be one which makes $K$ a bounded set. A simple calculation then shows that this is equivalent to searching for a diagonal matrix $D$ and a scalar $d$ such that $dDAD^{-1}$ has equal row- and column sums and the sum of all matrix entries is one. Clearly, this is equivalent to matrix balancing and we can apply Theorem \ref{thm:bapatraghavan}.

The connection to entropy minimisation is simple:
\begin{lem}[\cite{dar72}, Lemma 2] \label{lem:loglinentrop}
Given a positive (sub)probability distribution $\pi$, if a positive probability distribution $p$ satisfying (\ref{eqn:loglinearmod}) and the linear constraints exists, then it minimises relative entropy $\sum_i p_i\log(p_i/\pi_i)$ subject to the linear constraints.
\end{lem}
\begin{proof}
The proof in \cite{dar72} is a straightforward calculation and follows directly from \cite{kul66}. If $q$ is a probability distribution satisfying the linear constraints, then
\begin{align*}
	D(p||\pi) &= \sum_{i\in I} p_i(\log \xi + \sum_{s=1}^d C_{si}\log \xi_s) \\
		&=\log \xi \left(\sum_{i\in I} p_i\right)+\sum_{s=1}^d \log \xi_s \left(\sum_{i\in I} C_{si}p_i\right) \\
		&=\log \xi \left(\sum_{i\in I} q_i\right)+\sum_{s=1}^d \log \xi_s \left(\sum_{i\in I} C_{si}q_i\right) \\
		&=\sum_{i\in I} q_i\log (p_i/\pi_i) \\
		&=D(q||\pi)-D(q||p)
\end{align*}
which implies the lemma by the nonnegativity of relative entropy.
\end{proof}

\subsection{Continuous Approaches}
Nonnegative matrices were always tied to joint probability distributions. Obviously, there is no reason to only study discrete probability distributions. The first such generalisation was obtained in \cite{hob65}. Also the basic theorems of \cite{kul68} and \cite{csi75} are more general than counting measures (although both have problems with parts of their arguments, see \cite{bor94}). 

As pointed out in \cite{bor94}, there are essentially two approaches to continuous versions, the entropy maximisation approach studied by Kullback and later Csisz\'{a}r, and the approach via fixed point theorems or contractive ratios (one can see this as a precursor to nonlinear Perron-Frobenius theory) studied in, for instance, \cite{for41,nus87,nus93}. The natural continuous extension of the $DAD$ theorem for symmetric matrices was studied in \cite{now66,kar67} (via fixed points or iterative contractions). The most general results in \cite{bor94} combine these two approaches. To give a flavour of their results, we cite
\begin{thm}[\cite{bor94} Theorem 3.1]
Given a finite measure spaces $\mu(s,t)=k(s,t)\,ds\,dt$ and marginal distributions $\alpha(s),\beta(t)\in L^1(dt/ds)$, consider the following minimisation problem:
\begin{align}
	\begin{split}
		\min & \int_{S\times T} [u(x,y)\log(u(x,y))-u(x,y)]k(s,t)\,ds\,dt \\
			\mathrm{s.t.} &\int_T u(s,t)k(s,t)\,dt=\alpha(s)\quad a.e. \\
			&\int_S u(s,t)k(s,t)\,ds=\beta(t) \quad a.e.
	\end{split}
\end{align}
where $u\in L^1(dt,ds)$. Furthermore, we require $\int_S \alpha(s)\,ds=\int_T \beta(t)\,dt$. Then the minimisation problem has a unique optimal solution. If there exists a $u_0$ which fulfils the constraints and there exist $x_0\in L^{\infty}$ and $y_0\in L^{\infty}$ such that $\log u_0(s,t)=x(s)+y(t)$ almost everywhere, then $u_0$ is the unique solution. Conversely, if there exists a feasible solution $u$ with $u>0$ almost everywhere, then the unique optimal solution satisfies $u_0>0$ almost everywhere and there exist sequences $x_n\in L^{\infty}$ and $y_n\in L^{\infty}$ such that 
\begin{align*}
	\lim_{n\to \infty} (x_n(s)+y_n(t))=\log u_0(s,t)\quad a.e.
\end{align*}
\end{thm}
This in fact also covers the approximate scaling case. Note that the results also extend to more than two marginals.

\subsection{Infinite matrices}
Instead of continuous functions, we can also consider infinite matrices. Results usually posit that row and column sums should be finite in some norm. 

The first such result was obtained in \cite{net69}, which proves Theorem \ref{thm:sink} in the case where the column and row sums are uniformly bounded in the $l^1$-norm and the matrix entries are uniformly bounded. The proof is reminiscent of \cite{bru66} and Nonlinear Perron-Frobenius theory, using a fixed point argument involving Schauder's fixed point theorem.

Another approach was presented in \cite{ber79}, where matrices that are infinite in one direction are studied (rows or columns are finite in a $l^p$-norm). Once again, the matrix entries must be bounded uniformly (in this case, in a $l^p$-norm) and convergence of the iterative algorithm to a unique solution is proved in certain topologies.

\subsection{Generalised entropy approaches} \label{sec:generalentrop}
As we saw in Section \ref{sec:entrop}, we can write matrix scaling as the problem
\begin{align*}
	\min_P D(P\|Q) \quad \mathrm{s.t.~}P\in \Pi
\end{align*}
where $\Pi$ is an intersection of linear constraints. This approach can be generalised in two ways: First, one could consider other functions than relative entropy but related to it or second, one can consider more general sets $\Pi$. 

Telative entropy is a special case of \emph{Bregman divergences}. These were originally introduced in \cite{bre67a} and later named in \cite{cen81}. The idea is to study distance measures derived from functions $\phi:S\in\R^n\to \R$, which are defined on a closed convex set $S$, continuously differentiable and strictly convex. Then
\begin{align*}
	\Delta_\phi(x,y)=\phi(x)-\phi(y)-\langle \nabla \phi(y),x-y\rangle
\end{align*}
behaves similarly to a metric, although it is not necessarily symmetric and obeys no triangle inequality. If one takes $\phi(x)=\sum_{i} (x_i\log (x_i)-x_i)$ (negative entropy modulo the linear term), then $\Delta_\phi(x,y)=\sum_i (x_i\ln (x_i/y_i)-x_i+y_i)$. This example was already studied in \cite{bre67a} giving in addition an iterative algorithm to find the projections onto the minimum Bregman distance given linear constraints, which is a variant of the RAS method (see also \cite{lam81}).

Another way to generalise $D$ is the basic observation underlying \cite{mcd99}: Relative entropy for matrices is equivalent to the sum of cross-entropies between matrix columns, where a \emph{cross-entropy} of the column $j$ of the matrices $A,B$ is just $D_j(A\|B):=\sum_{i} B_{ij}\log(B_{ij}/A_{ij})$. Instead of taking the sum of all cross-entropies, it might be justified to take weighted sums of cross-entropies. This is relevant in economic settings and, aside from \cite{mcd99}, was studied in e.g.~\cite{gol96,gol97}\footnote{References corrected but taken from \cite{mcd99} as they were unavailable to me}. 

On the other hand, we can work with relaxed constraints. This was covered in \cite{bro93}: The extension of linear families of probability distributions is still covered by \cite{csi75}, while finding the I-projection for closed, convex but nonlinear constraints requires different means such as \emph{Dykstra's iterative fitting procedure} (cf. \cite{dyk85}).

\subsection{Row and column sum inequalities scaling} \label{sec:inequalityscaling}
Instead of wishing for matrices to have prespecified row and column sums, it might be interesting to consider cases where only lower and upper bounds on the row and column sums and the matrix entries are considered. 

If we denote the set of all nonnegative matrices with row sums between $r^-\in \R^n_+$ and $r^+\in \R^n_+$ and column sums between $c^-\in\R^n_+$ and $c^+\in\R^n_+$ and total sum of its entries $h$ by $R(r^-,r^+,c^-,c^+,h)$, then we can ask the question, whether for a given nonnegative $A\in\R^{n\times n}$, there exists $\delta>0$ and $D_1,D_2$ diagonal matrices such that
\begin{align*}
	B:=\delta D_1AD_2\in R(r^-,r^+,c^-,c^+,h)
\end{align*}
such that if $(D_1)_{ii}>1$, then $\sum_{j} B_{ij}=r^-$ and $(D_1)_{ii}<1$, then $\sum_j B_{ij}=r^+$ and the same conditions for $D_2$ and $c$. This problem was studied in \cite{bal89a,bal89b}, where they call such a matrix a \emph{fair share matrix}. The main purpose of the approach is described in Section \ref{sec:justbalinski}. Using the arguments from nonlinear Perron-Frobenius theory (Section \ref{sec:perfro}), they prove
\begin{thm}[\cite{bal89a}] \label{prop:balinski}
Let $A$ be a nonnegative matrix. There exists a unique fair share matrix for $A$ if and only if there exists a matrix $B\in R(r^-,r^+,c^-,c^+,h)$ with the same pattern as $A$.
\end{thm}

A different generalisation is called \emph{truncated matrix scaling}. It is studied in \cite{sch89,sch90a} and can also account for equivalence scaling. While the proofsuse a combination of optimisation techniques for an optimisation problem defined via entropy functionals, the motivation and interpretation uses graphs and transportation problems for graphs.

The problem considers matrix balancing and not equivalence scaling as its basic problem and then explains the connection. A problem consists of an ordered triple $(A,L,U)$ of nonnegative matrices in $\R^{n\times n}$ satisfying
\begin{enumerate}
	\item $0\leq L\leq U\leq \infty$,
	\item There is a matrix $X$ whose pattern is a subpattern of $A$ such that $X$ is balanced and $L\leq X\leq U$,
	\item $U_{ij}>0$ whenever $A_{ij}>0$.
\end{enumerate}
and asks for a diagonal matrix $D$ and a nonnegative matrix $\Lambda$ such that
\begin{enumerate}
	\item $X=\Lambda DAD^{-1}$ is balanced and $L\leq X\leq U$,
	\item $X$ and $\Lambda$ satisfy
		\begin{align*}
			\begin{cases} \Lambda_{ij}>1~\Rightarrow~X_{ij}=L_{ij} \\
				\Lambda_{ij}<1~\Rightarrow~X_{ij}=U_{ij}
			\end{cases}
		\end{align*}
\end{enumerate}
The conditions above are consistency conditions which are trivially necessary for the existence of $D,\Lambda$. One can easily see that for $L=0$ and $U=\infty$ the problem is equivalent to matrix scaling, because $\Lambda=\id$. Note the similarity of the treatment of inequalities to the ideas of Balinski and Demange.

The connection to equivalence scaling is simple (cf. \cite{sch89}): Given a nonnegative matrix $A$ and row and column sums $r,c$, we start with the graph of Figure \ref{fig:transportationgraph}: we join the two vertices $S_1$ and $S_2$ into one vertex (call it $S$), keeping everything else fixed. We label the edges between the nodes with the corresponding matrix entries $A_{ij}$. $A^{\prime}$ is now the matrix corresponding to the graph. 

Now we copy the graph twice and erase the weights of the edges and instead label the first graph by $l_{(i,j)}$ and the second by $u_{(i,j)}$ where
\begin{align*}
	l_{(i,j)}:=\begin{cases} 0 & \mathrm{if~} A_{ij}>0 \\ r_i & \mathrm{if~} j=0 \\ c_j & \mathrm{if~} i=0 \end{cases} 
	\quad 
	u_{(i,j)}:=\begin{cases} \infty & \mathrm{if~} A_{ij}>0 \\ r_i & \mathrm{if~} j=0 \\ c_j & \mathrm{if~} i=0 \end{cases}.
\end{align*}
Now $L^{\prime}$ ($U^{\prime}$) is the matrix corresponding to the graph with labels $l_{(i,j)}$ ($u_{(i,j)}$). Finally, $(A^{\prime},L^{\prime},U^{\prime})$ is the triple for truncated matrix scaling.

The main result of the paper then includes:
\begin{thm}[\cite{sch89} Theorem 14 (part of it)]
Let $(A,L,U)$ be a triple in $\R^{n\times n}$ fulfilling the consistency conditions 1.-3. above. Then the truncated matrix scaling has a solution satisfying the conditions 1. and 2. above if and only if there exists a balanced matrix $X$ such that 
\begin{itemize}
	\item $L\leq X\leq U$,
	\item $X_{ij}>0$ iff $A_{ij}>0$ always.
\end{itemize}
\end{thm}
Once again, the answer is dominated by the pattern of the matrix and the conditions boil down to the usual conditions for similarity scaling. In a sense, this gives another explanation as to why both problems, equivalence scaling and matrix similarity, need pattern conditions for feasibility: They are both similar graph-related problems.

\subsection{Row and column norm scaling} \label{sec:normscaling}
Instead of asking the question whether one can scale a matrix to prescribed row- and column sums, one can ask for a scaling to prescribed row- and column norms. 

For the $\infty$-norm, this is discussed in \cite{rot94}. Their proof relies on an algorithm for symmetric $DAD$ scaling using Observation \ref{obs:symvsequiv}. In fact, an algorithm for the problem had already been studied for the symmetric case in \cite{bun71}\footnote{Reference from \cite{kni08} among others. I could not obtain the reference.}.

\begin{thm}[\cite{rot94}]
Let $A\in \R^{m\times n}$ be a nonnegative matrix and $r\in \R^m_+$, $c\in \R^n_+$ be prescribed row and column maxima. Then the following are equivalent:
\begin{enumerate}
	\item There exist diagonal matrices $D_1$ and $D_2$ such that $D_1AD_2$ has prescribed row and column maxima $r$ and $c$.
	\item There exists a matrix $B$ with row and column maxima $r$ and $c$ with the same pattern as $A$.
	\item There exists a matrix $B$ with row and column maxima $r$ and $c$ with some subpattern of $A$.
	\item The vectors $r$ and $c$ fulfil
		\begin{align}
			\max_{i=1,\ldots m} r_i=\max_{j=1,\ldots n} c_j \\
			\max_{i\in I} r_i\leq \max_{j\in J^c} c_j \\
			\max_{j\in J} c_j\leq \max_{i\in J^c} r_i
		\end{align}
		for every subsets $I\subset \{1,\ldots,m\}$ and $J\subset \{1,\ldots,n\}$ such that $A_{IJ}=0$. 
\end{enumerate}
\end{thm}
There are two further technical conditions given in \cite{rot94} as well as an algorithm that converges to the solution.

The usual equivalence scaling now corresponds to $1$-norm scaling. For $p$-norms of row and columns with $0<p<\infty$, it is shown that this problem reduces to $1$-norm scaling in \cite{rot94}. If $A^{(p)}$ denotes the entrywise power, we have:

\begin{thm}[\cite{rot94}]
Let $A\in \R^{m\times n}$ be a nonnegative matrix and $r\in \R^m$, $c\in \R^n$. Then the following are equivalent:
\begin{enumerate}
	\item There exist matrices $D_1$ and $D_2$ such that $D_1AD_2=B$ has prescribed row and column $p$-norms $r$ and $c$.
	\item There exist matrices $D_1$ and $D_2$ such that $D_1^{(p)}A^{(p)}D_2^{(p)}$ has prescribed row and column sums $r^{(p)}$ and $c^{(p)}$. 
	\item There exists a matrix $B$ with the same pattern as $A$ and row and column sums given by $r^{(p)}$ and $c^{(p)}$. 
\end{enumerate}
\end{thm}
Hence the answer again reduces to a question of patterns. Likewise, the $\varepsilon$-scalability can immediately be transferred. Much weaker results were obtained in \cite{liv04}, where the problem of $2$-norm scaling was studied for arbitrary (not necessarily nonnegative) matrices. A (fast) algorithm is also derived in \cite{kni12}. 

\subsection{Row and column product scaling} \label{sec:rowprod}
At this point, one might wonder what happens when replacing the row- and column sums by row- and column products. This has been treated in \cite{rot92}, however it is not connected to entropy or maximum likelihood estimation, but instead to least square estimations, which is why we will not discuss the techniques here. However, this is interesting in light of the original justification of the RAS method in transportation planning by \cite{dem40}. The results are simple:

\begin{thm}[\cite{rot92}]
Let $A$ be a nonnegative matrix. Then the following are equivalent:
\begin{enumerate}
	\item There exist positive diagonal matrices $D_1$ and $D_2$ such that $D_1AD_2$ has row and column products $r_p$ and $c_p$.
	\item There exists a matrix $B$ with the same zero pattern as $A$ and row and column products $r_p$ and $c_p$.
\end{enumerate}
The scalded matrix $D_1AD_2$ is unique. 

Furthermore, if $A$ has no zero rows or columns, there always exists a matrix $D$ such that $DAD^{-1}$ has equal row and column products.
\end{thm}
Note that in the case of matrix balancing to equal row and column products, the result is also the same: This is possible if and only if a balanced matrix with the same pattern exists which is always the case (cf. \cite{rot92}, Theorem 5.2).

\section{Algorithms and Convergence complexity}
After the basic existence problems of matrix scaling were solved in the 60s to 80s, the focus shifted to algorithms and complexity theory in the 90s. The story is equally convoluted, not least because algorithmic complexity is difficult and not always well-defined in itself: One can decide to study worst case or average convergence speed, count algorithm steps or computational operations. Given the RAS method and the fact that it is a coordinate descent method for an intrinsically convex optimisation problem, which is amenable to a host of other techniques, the choice of a relevant class of algorithms is already not unique. 

Since this review is geared more towards the mathematical aspects of the problem, our focus will lie on exact complexity results instead of proofs by example. Papers focussed on numerical aspects appeared as early as the late 70s, early 80s with \cite{rob74} average convergence considerations, \cite{bac79} and \cite{par82}. A small overview about many of the recent developments can be found in \cite{kni12}.

\subsection{Scalability tests}
Most algorithms explicitly require that the matrix $A$ be scalable (or positive). This means that we first need to check for scalability.
\begin{prop}
Let $r\in \R^m_+,c\in\R^n_+$ be two positive vectors with $\sum_i r_i=\sum_j c_j$. Let $A$ be a nonnegative matrix, then one can check whether $A$ is approximately scalable in polynomial time $\mathcal{O}(pq\log(q^2/p))$ with $q=\min\{m,n\}$ and $p$ the number of nonzero elements in $A$. 

If $r\in \Q^m_+,c\in \Q^n_+$, then one can check for exact scalability in polynomial time of the same order.
\end{prop}
The fact that approximate scalability can be efficiently checked was probably first seen in \cite{lin00a}. A complete and well-readable proof giving explicit bounds appeared in \cite{bal04}, exact scalability can be found in \cite{kal08}.
\begin{proof}[Sketch of Proof]
We first follow the proof in \cite{bal04}, which uses the transportation graph described in Figure \ref{fig:transportationgraph}.

The matrix is approximately scalable iff the maximum flow of this network is equal to $\sum_i r_i$. The flows along the edges $E$ then define a matrix with the wanted pattern. Such a network flow problem can be solved in time $\mathcal{O}(pq\log(q^2/p))$ with $q=\min\{m,n\}$ and $p$ the number of nonzero elements in $A$ (\cite{ahu94}). 

In order to check for exact scalability, one has to check whether there exists a solution where each edge has a positive amount of flow (otherwise the entry would have to be reduced to zero). We can check for a solution to the maximum flow problem with minimum flow through each edge bigger than a prespecified value $\varepsilon$ with the same costs as solving a maximum flow problem twice. Clearly, this does not help as, we would have to check scalability for any $\varepsilon>0$. 

However (following \cite{kal08}) if $r,c$ have only rational entries, we can find a number $h$ such that $hr,hc$ have only integer values. In this case, the flow problem has a solution iff there exists a matrix $B$ with column sum $hc$ and row sum $hr$ where each positive entry fulfils $B\geq 1/|E|$, where $|E|$ denotes the number of edges in $E$. 

This implies that it suffices to check for a solution with capacities $hr, hc$ and minimum flow through each edge having prespecified value $1/(2|E|)$.
\end{proof}
For positive semidefinite matrices, scalability can also be checked easily:
\begin{prop}[\cite{kha96}]
Let $A\in\R^{n\times n}$ be positive semidefinite. Then $A$ is scalable if and only if $Ax=0$ and $e^Tx=1$ has no solution $x\geq 0$. This can be tested by a linear program.
\end{prop}
\begin{proof}
The formulation is already nearly in canonical form. We maximize $e^Tx$ subject to the equality constraints $Ax=0$ and $x\geq 0$. 
\end{proof}
For arbitrary matrices scalability is mostly NP-hard (see Section \ref{sec:genscaling}).

\subsection{The RAS algorithm}
The RAS algorithm, being the natural algorithm to compute approximate scaling, is also the most studied algorithm. For the case of doubly stochastic matrices, it has long been known (cf. \cite{sin67b}) that for positive matrices, the RAS converges linearly (sometimes called geometrically) in the $l_{\infty}$ norm. \cite{kru79} gave a simple argument that the iteration will get better at any step. This can also be inferred from the fact that the RAS method is iterated I-projection onto a convex set using \cite{csi75}. Later, \cite{fra89} showed that the convergence is also linear in Hilbert's projective metric, while \cite{sou91} showed linear convergence for all exactly scalable matrices basically in arbitrary vector norms, albeit without explicit bounds. Conversely, it was shown that only scalable matrices can have linear convergence meaning that the RAS converges sublinear for matrices with support that is not total (\cite{ach93}). We have the following best bounds:

\begin{thm}[\cite{kni08}, Theorem 4.5] \label{thm:convergencespeed}
Let $A\in\mathbb{R}^{n\times n}$ be a fully indecomposable matrix and denote by $D_1,D_2$ the diagonal matrices such that $D_1AD_2$ is doubly stochastic. Let $D_1^i$ and $D_2^i$ be the diagonal matrices after the $k$-th step of the Sinkhorn iteration, there exists a $K\in \mathbb{N}$ such that for all $k\geq K$, in an appropriate matrix norm
\begin{align}
\|D_1^{i+1}\oplus D_2^{i+1}-D_1\oplus D_2\|\leq \sigma_2^2\|D_1^{i}\oplus D_2^{i}-D_1\oplus D_2\|
\end{align}
where $\sigma_2$ is the second largest singular value of $D_1AD_2$.
\end{thm}
The proof of this theorem crucially relies on the fact that matrices with a doubly stochastic pattern are direct sums of primitive matrices (modulo row and column permutations). Hence it cannot easily be extended to matrices with arbitrary row and column sums if those matrix patterns allow for non-primitive matrices. The approach in \cite{fra89} for positive matrices can also be extended to arbitrary row and column sums.

We observe that the occurrence of the second singular value should not come as a big surprise: Given a stochastic matrix $A$, $A^k$ converges to a fixed matrix and the convergence is dominated also by the gap between the largest singular value $1$ and the second largest singular value of $A$.

For practical purposes, one then needs to work out how many operations are needed to obtain a given accuracy of the solution. The first such bounds can be derived from the bounds in \cite{fra89}. The main study of these questions was conducted in the early 90s and 2000s, starting with \cite{kal93}. Let $A\in \R^{n\times n \times \ldots \times n}$ with $d$ copies of $\R^n$ be a positive multidimensional matrix which can be scaled to doubly stochastic form, then they proved that the $RAS$ takes at most
\begin{align}
	\mathcal{O}\left(\left(\frac{1}{\varepsilon}+\frac{\ln(n)}{\sqrt{d}}\right)d^{3/2}\sqrt{n}\ln \frac{V}{\nu}\right)
\end{align}
steps, where all matrix entries are in the interval $(\nu,V]$ and the maximal error is upper-bounded by $\varepsilon$. (\cite{kal93}, Theorem 1). They also derive a bound for a randomised version of the RAS, where at each step, the direction of descent is selected randomly and once in a while, the whole error function is computed randomly. The expected runtime is then slightly lower. 

In the case of positive matrix scaling, \cite{kal08} give better bounds covering also the case of inequality constraints as in \cite{bal89a}. In particular, let $A\in \R^{n\times m}$ be a positive matrix with $\nu\leq A_{ij}\leq V$, let $N=\max\{n,m\}$, let $\rho=\max\{r_i,c_j\}$ and $h=\sum_{ij} A_{ij}$. Then the number of iterations needed to scale $A$ to accuracy $\varepsilon$ is of order
\begin{align*}
	\mathcal{O}\left(\left(\frac{1}{\varepsilon}+\ln(hN)\right)\rho\sqrt{N}\left(\ln(\rho)+\ln\left(\frac{V}{\nu}\right)\right)\right).
\end{align*}
The two results (specific bounds and asymptotic linear behaviour for convergence speed) imply that the RAS method has generally good convergence properties if the matrix is positive. 

A fully polynomial time algorithm (i.e.~without a factor involving the size of the matrix entries) for general marginals was given in \cite{lin00a} based on the RAS method with preprocessing. However, the algorithm scales with $\mathcal{O}(n^7\log(1/\varepsilon))$ for the general $(r,c)$-scaling and (using a different algorithm closer to the RAS) with $\mathcal{O}((n/\varepsilon)^2)$ for doubly-stochastic scaling.

In summary, the RAS method, while not fully polynomial by itself, can be tweaked in various ways to allow for fully polynomial algorithms. In addition, it has the advantage of being parallelisable as demonstrated in \cite{zen90a,zen90b}. However, the scaling behaviour is not particularly fast in specific examples (see for instance \cite{bal04,kni12}), in particular it doesn't seem to be very good at handling sparse matrices.

\subsection{Newton methods}
One of the first alternative algorithms to the RAS methods was provided in \cite{mar68} as a minimisation of $x^{T}Ay$ using a modified Newton method as described in \cite{gol66} (it is not related how the equality constraints are introduced into the problem. This can be done using a $C^2$-penalty function). 

Newton methods were also developed to solve the scaling problem for positive semidefinite matrices. They can either be seen as Newton's method applied to $x^TAx$ for symmetric $A$ (cf. \cite{kha92}) or as Newton's method applied to the Sinkhorn iteration equation $x_{k+1}:=e/(Ax_k)$ (cf. \cite{kni12}). Yet a different method was considered in \cite{fur04}. 

\cite{kal05} shows that their algorithm converges in $\mathcal{O}(\sqrt{n}\ln(n/(\mu\varepsilon)))$ Newton iteration steps, where $\mu:=\inf\{x^TAx|x\geq 0\}$, if the matrix is scalable.

\subsection{Convex programming}
As noted in section \ref{sec:convexprog}, the convex programming formulation of the problem makes it amenable to a host of (polynomial time) techniques such as the ellipsoid method or interior point algorithms.

In the case of nonnegative matrices $A\in \R^{n\times n}$ with doubly stochastic marginals, a good bound was found in \cite{kal96a}, with operations of order 
\begin{align*}
	\mathcal{O}(n^4\ln (n/\varepsilon)\ln(1/\nu)). 
\end{align*}	
The bound uses ellipsoid methods. Later, the bounds were extended to cover generalised marginals in \cite{nem99} (also including the generalisation discussed in \cite{rot89b}) specifically using ellipsoid methods for the convex optimisation formulation of equation (\ref{eqn:gormanconvex}). 
The first instance of an interior point algorithm was probably formulated in \cite{bal04} applied to the entropy formulation. The authors find a strongly polynomial algorithm which scales better than \cite{lin00a} with $\mathcal{O}(n^6\log(n/\varepsilon))$\footnote{It seems that the authors were unaware of Kalantari et al.~and \cite{nem99}.}.

A different ansatz for an algorithm was used in \cite{sch90a}, where the author uses the duality in convex programming and a coordinate ascent algorithm for the dual problem of his truncated matrix scaling. This algorithm will then be some form of generalisation of the RAS method.

\subsection{Other ideas}
We give a short primer of other algorithms considered in the literature:
\begin{enumerate}
	\item The first paper to develop new algorithms with a focus on speed and not only concepts was \cite{par82}, where a bunch of slightly different and optimised algorithms is derived.
 \item An algorithm which is somewhat related to convex algorithms is considered in \cite{kal96b}. It is a total gradient based, steepest descent algorithm for the homogeneous log-barrier potential. 
	\item In \cite{rot07}, using an algorithm for the matrix apportionment problem involving network flows and using ideas of \cite{kar97}, they provide an algorithm where the number of iterations scales with 
\begin{align}
	\mathcal{O}(n^3\log n(\log(1/\varepsilon)+\log(n)+\log\log(V/\nu))).
\end{align}
Once again, $A_{ij}\in [\nu,V]$ for all $i,j$.
	\item With ever larger matrices, it is sometimes infeasible to access each element of the matrix on its own, because the matrix is not stored in that form or processed somewhere else. This makes it interesting to consider algorithms that do not need access to all elements, such as the RAS method for nonnegative matrices. Algorithms that are ``matrix free'' in this sense were developed in \cite{bra10,bra11} for doubly stochastic scaling of positive semidefinte matrices.
	\item Finally, let us mention that algorithms were also developed for infinity norm scaling (cf. \cite{bun71,rui01,kni14}) and other norm scaling (cf. \cite{rui01}).
\end{enumerate}

\subsection{Comparison of the algorithms}
A first comparison of several algorithms was performed in \cite{sch90b}, however, the comparison is not really in terms of speed (for instance, all algorithms were implemented on different programming platforms), but in terms of useability. 

While there have been many papers claiming superior convergence speed for their algorithm, the most comprehensive analysis has probably been achieved in \cite{kni12}, which is limited to doubly-stochastic scalings. In the paper, the authors compare the RAS method, a Gauss-Seidel implementation of the ideas of \cite{liv04}, and the fastest algorithm in \cite{par82} with their own Newton-method algorithm. The test matrices are mostly large sparse matrices and the new algorithm is usually the fastest and most robust algorithm. The authors also claim that their Newton-based implementation is superior to \cite{kha92} and \cite{fur04}. They also suggest that the algorithm should outperform the convex optimisation based algorithms, albeit a direct comparison to the most recent algorithm in \cite{bal04} is missing, who only showed that their algorithm clearly outperforms the RAS method. \cite{bra11} also mention that their purely matrix free algorithm will outperform explicit methods such as those in \cite{kni12} if accessing single elements in the matrix is actually slow. 

In general, matrix scaling can today be done on a routine basis even for very large matrices.

\section{Applications of Sinkhorn's theorem} \label{sec:applic}
The following problem can be encountered in many areas of applied mathematics (see also \cite{sch90b}):
\begin{pro} \label{pro:prob1}
Let $A\in\R^{m\times n}$ be a nonnegative matrix. Find a matrix $B$ which is close to $A$ and which fulfills a set of linear inequalities, for instance 
\begin{align*}
	\sum_{j=1}^n A_{ij}=r_i, \qquad \sum_{i=1}^m A_{ij}=c_j.
\end{align*}
\end{pro}
We could also ask for balanced marginals or any other type of marginals. The problem is certainly not well-posed. What does ``close'' mean? This part of the review will try to give an overview why ``close'' means equivalence scaling in many applications. We will limit our attention mostly to the mathematical justification of matrix scalings, but I will try to give pointers to other literature. 

This implies that we only consider ``nearness'' leading to equivalence scaling or matrix balancing as the result. In the literature, other nearest matrices have also been considered such as addition of small matrices (e.g.~\cite{bac80}). \cite{sch90b} describe network flow algorithms that allow for a wider variety of applications.

Matrix scaling has many different real world applications, which implies that it also needs different justifications. While statistical justifications exist, many applications argue with the simplicity of the method and the fact that it performs well in practice. These are valid arguments, but they are unsatisfactory from a mathematical point of view. In the two following subsections we collect mathematically rigorous (or partly rigorous) justifications and their history.

\subsection{Statistical justifications}
As seen, matrix scaling solves entropy minimisation with marginal constraints. This is one of the most powerful entries for justifications of equivalence scaling as the right model, since relative entropy has strong statistical justifications, mostly in the form of maximum entropy or minimum discrimination information - see \cite{jay57} or later in \cite{kul66} for a justification in physics, or \cite{kul59} and \cite{gok78} for a justification in statistics.

Another justification closely connected to entropy minimisation is \emph{maximum likelihood} models. For instance, if given a set of distributions $Q$ and an empirical i.i.d. sample $P$, then the maximum likelihood for $P$ being a sample of $Q$ is given by the minimal relative entropy (\cite{csi89,dar72}). Max-Likelihood justifications for applications in contingency tables are given in \cite{fie70,goo63}.

A different class of justifications for the validity of the matrix scaling approach are arguments showing that matrix scaling conserves certain form of interactions within the matrix. For instance, matrix scaling conserves cross products (\cite{mos68}) and so-called $k$-cycles (\cite{ber79}, $k$-cycles are certain products of matrix and inverse matrix entries). Both can be desirable for modeling reasons. 

Finally, let us mention that the original justification (matrix scaling is a least-square type optimisation) made in \cite{dem40} turned out to be wrong very quickly and was superseded by real least-square methods in \cite{ste42} and later in \cite{fri61} or \cite{car81}. Those however are not the same as equivalence scaling (see Section \ref{sec:rowprod}).

\subsection{Axiomatic justification} \label{sec:justbalinski}
Another justification for matrix scaling, which is particularly useful for application in elections is given in \cite{bal89b}. Instead of considering just any matrix ``close'' to the original estimate, we want this matrix to fulfil a set of axioms. 

Let $A$ be a nonnegative matrix, $r_+,c_+$ ($r_-,c_-$) be upper (lower) bounds to the row and column sums and $h>0$ be a scalar. As in Section \ref{sec:inequalityscaling}, we denote the set of all matrices $B$ fulfiling the bounds $q:=(r_-,r_+,c_-,c_+,h)$ with $h=\sum_{ij} B_{ij}$ by $R(q)$. For any matrix $A$ and any set of bounds $q$, we search for a method $F(A,q)$ to allocate one out of potentially many matrices $A^{\prime}$ fulfiling $q$ and the following axioms:

\begin{itemize}[leftmargin=1.5cm]
	\item[Axiom 1] \textit{Excactness:} If $r_-=c_-=0$ and $r_+=c_+=\infty$ then $A^{\prime}=(h/\sum_{ij}A_{ij}) A$
	\item[Axiom 2] \textit{Relevance:} If $q^{\prime}$ is another set of bounds such that $R(q^{\prime})\subset R(q)$ and there exists a possible $A^{\prime}\in R(q^{\prime})$, then $F(A,q^{\prime})\subset F(A,q)\cap R(q^{\prime})$.
	\item[Axiom 3] \textit{Uniformity:} For any matrix $A^{\prime}$ with bounds $q$, if we construct a new matrix $A^{\prime,\prime}$ by exchanging any submatrix $A^{\prime}_{I\times J}$ by another submatrix $B_{I\times J}$ which fulfils the same row and column sums minus the part of these bound allocated in $A^{\prime}_{(I\times J)^c}$, then $A^{\prime \prime}\in F(A,q)$.
	\item[Axiom 4] \textit{Monotonicity:} If we have two matrices $A,B$ with $A_{ij}\leq B_{ij}$ for all $(i,j)$, then it also holds that $A^{\prime}_{ij}\leq B^{\prime}_{ij}$ for all possible allocations.
	\item[Axiom 5] \textit{Homogeneity:} Suppose $r_-=r_+$ and $c_-=c_+$. Then, if two rows of $A$ are proportional and are constrained to the same row sum, then the corresponding rows in $A^{\prime}$ are always equal.
\end{itemize}

Then \cite{bal89b} show that equivalence scaling (the fair share matrix of Section \ref{sec:inequalityscaling}) is the unique allocation method $F(A,q)$ for all nonnegative matrices $A$ where $R(q)$ contains a matrix with the same pattern as $A$. 

\subsection{A primer on applications}
We will only sketch applications here since a complete list and discussion is probably infeasible. 

\paragraph*{Transportation planning}
A natural problem in geography is connected to predicting flows in a traffic network. If one considers for example a network of streets in a city at rush hour and a number of workers that want to get home, it is important to know how the traffic will be routed through the network. This is to a large degree a problem of physical modeling and a number of methods have been developed in the last century (for a recent introduction and overview see \cite{ort11}\footnote{The authors also discuss the RAS method in chapter 5. I am however not convinced by their claim that Bregman provided the best analysis of the mathematics of the problem.}). 

For our purposes, the most interesting question results from estimating trip distribution patterns from prior or incomplete data. In a simplified model, one could consider only origin and destination nodes (e.g.~home quarters and work areas), given by a nonnegative matrix $A$. While the matrix is known for one year, it might be necessary to predict the changes given that the amount of trips to and from one destination change. 

Several papers have treated a justification of the RAS method in this case. For instance, \cite{eva70} argues that the method provides a unique outcome and it is easier to handle and to compute than other methods (Detroit method, growth factor method,...). A discussion of trip distribution with respect to Problem \ref{pro:prob1} can be found in \cite{sch90b}.

\paragraph*{Contingency table analysis}
In many situations ranging from biology to economics, contingency tables need to be estimated from sample data. Contingency tables list the frequency distributions of events in surveys, experiments, etc. They are highly useful to map several variables and study their relations.

As a specific example, suppose a small census in Germany tries to estimate migration between the states. While the number of citizens is recorded, which means that the total net migration is known, it is not known where each individual migrant came from. From a small survey among migrants, how can one estimate the true table with correct marginals in the best possible way? If one does a  maximum likelihood estimation, the result is once again matrix scaling (cf. \cite{fie70,pla82}).

\paragraph*{Social accounting matrices}
Social accounting matrices, or SAMs, are an old tool developed in \cite{sto62} and later popularised in \cite{pya76} to represent the national account of a country. To date, it is an important aspect of national and international accounting (as a random example see \cite{klo04} from the German national institute of statistics. An introduction to social accounting can also be found in \cite{pya85} and, from a short mathematical perspective, in \cite{sch90b}). 

The idea is to represent income and outcome of a national economy in a matrix. Often, good growth estimates are known for the row and column sums and certain estimates are known for individual cells. The account estimates are then often not balanced, which can be achieved using matrix balancing or matrix scaling. Justifications can be imported from statistics, most notably maximum likelihood.

\paragraph*{Schr\"odinger bridges}
In \cite{sch31}, the author considered the following setup: Suppose we have a Brownian motion and a model which we are very confident about. In an experiment we observe its density at two times $t_0, t_1$. Now suppose they differ significantly from the model predictions. How can we reconcile these observations by updating our model without discarding it completely? 

This problem has been studied in a whole line of papers since then from \cite{for41} to \cite{geo15}. The minimum relative entropy approach can be justified using large deviations (see \cite{rus95}).

\paragraph*{Decreasing condition numbers}
Given a system of linear equations $Ax=b$ with nonsingular $A$, solving it relies on the Gaussian elimination procedure, which is known to be numerically unstable for matrices with bad condition number $\kappa(A):=\|A\|_{\infty}\|A^{-1}\|_{\infty}$. In order to increase the stability, we have to modify $A$, for example by multiplying with diagonal matrices $D_1,D_2$ and considering $D_1AD_2$. Given that linear systems are ubiquitous in numerical analysis, it is of paramount importance to know how best to precondition a matrix in order to minimise calculation errors (see for instance the survey \cite{ben02}). The answer to this question is problem dependent. Particular problems, where equivalence scaling is helpful to go include integral controllability tests based on steady-state information and the selection of sensors and actuators using dynamic information (see \cite{bra94}). 

One of the first papers to consider minimisation of $\kappa$ using diagonal scaling was \cite{osb60}, who focused on matrix balancing instead of equivalence scaling (see also \cite{liv04} and \cite{che00} for sparse matrices). Since the condition number contains the maximum norm, it might be best to require balanced maximum rows and columns instead of balanced row sums as observed in \cite{bau63,cur71}. This works particularly well for sparse matrices. Equivalence scaling has been studied as early as \cite{hou06}\footnote{Reference from \cite{bra94}} and later in \cite{ols96}. If we use other $p$-norms in the definition of $\kappa$, a convex programming solution for minimising $\kappa$ using equivalence scaling is provided in \cite{bra94}.

Note that unlike in all applications studied so far, preconditioning a matrix is useful not only for nonnegative matrices. This is one reason why matrix balancing was studied for copositive and not simply nonnegative matrices. A very different measure of the ``goodness'' of scaling which might also be of numerical relevance was studied in \cite{rot80}, where the authors solved the problem of matrix balancing of a matrix such that the ratio between the biggest and smallest element of the scaled matrix becomes minimial. 

\paragraph*{Elections}
A very important application of equivalence scaling can be found in Voting: Given election results in a federal election, how can one best distribute the seats among the parties within the states such that each party and each state is represented according to the outcome of the election? Note that here, we need to adjust for natural numbers, which requires rounding (cf.~\cite{mai10}). 

Early methods based on a discretised RAS method were developed in \cite{bal89a}. The problem is very intricate in itself, because the justifications rely on what is perceived as ``fair'' and any method that is fair in some instances is unfair in others (see for instance the discussions in \cite{bal97,puk04}; for an overview, see \cite{nie08}).

\paragraph*{Other applications}
Various other applications of equivalence scaling and matrix balancing exist such as:
\begin{enumerate}
	\item A Sudoku Solver based on a stochastic algorithm based on the RAS was developed in \cite{moo09}. 
	\item An algorithm to rank web-pages was developed in \cite{kni08}. The RAS allows to derive an algorithm similar in scope to the HITS algorithm (\cite{kle99}).
	\item The RAS method is analysed as a relaxed clustering algorithm in data mining (\cite{wan10}). However, it turns out that methods based on other Bregman-divergences are more favourable. 
	\item Given a (discretised) quantum mechanical time evolution, can we construct a local hidden variable model of its evolution corresponding to a deterministic stochastic transition matrix (\cite{aar05})? 
	\item Given a Markov chain with a doubly stochastic transition matrix and given an estimate of the transition matrix, the best estimate of the real transition matrix is given by a scaled matrix (\cite{sin64}).	
	\item Regularising optimal transportation by an entropy penalty term such that it can be computed using the RAS, which is already much faster than optimal transport algorithms (\cite{cut13}).
\end{enumerate}

\section{Scalings for positive maps}
We have already seen that Sinkhorn scaling is interesting for classical Schr{\"o}dinger bridges as well as for scaling transition maps of Markov processes, etc. From a physics perspective, all these applications are classical physics, transforming classical states (probability distributions) to classical states. 

In quantum mechanics, the basic objects are \emph{quantum states}. For finite dimensional systems (such as spin systems), these quantum states are positive semidefinite matrices with unit trace. A \emph{quantum operation} then maps states to states, i.e.~it needs to be positive: If $A\geq 0$, then $\mathcal{T}(A)\geq 0$. In fact, this is not all that is required for quantum operations, but one actually needs $\mathcal{T}$ to be \emph{completely positive} (for an overview about quantum operations and quantum channels, see \cite{nie00,wol12}). (Completely) Positive trace-preserving maps are then the natural generalisation of stochastic matrices. This raises the question whether concepts as irreducibility and a Perron-Frobenius theorem exist also for quantum channels and indeed they do. A Perron-Frobenius analogue was probably first described in \cite{sch00}, while the analogue for full indecomposability was first used in \cite{gur04}. 

Let us define the concepts:
\begin{dfn} \label{dfn:irredindec}
A positive map $\mathcal{E}:\mathcal{M}_d\to \mathcal{M}_d$ with $\mathcal{M}_d=\C^{d\times d}$ is called \emph{irreducible} (as in \cite{eva78, far96}) if for any nonzero orthogonal projection $P$ such that
\begin{align}
	\mathcal{E}(P\mathcal{M}_dP)\subseteq P\mathcal{M}_dP
\end{align}
we have $P=\id$. 

Likewise, it is called \emph{fully indecomposable} if for any two nonzero orthogonal projections $P,Q$ with the same rank such that
\begin{align}
	\mathcal{E}(P\mathcal{M}_dP)\subseteq Q\mathcal{M}_dQ
\end{align}
we have $P=Q=\id$. 

Finally, a map is called \emph{positivity improving} (the analogue to positive matrices) if for all $A\geq 0$, $\mathcal{E}(A)>0$.
\end{dfn}
A lot of different characterisations have been found (see Appendix \ref{app:posprel}).

Furthermore, let us define:
\begin{dfn} \label{dfn:nonincr}
Let $\mathcal{E}:\mathcal{M}_d\to \mathcal{M}_d$ be a positive map. Then $\mathcal{E}$ is called \emph{rank non-decreasing} if for all $A\geq 0$
\begin{align}
	\operatorname{rank}(\mathcal{E}(A))\geq \operatorname{rank}(A). \label{eqn:geqsign}
\end{align}
It is called \emph{rank increasing}, if the $\geq$ sign in equation (\ref{eqn:geqsign}) is replaced by a $>$.
\end{dfn}
The connections of Definitions \ref{dfn:irredindec} and \ref{dfn:nonincr} are explained in Appendix \ref{app:posprel}.

Let us now define what we mean by scaling a positive map:
\begin{dfn}
Let $\mathcal{E}:\mathcal{M}_n\to\mathcal{M}_n$ be a positive, linear map. We say that $\mathcal{E}$ is \emph{scalable} to a doubly stochastic map, if there exist $X,~Y\in\mathcal{M}_d$ such that
\begin{align}
	\mathcal{E}^{\prime}(\cdot):=Y^{\dagger}\mathcal{E}(X\cdot X^{\dagger})Y \label{eqn:scaling}
\end{align}
is doubly stochastic (i.e.~$\mathcal{E}^{\prime}(\id)=\mathcal{E}^{\prime *}(\id)=\id$). \\
We call a positive map \emph{$\varepsilon$-doubly stochastic}, if 
\begin{align}
	\operatorname{DS}(\mathcal{E}):=\tr((\mathcal{E}(\id)-\id)^2)+\tr((\mathcal{E}^{*}(\id)-\id)^2)\leq \varepsilon^2
\end{align}
We call $\mathcal{E}$ \emph{$\varepsilon$-scalable} if there exists a scaling as in equation (\ref{eqn:scaling}) to an $\varepsilon$-doubly-stochastic map $\mathcal{E}^{\prime}$. 
\end{dfn}
The error function $\operatorname{DS}$, which is similar to an $L^2$-error function for matrices, will serve twofold: first, it defines approximate scalability (which can alternatively be defined by convergence of the RAS) and second, it defines a progress measure for convergence similar to error functions as considered in \cite{bal89b}.

We can now state the full analogue of equivalence scaling to doubly stochastic form:
\begin{thm} \label{thm:sinkqm}
Given a positive map $\mathcal{E}:\mathcal{M}_d\to \mathcal{M}_d$, it is scalable to a doubly stochastic map iff there exist some matrices $X$ and $Y$ such that $Y\mathcal{E}(X\cdot X^{\dagger})Y^{\dagger}$ is a direct sum of fully indecomposable maps.

The scaling matrices are unique iff $\mathcal{E}$ is fully indecomposable.
\end{thm}
The fact that fully indecomposable matrices are uniquely scalable was first proved in \cite{gur03}. His work built on earlier work in \cite{gur02a,gur02b} (see also \cite{gur04}), based on a generalisation of the convex approach in equation (\ref{eqn:gormanconvex}) and the London-Djokovic approach in equation (\ref{eqn:london}). 

Recently, the problem was considered with the hope to apply it for unital quantum channels in \cite{ide13} (which has never been formally published) and shortly afterwards in \cite{geo15} while trying to define and study ``quantum'' Schr\"odinger bridges. Both approaches use nonlinear Perron-Frobenius theory and thereby a generalisation of equation (\ref{eqn:menon}) to get a result. The approaches derived from classical results are discussed in Section \ref{sec:classicalquantum}.

Even earlier than Gurvits, a very limited version of the theorem was proven in \cite{ken99} (with subsequent generalisations) using an approach that does not derive from any of the classical approaches but instead uses the Choi-Jamiolkowski isomorphism. This is described in Section \ref{sec:statechannel}.

The extension of the theorem to necessary and sufficient conditions has as far as I know not been formally published\footnote{Gurvits actually claims a proof for the fact that a positive map is uniquely scalable iff it is fully indecomposable, but I did not understand how the only if part follows.}.

Furthermore, we can state an analogue of approximate scaling, which to date has only been considered in \cite{gur04}:
\begin{thm}\label{thm:approxsinkquant}
Let $\mathcal{E}:\mathcal{M}_n\to\mathcal{M}_n$ be a positive, linear map. Then $\mathcal{E}$ is approximately scalable (i.e.~$\varepsilon$-scalable for any $\varepsilon>0$) if and only if $\mathcal{E}$ is rank non-increasing.
\end{thm}
An overview about different approaches and how they derive from existing approaches can be found in Figure \ref{fig:overviewpositive}.

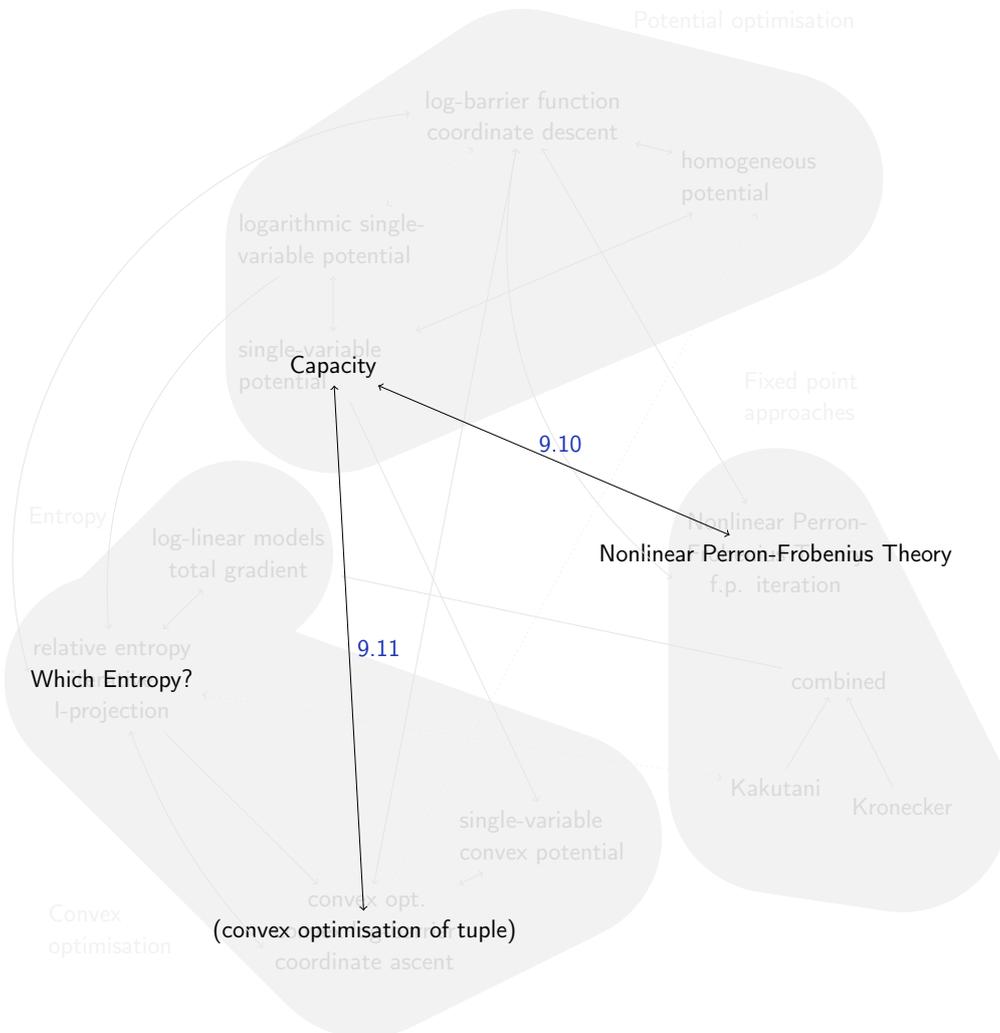
\begin{figure}[htbp] \label{fig:overviewpositive}
\resizebox{.95\textwidth}{!}{%
\begin{tikzpicture}[font=\sffamily]
  \node[black!15, text width=3.3cm,align=center] at (6.5,13) (A) {log-barrier function \\ coordinate descent};
  \node[black!15, text width=3cm] at (3.5,11) (B) {logarithmic single-variable potential};
  \node[black!15, text width=3cm] at (3.5,9) (C) {single-variable potential};
	\node at (3.5,9) (Cnew) {Capacity};
  \node[black!15, text width=3cm] at (10.5,12) (D) {homogeneous \\ potential};
	
	\node[black!15, text width=2.6cm,align=center] at (0,4) (E) {relative entropy \\ iterative I-projection};
	\node[black!15] at (-1.3,4) (Einv) {};
	\node at (0,4) (Enew) {Which Entropy?};
	\node[black!15, text width=3cm,align=center] at (2,6)(L) {log-linear models \\ total gradient};
	
	\node[black!15, text width=3cm,align=center] at (10.5,6) (F) {Nonlinear Perron-Frobenius Theory \\  f.p. iteration};
	\node[black!15] at (9,5.5) (Finv) {};
	\node at (10.5,6) (Fnew) {Nonlinear Perron-Frobenius Theory};
	\node[black!15] at (11.5,4) (G) {combined};
	\node[black!15] at (10.5,2.3) (H) {Kakutani};
	\node[black!15] at (12.5,2) (I) {Kronecker};
	
	\node[black!15, text width=3cm,align=center] at (4,0) (J) {convex opt. \\convex log-barrier \\ coordinate ascent};
	\node at (4,0) (Jnew) {(convex optimisation of tuple)};
	\node[black!15] at (2.5,-0.4) (Jinv) {};
	\node[black!15, text width=3cm] at (7,1.5) (K) {single-variable convex potential};
	
  \begin{pgfonlayer}{background}
    \foreach \nodename in {A,B,C,D,E,F,G,H,I,J,K,L} {
      \coordinate (\nodename') at (\nodename);
    }
		\node[color=black!05] at (10,14.5) (nlino) {Potential optimisation};
		\node[color=black!05,text width=2cm] at (0,0) (conv) {Convex \\ optimisation};
		\node[color=black!05,text width=3cm] at (11.5,8.5) (top) {Fixed point \\ approaches};
		\node[color=black!05] at (-0.7,6.6) (entr) {Entropy};
    \path[fill=black!05,draw=black!05,line width=3.4cm, line cap=round, line join=round] 
    (A') to (B') 
         to (C') 
         to (D') 
         to (A') -- cycle;
		\path[fill=black!05,draw=black!05,line width=3.4cm, line cap=round, line join=round]
		(E') to (J')
		     to (K')
				 to (E') -- cycle;
		\path[fill=black!05,draw=black!05,line width=3.4cm, line cap=round, line join=round]
		(F') to (G')
		     to (I')
				 to (H')
				 to (F') -- cycle;
		\path[fill=black!05,draw=black!05,line width=3cm, line cap=round, line join=round]
		(E') to (L')
				 to (E') -- cycle;
  \end{pgfonlayer}

	\path[<->] (A) edge[black!10] (F);
	\path[<->,dotted] (A) edge[black!10] (B);
	\path[<->] (B) edge[black!10] (C);
	\path[<->] (C) edge[black!10] (D);
	\path[<->] (D) edge[black!10] (A);
	\path[<->] (A) edge[black!10] (J);	
	\path[->] (H) edge[black!10] (G);
	\path[->] (I) edge[black!10] (G);
	\path[<->,dotted] (H) edge[black!10] (E);
	\path[->] (E) edge[black!10] (J);
	\path[<->] (J) edge[black!10] (K);
	\path[->] (C) edge[black!10] (K);
	\path[->] (B) edge[bend right,black!10] (E);
	\path[<->] (E) edge[black!10] (L);
	\path[<-] (L) edge[black!10] (G);
	
	\path[<->] (A) edge[bend right=30,black!10] node[right]{} (Finv); 
	\path[<->] (A) edge[bend right=50,black!10] node[right]{} (Einv);
	\path[<->] (E) edge[bend right=10,black!10] node[right]{} (Jinv);
	\path[->,dotted] (J) edge[black!10] (D);
	
	\path[<->] (Cnew) edge node[above]{\hspace{.2cm}\ref{lem:gurscaleequiv}} (Fnew);
	\path[<->] (Cnew) edge node[right]{\ref{lem:convexgurvits}} (Jnew);
\end{tikzpicture}
}
\caption{Approaches to positive map scalings and their connections. The classical approaches of Figure \ref{fig:overview} are depicted in grey, while positive map approaches derived from classical approaches are overlayed in black.}
\end{figure}

\subsection{Operator Sinkhorn theorem from classical approaches} \label{sec:classicalquantum}
We will now study how the theorems above were derived extending classical approaches to positive maps starting with an analogue of the RAS method for positive maps:
\begin{alg} \label{alg:rasqm}
Let $\mathcal{E}:\mathcal{M}_n\to\mathcal{M}_n$ be a positive, linear map.
\begin{enumerate}
	\item Start with $\mathcal{E}_0:=\mathcal{E}$.
	\item For each $i=0,\ldots,n$ define:
		\begin{align}
	\mathcal{E}_{2i+1}(\cdot)&:=\mathcal{E}_{2i}(\id)^{1/2}\mathcal{E}_{2i}(\cdot)\mathcal{E}_{2i}(\id)^{1/2} \label{eqn:iterate1}\\
	\mathcal{E}_{2i+2}(\cdot)& :=\mathcal{E}_{2i+1}(\mathcal{E}_{2i+1}^{*}(\id)^{1/2}\cdot\mathcal{E}_{2i+1}^{*}(\id)^{1/2}) \label{eqn:iterate2}
\end{align}
	\item Iterate till convergence
\end{enumerate}
By construction, we iterate between trace-preserving (even) and unital (odd) maps.
\end{alg}

\subsubsection{Potential Theory and Convex programming}
As stated, an approach along the lines of the London-Djokovic approach of equation (\ref{eqn:london}) is found in \cite{gur03, gur04} (with methods of \cite{gur02b,gur00,gur02a}). Since the complete proofs are lengthy and scattered over several papers, we provide full proofs in Appendix \ref{app:gurvitsproof} for the benefit of the reader. In this section, we only sketch the path of the proofs.

Recall that a matrix scaling exists iff the following is positive and the minimum is attained:
\begin{align}
	c(A):=\inf\left\{\prod_{i=1}^n \sum_{j=1}^n A_{ij}x_j \middle|\prod_{i=1}^n x_i=1\right\}
\end{align}
Exchanging products with determinants and sums with traces, we obtain the following definition:
\begin{dfn}
Let $\mathcal{E}:\mathcal{M}_n\to\mathcal{M}_n$ be a positive, linear map. Then define the \emph{capacity} via
\begin{align}
	\operatorname{Cap}(\mathcal{E}):=\inf \{\detm(\mathcal{E}(X))|X>0, \detm(X)=1\}
\end{align}
\end{dfn}
We will start with covering approximate scaling:

\paragraph*{Approximate scalability}
The capacity is the right functional to study scaling:
\begin{lem}[\cite{gur04}] \label{lem:approxscaling}
Let $\mathcal{E}:\mathcal{M}_d\to\mathcal{M}_d$ be a positive map. If $\operatorname{Cap}(\mathcal{E})>0$ then the RAS method of Algorithm \ref{alg:rasqm} converges and $\mathcal{E}$ is $\varepsilon$-scalable for any $\varepsilon>0$.
\end{lem}
The proof of this lemma uses the following observation: For any $C_1,C_2>0$ we have
\begin{align*}
	\operatorname{Cap}(C_1\mathcal{E}(C_2^{\dagger}\cdot C_2)C_1^{\dagger})=\det(C_1C_1^{\dagger})\det(C_2C_2^{\dagger})\operatorname{Cap}(\mathcal{E}).
\end{align*}
Then, a quick calculation shows that Algorithm \ref{alg:rasqm} only decreases $\operatorname{Cap}$ using this equality. If $\operatorname{Cap}(\mathcal{E})\neq 0$, one can then show that $\operatorname{DS}(\mathcal{E}_i)\to 0$ for $i\to \infty$.

Next, we need to see when the capacity is actually positive. To do this, for every unitary $U$ we need to define the tuple
\begin{align}
	\mathbf{A}_{\mathcal{E},U}:=(\mathcal{E}(u_1u_1^{\dagger}),\ldots,\mathcal{E}(u_nu_n^{\dagger})),
\end{align}
where $u_i$ is the $i$-th column of $U$. This is done to connect the capacity with so called \emph{mixed discriminants} (see also \ref{app:posprel}), which are needed for the proof. In fact, we have:
\begin{lem} \label{lem:tuplecap}
Let $\mathcal{E}:\mathcal{M}_n\to\mathcal{M}_n$ be a positive, linear map and $U\in U(n)$ a fixed unitary. Then defining
\begin{align*}
	\operatorname{Cap}(\mathbf{A}_{\mathcal{E},U}):=\inf \left\{\detm \left(\sum_i \mathcal{E}(u_iu_i^{\dagger})\gamma_i\right) | \gamma_i>0, \prod_{i=1}^n\gamma_i=1\right\}
\end{align*}
where $u_i$ are once again the rows of $U$, we have the following properties:
\begin{enumerate}
	\item Using the mixed discriminant $M$ defined in Appendix \ref{app:posprel}, we have
	\begin{align*}
		M(\mathbf{A}_{\mathcal{E},U})\leq \operatorname{Cap}(\mathbf{A}_{\mathcal{E},U})\leq \frac{n^n}{n!}M(\mathbf{A}_{\mathcal{E},U})
	\end{align*}
	\item $\inf\limits_{U\in U(n)} \operatorname{Cap}(\mathbf{A}_{\mathcal{E},U})=\operatorname{Cap}(\mathcal{E})$
\end{enumerate}
\end{lem}
Most of the proof is very technical and found in \cite{gur02a}. Some parts are explained in Appendix \ref{app:gurvitsproof}.

Finally, this proves most of Theorem \ref{thm:approxsinkquant}: We know that $\mathcal{E}$ is rank non-decreasing if and only if $\operatorname{Cap}(\mathcal{E})$ is positive. In that case, the RAS algorithm converges in which case the map is approximately scalable. For the other direction, one can use a simple contradiction argument: Any map close to a doubly stochastic map must be rank non-decreasing and as scaling does not change this property of a map, any approximately scalable map must be rank non-decreasing. 

\paragraph*{Exact scalability}
The capacity is also the correct generalisation for exact scaling:
\begin{lem}[\cite{gur04}] \label{lem:gurscaleequiv}
Let $\mathcal{E}:\mathcal{M}_d\to\mathcal{M}_d$ be a positive map. Then $\mathcal{E}$ is scalable to a doubly-stochastic map if and only if $\operatorname{Cap}(\mathcal{E})>0$ and the capacity can be achieved.
\end{lem}

The proof of this Lemma following \cite{gur04} is given in Appendix \ref{app:gurvitsproof}. The direction ``Capacity is achieved $\Rightarrow$ the map is scalable'', is proved by taking the Lagrangian and showing that at the minimum we have that
\begin{align*}
	\nabla \ln(\det(\mathcal{E}(X)))=\mathcal{E}^*(\mathcal{E}(C)^{-1})
\end{align*}
which implies scalability by Lemma \ref{lem:sinkhornderivative}. The converse direction is given by a direct calculation.

In order to prove that a map can be scaled to doubly stochastic form, one then needs to connect this lemma to full indecomposability of matrices. The proof is done using an argument involving strict convexity. Like the original London-Djokovic potential (\ref{eqn:london}), the capacity is not a convex function, but one can make a substitution similar to Formulation (\ref{eqn:gormanconvex}) by considering the following function for any tuple $(A_i)_i$ of positive definite matrices:
\begin{align}
	f_A(\xi_1,\ldots,\xi_n):=\ln \detm(e^{\xi_1}A_i+\ldots+e^{\xi_n}An).
\end{align}
Then we have: 

\begin{lem} \label{lem:convexgurvits}
Let $\mathcal{E}:\mathcal{M}_n\to\mathcal{M}_n$ be a positive, linear map and given $U\in U(n)$, let $A=\mathbf{A}_{\mathcal{E},U}$. Then 
\begin{enumerate}
	\item $f_A$ is convex on $\mathbb{R}^n$.
	\item If $\mathcal{E}$ is fully indecomposable, then $f_A$ is strictly convex on $\{\xi=(\xi_1,\ldots, \xi_n)\in\mathbb{R}^n|\sum_i \xi_i=0\}$.
\end{enumerate}
\end{lem}
The proof is technical and uses mixed discriminants as well as results about them from \cite{bap89b}, which is why we only discuss it in the appendices. This is then used to prove

\begin{lem}\label{lem:quantumgurvits}
Let $\mathcal{E}:\mathcal{M}_n\to\mathcal{M}_n$ be a positive, linear map. If $\mathcal{E}$ is fully indecomposable, there exists a unique scaling of $\mathcal{E}$ to a doubly stochastic map.
\end{lem}
The idea is somewhat similar to the approximate Sinkhorn theorem: Since fully indecomposable maps are in particular rank non-decreasing, we know that the capacity is positive. For any $X>0$, which is diagonalised by $U$, using the tuple $\mathbf{A}_{\mathcal{E},U}$, one can then see that $\det(\mathcal{E}(X))=f_A(\log \lambda)$ with the eigenvalues $\lambda$ of $X$. Showing that the infimum must lie inside a compact set then finishes the proof, since Lemma \ref{lem:convexgurvits} implies existence and uniqueness of the minimum as $f_A$ is strictly convex.

Using Lemma \ref{lem:kernelred}, we can see then see that up to a unitary, every doubly stochastic map is a direct sum of fully indecomposable maps (much like doubly stochastic matrices are up to permutations a direct sum of fully indecomposable matrices, see Proposition \ref{prop:fullyindec}). Hence, any map which is a direct sum of fully indecomposable maps up to some scaling is clearly scalable to doubly stochastic maps. Sadly, the condition seems not very useful and the question remains open, whether one can simplify this condition.

However, that does in fact answer the question of unique scaling: Since the scaling for direct sums is not unique (we can always interchange summands), a map is uniquely scalable if and only if it is fully indecomposable.

\subsubsection{Nonlinear Perron-Frobenius theory}
The main idea of the alternative proofs in my Master's Thesis (\cite{ide13}) and the paper \cite{geo15} is to extend the Menon operator to positive semidefinite matrices:
\begin{dfn}
Let $\mathcal{E}:\mathcal{M}_n\to\mathcal{M}_n$ be a positive map, such that $\mathcal{E}(A),\mathcal{E}^*(A)>0$ for all $A>0$. Let $\mathcal{D}$ denote matrix inversion, then we define the following nonlinear map:
\begin{align*}
	&\mathbf{T}_{\mathrm{pos}}:\{A\in \mathcal{M}_n|A>0\}\to \{A\in \mathcal{M}_n|A>0\} \\
	&\mathbf{T}_{\mathrm{pos}}(\cdot)
		:=\mathcal{D}\circ\mathcal{E}^*\circ\mathcal{D}\circ\mathcal{E}(\cdot)
\end{align*}
This map is well-defined and after normalisation, it sends positive definite matrices of trace one onto itself.
\end{dfn}
We can then reformulate the existence problem into a fixed point problem:
\begin{lem} \label{lem:sinkhornderivative}
Let $\mathcal{E}:\mathcal{M}_n\to\mathcal{M}_n$ be a positive map such that $\mathcal{E}(A),\mathcal{E}^*(A)>0$ for all $A>0$. Then there exist invertible $X,Y\in\mathcal{M}_n$ such that $Y^{-1}\mathcal{E}(X(\cdot)X^{\dagger})Y^{-\dagger}$ is a doubly stochastic map if and only if $\mathbf{T}_{\mathrm{pos}}$ has an eigenvector (a fixed point after normalisation) in the set of positive definite trace one matrices. Furthermore, $X,Y$ can be chosen such that $X,Y>0$.
\end{lem}
\begin{proof}
Let $\rho>0$ be the positive definite eigenvector of $\mathcal{G}$. Then define $0<\sigma:=\mathcal{E}(\rho)$. Since $\rho$ is an eigenvector, one immediately sees that $\mathcal{E}^*(\sigma^{-1})=\lambda \rho^{-1}$ with $\lambda=\tr(\mathcal{E}^*(\sigma_{-1}))^{-1}$. Now define $X:=\sqrt{\rho}$ and $Y:=\sqrt{\sigma}$ (i.e.~$XX^{\dagger}=\rho$ and $YY^{\dagger}=\sigma$), then $X,Y$ are positive definite and if we define the map:
\begin{align*}
	&\mathcal{E}^{\prime}:\mathcal{M}_n\to\mathcal{M}_n \\
	&\mathcal{E}^{\prime}(\cdot):=Y^{-1}\mathcal{E}(X(\cdot)X^{\dagger})Y^{-\dagger}
\end{align*}
then a quick calculation shows $\mathcal{E}^{\prime}(\id)=\id$ and $\mathcal{E}^{\prime*}(\id)=\id$: 
\begin{align*}
	\mathcal{E}^{\prime}(\id)&=Y^{-1}\mathcal{E}(X(\id)X^{\dagger})Y^{-\dagger}=Y^{-1}\mathcal{E}(\rho)Y^{-\dagger} \\
	 &=Y^{-1}\sigma Y^{-\dagger}=Y^{-1}YY^{\dagger}Y^{-\dagger}=\id
\end{align*}
On the other hand, a similar calculation shows
\begin{align*}
	\mathcal{E}^{\prime *}(\id)&=X^{\dagger}\mathcal{E}^{*}(Y^{-\dagger}\id Y^{-1})X=\lambda \id 
\end{align*}
but since $\mathcal{E}^{\prime}$ was shown to be unital, $\mathcal{E}^{\prime *}$ is trace-preserving and $\lambda=1$ to begin with. \\
Conversely, given $X,Y$ as in the lemma, $XX^{\dagger}$ would be a fixed point of the Menon-operator. 
\end{proof}
Note that this completes the proof of Lemma \ref{lem:gurscaleequiv}. We only have to see that the conditions at the minimum are met if and only if the Menon operator has a fixed point. This also provides the connection between the Menon operator and Gurvits' approach: As in the classical case, the conditions for a fixed point of the Menon operator are given by the Lagrange conditions of the London-Djokovic potential.

We observe that the Menon-operator, if it is defined, is a continuous, homogeneous positive map. This lets us give a proof of a weak form of the Operator Sinkhorn Theorem:
\begin{prop}[\cite{ide13}] \label{thm:normalform}
Given a positive trace-preserving map $\mathcal{E}:\mathcal{M}_n\to\mathcal{M}_n$ such that there exists an $\varepsilon>0$ such that for all matrices $\rho\geq \varepsilon \id$ with unit trace it holds that $\mathcal{E}(\rho)\geq \frac{n\varepsilon}{1+(n-1)n\varepsilon}\id$, then we can find $X,Y>0$ such that $Y^{-1}\mathcal{E}(X(\cdot)X^{\dagger})Y^{-\dagger}$ is a doubly stochastic map.
\end{prop}
\begin{proof}
Let $\mathbf{T}_{\mathrm{normpos}}(\cdot):=\mathbf{T}_{\mathrm{pos}}(\cdot)/\tr(\mathbf{T}_{\mathrm{pos}}(\cdot))$ is the normalised operator. Now assume that for all $\rho\geq \varepsilon\id$ with $\tr(\rho)=1$, it holds $\mathcal{E}(\rho)\geq \delta \id$. In particular, if we call $\lambda_{max}$ the maximal eigenvalue of $\mathcal{E}(\rho)$, then $\lambda_{max}\leq 1-(n-1)\delta$. Hence we have:
\begin{align*}
	&\delta \id\leq \mathcal{E}(\rho)\leq (1-(n-1)\delta)\id \\
	&\Rightarrow \quad \frac{1}{\delta}\id\geq \mathcal{D}(\mathcal{E}(\rho))\geq \frac{1}{1-(n-1)\delta}\id \\
	&\Rightarrow \quad \frac{1}{\delta}\id\geq \mathcal{E}^*(\mathcal{D}(\mathcal{E}(\rho)))\geq \frac{1}{1-(n-1)\delta}\id \\ 
		&\Rightarrow \quad \delta \id\leq \mathcal{D}(\mathcal{E}^*(\mathcal{D}(\mathcal{E}(\rho))))\leq (1-(n-1)\delta)\id
\end{align*}
where we used the unitality of $\mathcal{E}^*$ in the third step. This implies
\begin{align*}
	\mathbf{T}_{\mathrm{normpos}}(\rho)\geq \frac{\delta}{1-(n-1)\delta}\id/n
\end{align*}
Now we want $\frac{\delta}{1-(n-1)\delta}\geq \varepsilon n$, in which case the compact set of matrices $\{\rho>0|\tr(\rho)=1, \rho\geq \varepsilon \id/n\}$ is mapped into itself, hence by Brouwer's fixed point theorem, we obtain a positive definite fixed point of $\mathcal{G}$. A quick calculation shows that this implies $\delta>n\varepsilon/(1+(n-1)n\varepsilon)$.
\end{proof}
Since a positive map $\mathcal{E}:\mathcal{M}_n\to\mathcal{M}_n$ can always be converted into a trace-preserving map $\mathcal{E}^{\prime}$ by setting $\rho:=\mathcal{E}^*(\id)$ and $\mathcal{E}^{\prime}(\cdot):=\mathcal{E}(\sqrt{\rho^{-1}}\cdot\sqrt{\rho^{-1}})$, the assumption that $\mathcal{E}$ be trace-preserving is not really necessary. As a direct corollary, we obtain a similar result, which might be easier to use:
\begin{cor}[\cite{ide13,geo15}]
Let $\mathcal{E}:\mathcal{M}_n\to\mathcal{M}_n$ be a trace-preserving and positivity improving map, then there exist maps $X,Y>0$ such that $Y^{-1}\mathcal{E}(X(\cdot)X^{\dagger})Y^{-\dagger}$ is a doubly stochastic map. 
\end{cor}
As in the classical matrix case in \cite{bru66}, one idea to obtain necessary and sufficient criteria is to extend the map $\mathbf{T}_{\mathrm{normpos}}$ to positive semidefinite matrices for all cases and then prove that there is a fixed point of the map inside the cone of positive definite matrices. However, we run into additional problems, since the cone of positive semidefinite matrices is not polyhedral (cf. \cite{lem12}: there is no and cannot exist an equivalent version of theorem \ref{thm:context}; this does not preclude that an extension exists, but such a result must depend on the operator in question). Moreover, even if a continuous extension may be possible by using perturbation theory (for instance), the hardest part is to prove the existence of a fixed point inside the cone.

\subsubsection{Other approaches and generalised scaling}
We have seen that at least two classical approaches for proving that a matrix can be scaled to a doubly stochastic matrix can be extended to the quantum case without too much trouble (the proofs however might be more difficult): nonlinear Perron-Frobenius theory and the barrier function approach. In a sense, we have also seen convex programming approaches. An immediate question is whether one can extend the entropy approach. This was also asked in \cite{gur04} and it is a major open question in \cite{geo15}, since the motivation of Schr\"odinger bridges heavily relies on relative entropy minimisation. The answer is not clear since something like a quantum relative entropy is only used only on the level of matrices and a justification via the Choi-Jamiolkowski isomorphism is not immediate (see Section \ref{sec:statechannel}):
\begin{align}
	D(\rho\|\sigma)=\tr(\rho\log \rho - \rho\log \sigma).
\end{align}

Another question is how to extend the theorems from the doubly-stochastic map to cover arbitrary marginals, i.e.~we want to scale a positive map $\mathcal{E}$ such that
\begin{align}
	\mathcal{E}(\rho)=\sigma, \qquad \mathcal{E}^*(\id)=\id
\end{align}
with some prespecified $\rho,\sigma$. For Gurvits' approach based on equation (\ref{eqn:london}) this is not really straightforward, since it is unclear how to take appropriate powers of $P,Q$. For the approach via nonlinear Perron-Frobenius theory, this can be done to some degree:
\begin{thm}[\cite{geo15}] \label{thm:arbscale}
Given a positivity improving map $\mathcal{E}:\mathcal{M}_d\to\mathcal{M}_d$ and two matrices $V,W>0$ with $\tr(V)=\tr(W)$, there exist matrices $X,Y\in\mathcal{M}_d$ and a constant $\lambda>0$ such that $\mathcal{E}^{\prime}(\cdot):=Y\mathcal{E}(X\cdot X^{\dagger})Y^{\dagger}$ fulfills
\begin{align*}
	\mathcal{E}^{\prime}(V)=W \\
	\mathcal{E}^{\prime *}(\id)=\id
\end{align*}
\end{thm}
\begin{proof}[Sketch of proof]
The proof is a variation of the methods for the case $V=W=\id$. We consider the following Menon-type operator, which was essentially defined in \cite{geo15}:
\begin{align*}
	\mathbf{T}_{\mathcal{E},V,W}:=\mathbf{D}_1\circ \mathcal{E}^*\circ \mathbf{D}_2 \circ \mathcal{E}
\end{align*}
where
\begin{align*}
	\mathbf{D}_1(\rho)=\rho^{-1/2}V^{-1}\rho^{-1/2} \\
	\mathbf{D}_2(\rho)=(W^{1/2}(W^{-1/2}\rho^{-1}W^{-1/2})^{1/2}W^{1/2})^2
\end{align*}

\textbf{Step 1:} Let $\mathcal{E}$ be positivity improving, then $\mathbf{T}_{\mathcal{E},V,W}:\overline{\mathcal{C}^d}\to\mathcal{C}^d$ is a well-defined, continuous, and homogeneous map. It is well-defined, since $\mathcal{E}$ maps $\overline{\mathcal{C}^d}\to\mathcal{C}^d$ and $\mathbf{D}_1$ and $\mathbf{D}_2$ send $\mathcal{C}^d\to\mathcal{C}^d$ if $V,W\in\mathcal{C}^d$. It is homogeneous, because $\mathcal{E}$ is linear and $\mathbf{D}_i(\lambda \rho)=\lambda^{-1} \mathbf{D}_i(\rho)$ for $i=1,2$. Finally, $\mathbf{D}_1$ is continuous as taking the square root of a positive definite matrix is continuous and matrix multiplication and inversion of positive definite matrices is continuous. Likewise, $\mathbf{D}_2$ is continuous and thus $\mathbf{T}_{\mathcal{E},V,W}$ as composition of continuous maps. 

\textbf{Step 2:} We now claim that a scaling of $\mathcal{E}$ with as in the theorem with $X,Y>0$ exists iff $\mathbf{T}_{\mathcal{E},V,W}$ has an eigenvector. This was observed in \cite{geo15} and is a straightforward but lengthy calculation.

\textbf{Step 3:} Finally, we can prove the existence of $X,Y>0$ such that a scaling exists by invoking Brouwer's fixed point theorem. The map 
\begin{align*}
	\tilde{\mathbf{T}}(\cdot):\overline{\mathcal{C}^d_1}\to\mathcal{C}^d_1
	\tilde{\mathbf{T}}(\cdot):=\mathbf{T}_{\mathcal{E},V,W}(\cdot)/\tr(\mathbf{T}_{\mathcal{E},V,W}(\cdot))
\end{align*}
is a continuous, well-defined map, hence it has a fixed point. This is necessarily an eigenvector of $\mathcal{T}_{\mathcal{E},V,W}$, hence defines a scaling. 
\end{proof}
The problem with this proof is that this Menon operator is no longer clearly a contraction mapping, hence uniqueness and convergence speed of the algorithm are not clear. Also, the obvious algorithm derived from this proof differs from the usual RAS algorithm. It is not clear how to remedy this or extend one of the other approaches. Also, this map has even worse prospect of being generalised to positive and not necessarily positivity improving maps. In any case, it is not immediately clear what the right necessary and sufficient conditions are. In the case of matrices, patterns were the important concept, but what is a pattern supposed to be for positive maps? One can always choose a basis and represent the map as matrix, but this is very much map-dependent and it is not clear what the correct interpretation will be.

Nevertheless, partial results for uniqueness have been achieved in \cite{fri16}: The author proves that for positivity preserving maps $\mathcal{E}$, there exists a ball around $\id$ such that if $V,W$ lie inside this ball, there exists a unique scaling of $\mathcal{E}$ to a trace-preserving positive map with $\mathcal{E}(V)=W$.

\subsection{Operator Sinkhorn theorem via state-channel duality} \label{sec:statechannel}
Another formulation of the operator Sinkhorn theorem is given by the Choi-Jamiolkow\-ski isomorphism. It states that given any positive map $\mathcal{E}:\mathcal{M}_n\to\mathcal{M}_n$, we have that 
\begin{align}
	\tau_{\mathcal{E}}:=(\operatorname{id}\otimes \mathcal{E})(\omega)
\end{align}
is a \emph{block-positive} matrix (i.e.~$\langle \phi_1|\langle \phi_2|\tau_{\mathcal{E}}|\phi_1\rangle|\phi_2\rangle\geq 0$ for all $|\phi_1\rangle,|\phi_2\rangle\in \mathcal{M}_n$). Here $\omega:=1/d\sum_{i,j=1}^n |ii\rangle\langle jj|\in\mathcal{M}_{n^2}$ is the so-called \emph{maximally entangled state}. If $\mathcal{E}$ is completely positive, i.e.~$\mathcal{E}\otimes \operatorname{id}_n$ is a positive map, then $\tau_{\mathcal{E}}$ is a positive semi-definite matrix. Now consider $X_1,~X_2\geq 0$ and $\mathcal{E}^{\prime}:=X_2^{\dagger}\mathcal{E}(X_1\cdot X_1^{\dagger})X_2$. We have
\begin{align}
	\tau_{\mathcal{E}^{\prime}}=(X_1^{tr}\otimes X_2^{\dagger})\tau_{\mathcal{E}}(X_1^{tr}\otimes X_2^{\dagger})^{\dagger}
\end{align}
where we use $(\id\otimes X_1)\sum_{i}|ii\rangle=(X_1^{tr}\otimes \id)\sum_i|ii\rangle$ and therefore 
\begin{align}
	\tau_{\mathcal{E}^{\prime}}&=(\id\otimes X_2^{\dagger})(\operatorname{id}\otimes \mathcal{E})((\id\otimes X_1)\omega(\id\otimes X_1)^{\dagger})(\id\otimes X_2) \\
	&=(X_1^{tr}\otimes X_2^{\dagger})(\operatorname{id}\otimes \mathcal{E})(\omega)(X_1^{tr \dagger}\otimes X_2)
\end{align}
Therefore, the task can be reformulated: Given a block positive matrix $\tau$, find $X_1,~X_2\in\mathcal{M}_d$ such that
\begin{align*}
	\tau^{\prime}:=(X_1\otimes X_2)\tau(X_1\otimes X_2)^{\dagger}
\end{align*}
fulfils $\tr_2(\tau)=\tr_1(\tau)=\id/d$, where $\tr_i$ denotes the partial trace over the $i$-th system in $\mathcal{M}_d\otimes \mathcal{M}_d\equiv \mathcal{M}_{d^2}$. For $\tau\geq 0$ these operations are called \emph{(local) filtering operations}. Often (c.f. \cite{git08,wol12}), one asks for $X_1,X_2\in SL(d)$ and the resulting trace being merely proportional to the identity, but this is of course just a normalisation.

We can then state an equivalent version of Sinkhorn scaling for positive map:
\begin{prop}[\cite{ken99,lei06,ver01}]
Let $\rho\in\mathcal{M}_d\otimes \mathcal{M}_d$ be a positive definite density matrix. Then there exist matrices $X_1, X_2\in\mathcal{M}_d$ such that
\begin{align}
	(X_1\otimes X_2)\rho(X_1\otimes X_2)^{\dagger}=\frac{1}{d^2}\id+\sum_{k=1}^{k^2-1}\xi_k J_k^1\otimes J_k^2
\end{align}
where $\{J_k^1\}_k\subset \mathcal{M}_d$ and $\{J_k^2\}_k\subset\mathcal{M}_d$ form a basis of the traceless complex matrices and $\xi\in\mathbb{C}$ for the first and second tensor factor in $\mathcal{M}_d\otimes \mathcal{M}_d$ respectively. 
\end{prop}
\begin{proof}
Note that a positive definite $\rho$ corresponds to a completely positive map, which maps positive semidefinite matrices to positive definite ones. In particular, the corresponding map is fully indecomposable. Hence by Theorem \ref{thm:sinkqm}, there exists a scaling to a doubly stochastic map, which again corresponds to a positive definite $\tilde{\rho}\in\mathcal{M}_d\otimes \mathcal{M}_d$ such that $\tr_1(\tilde{\rho})=\tr_2(\tilde{\rho})=\id/d$. By construction, $\{\id/d,J_k^1\}_k$ and $\{\id/d,J_k^2\}$ form an orthonormal basis of $\mathcal{M}_d$, hence we can express $\tilde{\rho}$ as
\begin{align*}
	\tilde{\rho}=\frac{1}{d^2}\id + \sum_{k=1}^{k^2-1}\xi_k J_k^1\otimes J_k^2 + \sum_{k=1}^{k^2-1} \chi^1_k \frac{\id}{d}\otimes J^1_k+\chi^2_k J_k^2\otimes \frac{\id}{d}
\end{align*}
with $\xi_k,\chi^1_k,\chi^2_k\in\mathbb{C}$ for all $k$. Then
\begin{align*}
	\tr_1(\tilde{\rho})&=\frac{1}{d}\id+ \sum_{k=1}^{k^2-1}\xi_k \tr(J_k^1)\otimes J_k^2 + \sum_{k=1}^{k^2-1} \left(\chi^1_k \tr\left(\frac{\id}{d}\right)\otimes J^1_k+\chi^2_k \tr(J_k^2)\otimes \frac{\id}{d}\right) \\
	&=\frac{1}{d}\id+ \sum_{k=1}^{k^2-1} \chi^1_k J^1_k \stackrel{!}{=}\frac{1}{d}\id
\end{align*}
But then, since the $J_k^1$ are linearly independent, $\chi^1_k=0$ for all $k$. Likewise, $\chi^2_k=0$ for all $k$ and we have the required normal form.
\end{proof}
The proposition has direct proofs and extensions to more than two parties (see for instance \cite{ver02a,ver03, wol12}). Here, it only uses the sufficient part of the criterion for scalability of positive maps, hence we can strengthen it to include parts of all block-positive matrices. Since, however, not all completely positive maps are fully indecomposable (e.g.~the map $\mathcal{E}:\rho\to |\psi\rangle\langle \psi|$ for some vector $|\psi\rangle\in\mathbb{C}^d$ is not), it certainly does not extend to all states. 

\subsection{Convergence speed and stability results} \label{sec:convstab}
Gurvits' proof already gives an estimate for the convergence speed of the scheme (see Theorem 4.7.3.~in \cite{gur04}). Let us give an alternative proof using Hilbert's metric which is equivalent to the classical proof in \cite{fra89} and reminiscent of the convergence proof in \cite{geo15}. Throughout the proof, we use several results from Appendix \ref{app:perfro}, in particular the definition of Hilbert's projective metric $d_H$ on the cone of positive semidefinite matrices and the definition of the contraction ratio $\gamma$ in equation (\ref{eqn:contraction}). To proceed, we define a metric on the space of positive maps that are scalable:
\begin{dfn}
Let $\mathcal{E},\mathcal{T}:\mathcal{M}_n\to\mathcal{M}_n$ be two positive maps such that $\mathcal{T}(\cdot)=Y\mathcal{E}(X\cdot X^{\dagger})Y^{\dagger}$ for some positive matrices $X,Y$. Then
\begin{align}
	\Delta(\mathcal{E},\mathcal{T})=d_H(X,\id)+d_H(Y,\id)
\end{align}
defines a metric on the space of positive maps (two maps that cannot be scaled to each other have infinite distance). 
\end{dfn}
A proof that this constitutes a metric may be found in \cite{lem12}, Chapter 2. Recall the Sinkhorn iteration as defined in equations (\ref{eqn:iterate1})-(\ref{eqn:iterate2}). For convenience we use a slightly different notation:
\begin{align*}
	\mathcal{E}^{(i)}&:=\mathcal{E}_{2i} \quad i>0 \\
	\mathcal{E}^{(i)\prime}&:=\mathcal{E}_{2i+1} \quad i\geq 0 \\
	\rho^{(i)}&:=\mathcal{E}^{(i-1)}(\id) \quad i>0 \\
	\sigma^{(i)}&:=\mathcal{E}^{(i)\prime \dagger}(\id) \\
	\mathcal{E}^{(0)}&:=\mathcal{E}.
\end{align*}
Then:
\begin{prop}
Let $\mathcal{E}:\mathcal{M}_n\to\mathcal{M}_n$ be a positivity improving, trace preserving map. Let $\mathcal{T}:=Y^{-1}\mathcal{E}(X\cdot X^{\dagger})Y$ be the unique doubly stochastic scaling limit.
\begin{align}
	\Delta(\mathcal{E}^{(k)},\mathcal{T})\leq \frac{\gamma^k}{1-\gamma}(d_H(\rho^{(1)},e)+d_H(\sigma^{(1)},e)) \\
	\Delta(\mathcal{E}^{(k)\prime},\mathcal{T})\leq \frac{\gamma^k}{1-\gamma}(d_H(\rho^{(1)},e)+d_H(\sigma^{(1)},e))
\end{align}
where $\gamma^{1/2}=\gamma^{1/2}(\mathcal{E})$ is the contraction ratio of equation (\ref{eqn:contraction}). In particular, this implies via proposition \ref{prop:prophilb} (implying that here, $\gamma<1$) that the convergence is geometric.
\end{prop}
\begin{proof}
The proof is similar to the classical one in \cite{fra89}. First recall the definition of $\Delta$ from Appendix \ref{app:perfro}:
\begin{align*}
	\Delta(\mathcal{E}):=\sup\{d_H(\mathcal{E}(\rho),\mathcal{E}(\sigma))|\rho,\sigma\geq 0\}
\end{align*}
Then $\Delta>0$ but finite, since $\mathcal{E}$ is a positivity improving map and the maximum is attained. This is true, because it suffices to consider $d_H(\mathcal{E}(\rho),\mathcal{E}(\sigma))$ on the compact set $\{A\geq 0| \|A\|_\infty=1\}$ using Proposition \ref{prop:prophilb} (iii).

We first make the following observations:
\begin{align}
	d_H(\rho,\sigma)&=d_H(\sigma^{-1/2}\rho\sigma^{-1/2},\id) \quad \forall \rho,\sigma>0\label{eqn:obs2b} \\
	d_H(\mathcal{E}(\rho),\mathcal{E}(\sigma))&\leq \gamma^{1/2}(\mathcal{E})d_H(\mathcal{E}(\rho),\mathcal{E}(\sigma)) \quad \forall \rho,\sigma>0,\mathcal{E}:\mathcal{M}_n\to\mathcal{M}_n\label{eqn:obs3b}
\end{align}
Equation (\ref{eqn:obs3b}) follows from the definition of $\gamma^{1/2}$. Equation (\ref{eqn:obs2b}) follows from the definition of $M$ and $m$ in the definition of the Hilbert metric and the fact that taking noncommutative inverses does not change positivity.

Let us now focus on $\gamma(\mathcal{E})$. Let $X,Y\in\mathcal{M}_n$ be invertible, then
\begin{align*}
	\Delta&=\sup\{d_H(Y^{\dagger}\mathcal{E}(X\rho X^{\dagger})Y,Y^{\dagger}\mathcal{E}(X\sigma X^{\dagger})Y)|\rho,\sigma\geq 0\} \\
		&=\sup\{d_H(\mathcal{E}(X\rho X^{\dagger}),\mathcal{E}(X\sigma X^{\dagger}))|\rho,\sigma\geq 0\} \\
		&=\sup\{d_H(\mathcal{E}(\tilde{\rho}),\mathcal{E}(\tilde{\sigma}))|\tilde{\rho},\tilde{\sigma}\geq 0\}
\end{align*}
using observation (\ref{eqn:obs3b}) and then $X$ being invertible. In particular, this implies that for every $\gamma^{1/2}(\mathcal{E}^{(i)})$ we have a universal upper bound 
\begin{align}
	\gamma^{1/2}(\mathcal{E}^{(i)})<\operatorname{tanh}(\Delta/4).
\end{align}
Since $\Delta>0$ but finite, this implies that we can upper bound each $\gamma(\mathcal{E}^{(i)})$ and $\gamma(\mathcal{E}^{\prime\,(i)})$ by some $\gamma<1$. The rest is basically an iteration.

Consider $d_H(\rho^{(2)},\id)$. By definition, $\rho^{(2)}=\mathcal{E}^{(1)}(\id)=\mathcal{E}^{(1)\prime}((\sigma^{(1)}){-1})$, and since all $\mathcal{E}^{(i)\prime}$ are unital:
\begin{align}
	d_H(\rho^{(2)},\id)&=d_H(\mathcal{E}^{(1)\prime}((\sigma^{(1)})^{-1}),\mathcal{E}^{(1)\prime}(\id)) \leq \gamma^{1/2}(\mathcal{E}^{(1)\prime})d_H(\id, \sigma^{(1)}) \label{eqn:obs4}
\end{align}
where we used (\ref{eqn:obs3b}) and then (\ref{eqn:obs2b}).
Similarly, since $\sigma^{(1)}=\mathcal{E}^{(0)\prime*}(\id)=\mathcal{E}^{(0)*}((\rho^{(1)})^{-1})$ and $\mathcal{E}^{(0)*}(\id)=\id$ by construction, we obtain:
\begin{align}
	d_H(\id, \sigma^{(1)})=d_H(\mathcal{E}^{(0)*}(\id),\mathcal{E}^{(0)*}((\rho^{(1)})^{-1}) \leq \gamma^{1/2}(\mathcal{E}^{(0)*})d_H(\rho^{(1)},\id) \label{eqn:obs5}
\end{align}
Combining (\ref{eqn:obs4}) and (\ref{eqn:obs5}) we obtain:
\begin{align}
	d_H(\rho^{(2)},\id)\leq \gamma d_H(\rho^{(1)},\id) \label{eqn:obs6}
\end{align}
Similarly,
\begin{align}
	d_H(\sigma^{(2)},\id)\leq \gamma d_H(\sigma^{(1)},\id) \label{eqn:obs7}
\end{align}
These are the key observations. Now using the definition of $\Delta(\cdot,\cdot)$ we obtain:
\begin{align*}
	\Delta(\mathcal{E}^{(k)},\mathcal{E}^{(k+1)})&=d_H((\rho^{(k)})^{-1},\id)+d_H((\sigma^{(k)})^{-1},\id)\\
		&\leq \gamma^{k-1}(d_H(\rho^{(1)},\id)+d_H(\sigma^{(1)},\id))
\end{align*}
Hence we have by the triangle inequality
\begin{align*}
		\Delta(\mathcal{E}^{(0)},\mathcal{E}^{(k+1)})\leq \sum_{l=0}^{k-1}\gamma^{l}(d_H(\rho^{(1)},\id)+d_H(\sigma^{(1)},\id))
\end{align*}
and therefore, if $\mathcal{T}$ denotes the limit of the Sinkhorn iteration, using the geometric series
\begin{align}
	\Delta(\mathcal{E}^{(0)},\mathcal{T})&\leq \frac{1}{1-\gamma}(d_H(\rho^{(1)},\id)+d_H(\sigma^{(1)},\id)) \label{eqn:obsstab}\\
	\Delta(\mathcal{E}^{(k)},\mathcal{T})&\leq \frac{\gamma^k}{1-\gamma}(d_H(\rho^{(1)},\id)+d_H(\sigma^{(1)},\id))
\end{align}
The other inequality for the maps $\mathcal{E}^{\prime}$ follows from symmetric arguments.
\end{proof}	
Note that in contrast to the classical case in \cite{fra89}, because of the noncommutativity in equation (\ref{eqn:obs2b}), a simple extension to the general scaling of positivity improving maps seems not possible.

Next, we wish to generalise also the stability results. It seems natural that this should follow from the contraction results above:
\begin{cor}
Let $\mathcal{E}:\mathcal{M}_n\to\mathcal{M}_n$ be positivity improving, then the scaling is continuous in $\mathcal{E}$.
\end{cor}
\begin{proof}
Let $\mathcal{E}$ be a positivity improving map and $\mathcal{E}^{\prime}=\mathcal{E}+\delta \mathcal{T}$ be a perturbation which is again positivity preserving, where $\mathcal{T}$ is a positive map with $\|\mathcal{T}\|=1$ (for instance in the operator norm). 

Then let $X,Y$ be such that they scale $\mathcal{E}$ to a doubly stochastic map. This implies that 
\begin{align*}
	Y\mathcal{E}^{\prime}(XX^{\dagger})Y^{\dagger}&=\id+\delta Y\mathcal{T}(XX^{\dagger})Y^{\dagger} \\
	X^{\dagger}\mathcal{E}^{\prime\,*}(Y^{\dagger}Y)X&=\id+\delta X^{\dagger}\mathcal{T}^{*}(Y^{\dagger}Y)X
\end{align*}
and the marginals are also close to $\id$. In fact, for any $\varepsilon>0$ we can find $\delta>0$ such that
\begin{align*}
	d_H(Y\mathcal{E}^{\prime}(XX^{\dagger})Y^{\dagger},\id)+d_H(X^{\dagger}\mathcal{E}^{\prime\,*}(Y^{\dagger}Y)X,\id),<\varepsilon
\end{align*}	
But then, by equation (\ref{eqn:obsstab}), we have that if $\mathcal{E}^{\prime,\prime}$ is the scaling of $Y\mathcal{E}^{\prime}(XX^{\dagger})Y^{\dagger}$ to a doubly stochastic map, then
\begin{align*}
	\Delta(Y\mathcal{E}^{\prime}(XX^{\dagger})Y^{\dagger},\mathcal{E}^{\prime,\prime})\leq \frac{1}{1-\gamma}(d_H(\rho^{(1)},\id)+d_H(\sigma^{(1)},\id))<\frac{1}{1-\gamma}\varepsilon
\end{align*}
Using the triangle inequality and the fact that $\Delta(\mathcal{E},\mathcal{E}^{\prime})<C\varepsilon$ for some constant $C$ finishes the proof.
\end{proof}
As noted, both theorems can be extended to cover all exactly scalable positive maps using Gurvits' Theorem 4.7.3. of \cite{gur04}. Given the result for classical matrices, it seems natural that the convergence speed for rank non-decreasing but not exactly scalable matrices should not be geometric.

\subsection{Applications of the Operator Sinkhorn Theorem}
Let us finally mention applications of the operator version of Sinkhorn's theorem. The state-version of the theorem, since it can be seen as a normal form for states under local operations, has been applied in the study of states under LOCC operations (see for instance \cite{ken99,lei06}). 

The approximate operator version was developed to obtain polynomial-time algorithms (Sinkhorn scaling) for a problem known as ``Edmond's problem''. It asks the following question (\cite{gur04}): Given a linear subspace $A$ of $\mathcal{M}_n$, does there exist a nonsingular matrix in $V$? The question can be asked also over different number fields and in different contexts. It is particularly interesting, because it is related to rational identity testing over non-commutative variables as studied in \cite{gar15,iva15}. For further input we refer the reader to the extended and well-written review of the applications in \cite{gar15}. 

Finally, the exact scalability of fully indecomposable positive maps provided bounds on the \emph{mixed discriminant} of matrix tuples, which is interesting to provide permanent bounds (\cite{gur00}).

\paragraph*{Acknowledgement}
I was first acquainted with the topic of matrix scaling through my Master's thesis supervisor Michael Wolf. We also worked together on unitary scaling and I thank him for his valuable input. I am supported by the Studienstiftung des deutschen Volkes.

\printbibliography

\appendix
\section{Preliminaries on matrices} \label{app:prelimmatrices}
Sinkhorn's theorem is closely related to irreducibility and notions connected to it. 

Therefore, we first recall the following charactersation of irreducible matrices.
\begin{prop} \label{prop:irred}
Let $A\in\R^{n\times n}$ be nonnegative. The following are equivalent:
\begin{enumerate}
	\item $A$ is irreducible.
	\item The digraph associated to $A$ is strongly connected.
	\item For each $i$ and $j$ there exists a $k$ such that $(A^k)_{ij}>0$.
	\item For any partition $I\cup J$ of $\{1,\ldots, n\}$, there exists a $j\in J$ and an $i\in I$ such that $A_{ij}\neq 0$.
\end{enumerate}
\end{prop}
An overview about these and similar properties can be found in \cite{sch77}. Graph related properties are proven in \cite{bru68}.

Let us now describe a graph for any matrix: Given a nonnegative matrix $A\in\mathbb{R}^{n\times n}$ and any partition $I\cup J=\{1,\ldots,2n\}$, let $I\cup J=V$ and $E=\{(i,j)|A_{ij}>0\}\subset I\times J$ be the vertices and edges of the (bipartite) Graph $G_A:=(V,E)$. 
\begin{dfn}
A bipartite graph $G=(V,E)$ has a \emph{perfect matching}, if it contains a subgraph where the degree of any vertex is exactly one, i.e.~any vertex is matched with exactly one other vertex.
\end{dfn}

Note that this definition is not dependent on the size of the entries of $A$, it only depends on whether an entry is positive or zero.

\begin{prop} [\cite{bru66} Lemma 2.3]
Let $A\in \R^{m\times n}$ be nonnegative. The following are equivalent:
\begin{enumerate}
	\item $A$ is fully indecomposable,
	\item $PAQ$ is fully indecomposable for all permutations $P,Q$,
	\item There exist permutations $P,~Q$ such that $PAQ$ is irreducible and has a positive main diagonal.
	\item For any $(i,j)\in E$ the edge set of the bipartite graph $G_A$ for $A$, there exists a perfect matching in $G_A$ containing this edge.
\end{enumerate}
\end{prop}
\begin{proof}[Sketch of proof] The equivalence $(1)\Leftrightarrow (2)$ is obvious and $(1)\Leftrightarrow (3)$ is done in \cite{bru66}. The direction $\Rightarrow$ follows from the Frobenius-K\"{o}nig theorem (cf. \cite{bha96}, Chapter 2). The converse direction follows from a short contradiction proof.

Finally, $(1)\Leftrightarrow (4)$ follows essentially from Theorem 4.1.1 in \cite{las09}, which was first observed in \cite{het64}\footnote{In Hungarian. Reference taken from \cite{las09}}.
\end{proof}

Since multiplication of positive diagonal matrices from the right and from the left does not change the pattern of a matrix, having a matrix that has the required row and column sums is an easy necessary condition for scalability. 

Let us now consider the special case of doubly stochastic matrices in more detail. We define:
\begin{dfn}
Let $A\in \mathbb{R}^{n\times n}$ be a nonnegative matrix. Then $A$ has \emph{total support} if it is nonzero and for every $A_{ij}>0$ there exists a permutation $\sigma$ such that $\sigma(i)=j$ and $\prod_{k=1}^n A_{\sigma(k)k}\neq 0$. In other words, $A$ has total support if any nonzero element lies on a positive diagonal (\cite{sin67b}). 

Furthermore, $A$ has \emph{support}, if there exists a positive diagonal, i.e.~there exists an $A_{ij}$ such that for some permutation $\sigma$ with $\sigma(i)=j$ we have $\prod_{k=1}^n A_{\sigma(k)k}\neq 0$.
\end{dfn}
\begin{prop} \label{prop:fullyindec}
Let $A\in\mathbb{R}^{n\times n}$ be a nonnegative matrix. The following are equivalent:
\begin{enumerate}
	\item After independent permutations of rows and columns, $A$ is a direct sum of fully indecomposable matrices.
	\item $A$ has a doubly-stochastic pattern.
	\item $A$ has total support
\end{enumerate}
Furthermore, $A$ has support if and only if there exists a matrix $B$ with a subpattern of $A$ with total support.
\end{prop}
\begin{proof}[Sketch of Proof] For $1\Leftrightarrow 2$ we follow the proof in \cite{bru66}.

Let $A$ be doubly stochastic. Since we can permute rows and columns independently, we can assume that $A$ is of the form
\begin{align*}
	A:=\begin{pmatrix}{} A_1 & 0 & \ldots & 0 \\ A_{21} & A_2 & \ldots & 0 \\ \vdots & & & \vdots \\ A_{k1} & A_{k2} & \ldots & A_k \end{pmatrix}
\end{align*}
for some $k\in\mathbb{N}$. All $A_i$ are either $1\times 1$ zero-matrices or fully indecomposable (otherwise iterate). Since $A$ is doubly stochastic one can quickly see that $A_{ij}=0$ for all $i<j$. Furthermore, no $A_i$ can be zero, because this would then result in a zero-row. Hence, $A$ can be decomposed as a direct sum of fully indecomposable maps.

For $2\Leftrightarrow 3$ see \cite{sin67b}.

Finally, consider a matrix $B$ with total support. Clearly, if it is a submatrix of some other matrix $A$, then $A$ will have support, since any element $A_{ij}>0$ which is contained in $B$ will lie on a nonzero diagonal. Conversely, if $A$ has support, setting any element $A_{ij}$ which does not lie on a positive diagonal to zero produces a matrix that has total support.
\end{proof}

\section{Introduction to Nonlinear Perron-Frobenius theory} \label{app:perfro}
The basic result underlying Perron-Frobenius theory is an old theorem from \cite{per07} and \cite{fro12} stating:
\begin{thm}
Let $A\in \R^{n\times n}$ be a nonnegative, irreducible matrix with spectral radius $\rho(A)$. Then $\rho(A)>0$ is a nondegenerate positive eigenvalue of $A$ with a one-dimensional eigenspace consisting of a vector $x$ with only positive components.
\end{thm}
The theorem was later interpreted geometrically in \cite{bir57,sam57}, where the authors noted that it follows using Hilbert's projective metric and contraction principles. From then on, it was slowly extended to (not necessarily linear) operators, which lies at the heart of nonlinear Perron-Frobenius theory. The connection to matrix scaling became clear int the 60s to 80s and parts of each theory have developed alongside each other since. Probably the best current reference on the topic is \cite{lem12}. In the following, we sketch some of the most important ideas surrounding the theory:

Recall that given a topological vector space $\mathcal{V}$, a \emph{cone} is a set $\mathcal{C}\subset \mathcal{V}$ such that for all $v\in\mathcal{C}$, $\alpha v\in\mathcal{C}$ for all $\alpha> 0$. A \emph{convex cone} is a cone that contains all convex combinations. By definition, it is equivalent to say that $\mathcal{C}$ is a convex cone if and only if $\alpha v+\beta w\in\mathcal{C}$ for all $v,w\in\mathcal{C}$ and all $\alpha,\beta> 0$. A cone is called \emph{solid}, if it contains an interior point in the topology of the vector space and \emph{closed} if it is closed in the given topology. It is called \emph{polyhedral}, if it is the intersection of finitely many closed half-spaces (cf. \cite{roc97}).

An easy way to construct a convex, solid cone is by using a partial order $\geq$. Then the set $\mathcal{C}$ defined via $v\in\mathcal{C}~\Leftrightarrow v\geq 0$ is a closed convex cone (see \cite{roc97} for the connection between ordered vector spaces and convex cones). Given two cones $\mathcal{C},\mathcal{K}$ which are defined by a partial order, we call a map $\mathbf{T}: \mathcal{C}\to\mathcal{K}$ \emph{order-preserving} or \emph{monotonic}, if for $v\geq w$ we have $\mathbf{T}(v)\geq \mathbf{T}(w)$. We call it \emph{strongly order-preserving}, if for every $v\geq w\in\mathcal{V}$ we have $\mathbf{T}(v)-\mathbf{T}(w)\in\mathcal{V}^{\prime}>0$. Last but not least, we call $\mathbf{T}$ \emph{homogeneous}, if $\mathbf{T}(\alpha v)=\alpha \mathbf{T}(v)$ for all $\alpha\in\mathbb{R}$.

As is the case with classical Perron-Frobenius theory, the spectral radius is the crucial notion. For general maps between cones, there are several definitions, which turn out to be the same for most purposes, hence we restrict to one such notion:
\begin{dfn}[\cite{lem12}, Chapter 5.2]
Let $\mathcal{K}$ be a solid closed cone in a finite dimensional vector space with a fixed norm $\|\cdot\|$ and $f:\mathcal{K}\to\mathcal{K}$ a continuous homogeneous map. Define the \emph{cone spectral radius} as
\begin{align*}
	r_{\mathcal{K}}(f):=\sup\{\limsup\limits_{m\to\infty} \|f^m(x)\|^{1/m}|0\neq x\in \mathcal{K}\}
\end{align*}
\end{dfn}
The first crucial observation is that the spectral radius is actually attained for homogeneous order-preserving maps (which is not the case for general maps):
\begin{thm}[\cite{lem12}, Cor. 5.4.2] \label{thm:fpexists}
Let $\mathcal{K}$ be a solid closed cone in a finite dimensional vector space $\mathcal{V}$. If $f:\mathcal{K}\to\mathcal{K}$ is a continuous, homogeneous, order-preserving map, then there exists $x\in\mathcal{K}\setminus\{0\}$ with $f(x)=r_{\mathcal{K}}(f)x$
\end{thm}
Note that the theorem does not tell us whether the eigenvector lies inside the cone or on its boundary. The next and very powerful theorem does also not settle existence, but if existence is known, then it assures uniqueness and convergence:
\begin{thm}[\cite{lem12}, Thm. 6.5.1] \label{thm:uniquefp}
Let $\mathcal{K}$ be a solid closed cone in a finite dimensional vector space $\mathcal{V}$ and let $\varphi\in \mathcal{K}^*$ the dual cone. If $f:\operatorname{int}(\mathcal{K})\to\operatorname{int}(\mathcal{K})$ is a homogeneous strongly order-preserving map and there exists a $u\in\operatorname{int}(\mathcal{K})$ with $\varphi(u)=1$ such that $f(u)=ru$, then
\begin{align}
	\lim\limits_{k\to\infty} \frac{f^k(x)}{\varphi(f^k(x))}=u \label{eqn:fixedpointconv}
\end{align}
for all $x\in\operatorname{int}(\mathcal{K})$.
\end{thm}
This theorem is essentially due to the fact that the map is contractive in what is known as Hilbert's projective metric of the cones. Once the attractiveness is established, uniqueness follows immediately, since if we had another fixed point $v\in\operatorname{int}(\mathcal{K})$, then $f^k(v)=v$ would imply a contradiction to equation \ref{eqn:fixedpointconv}. Since the attractiveness can be used to estimate convergence speeds (see \cite{fra89}), we include the relevant proposition due to Birkhoff (\cite{bir57}).

\begin{dfn}[\cite{lem12}, Chapter 2.1]
Let $\mathcal{K}\subset V$ be a closed, convex, solid cone in some vector space, then for any $x,y\in\mathcal{K}$ such that $x\leq \alpha y$ and $y\leq \beta x$ for some $\alpha,\beta> 0$, define 
\begin{align*}
	M(x/y;\mathcal{K}):=\inf\{\beta>0| x\leq \beta y\} \\
	m(x/y;\mathcal{K}):=\sup\{\alpha>0| \alpha y\leq x\}
\end{align*}
Then we can define \emph{Hilbert's projective metric} as
\begin{align}
	d_{H}(x,y,\mathcal{K}):=\ln \left(\frac{M(x/y)}{m(x/y)}\right)
\end{align}
We will leave out $\mathcal{K}$, when it is clear from the context. Furthermore, we set $d_{H}(0,0)=0$ and $d_{H}(x,y)=\infty$ if $d_{H}$ is otherwise not well-defined.
\end{dfn}

We have the following properties:
\begin{prop}[\cite{lem12} Proposition 2.1.1] \label{prop:prophilb}
Let $\mathcal{K}\subset V$ be a closed, convex, solid cone in some vector space $V$. Then $d_H$ satisfies:
\begin{enumerate}
	\item[(i)] $d_H(x,y)\geq 0$ and $d_H(x,y)=d_H(y,x)$ for all $x,y\in\mathcal{K}$.
	\item[(ii)] $d_H(x,z)\leq d_H(x,y)+d_H(y,z)$ for all $x,y,z\in\mathcal{K}$ such that the quantities are well-defined.
	\item[(iii)] $d_H(\alpha x,\beta y)=d_H(x,y)$ for all $x,y\in\mathcal{K}$ and $\alpha,\beta>0$. 
\end{enumerate}
\end{prop}
Note that the first two properties show that $d_H$ is indeed a metric and the third property shows why it is called a \emph{projective metric}. 

\begin{prop}[\cite{bir57}] \label{prop:birkhoff}
Let $\mathcal{K}\subset V$ be a bounded, closed, convex and solid cone in some vector space $V$. Let $\mathcal{E}:\mathcal{K}\to\mathcal{K}$ be a linear map, then
\begin{align}
	\gamma^{1/2}(\mathcal{E}):=\sup \left\{ \frac{d_H(\mathcal{E}(x),\mathcal{E}(y))}{d_H(x,y)}\middle|x,y\in\mathcal{K}\right\}\leq \operatorname{tanh}(\Delta/4) \label{eqn:contraction}
\end{align}
where $\Delta:=\max\left\{d_H(\mathcal{E}(x),\mathcal{E}(y))|x,y\in\mathcal{K}\right\}$.

Furthermore, $\gamma^{1/2}(\mathcal{E}\circ\mathcal{F})\leq \gamma^2(\mathcal{E})\gamma^2(\mathcal{F})$ and $\gamma^{1/2}(\mathcal{E}^*)=\gamma^2(\mathcal{E})$.
\end{prop}
A proof can be found in \cite{bir57,bau65}. The result can be extended to much more general scenarios, see also \cite{eve95} and references therein. One can actually show that equality holds, i.e.~$\operatorname{tanh}(\Delta/4)$ is also attained, but this is not important here.

With all this machinery, we still need to prove existence of a fixed point in the interior of $\mathcal{P}$. The general theory for proving that a fixed point lies in the interior is weak and we generally have to prove it ``by hand''. 

For, the Menon operator we have an additional problem: It is at first only well-defined on the interior of the cone of positive semidefinite matrices. A natural question is whether it can be extended to cover the closed cone as well. For matrices, this is covered by the following theorem:

\begin{thm}[\cite{lem12}, theorem 5.1.5] \label{thm:context}
Let $\mathcal{C},\mathcal{K}$ be cones and $\mathbf{S}:\mathcal{C}\to\mathcal{K}$ be an order-preserving, homogeneous map. If $\mathcal{C}$ is solid and polyhedral, then there exists a continuous, order-preserving, homogeneous extension $\overline{\mathbf{S}}:\overline{\mathcal{C}}\to\overline{\mathcal{K}}$. 
\end{thm}

\section{Preliminaries on Positive Maps} \label{app:posprel}
In this section, let $\mathcal{C}^n\subset \mathcal{M}_n$ be the cone of positive definite matrices (elements will also be written as $A>0$) with its closure $\overline{\mathcal{C}^n}$, the cone of positive semidefinite matrices (elements will also be written as $A\geq 0$). Likewise, a subscript $_1$ at any of the cones denotes the bounded subset of unit trace matrices. Positive maps on cones are elements form a cone themselves, the dual cone, which will be denoted by $(\mathcal{C}^n)^*$ and $(\overline{\mathcal{C}^n})^*$.
 
Let us start with irreducible maps. Many different characterizations exist, which we recall for the reader's convenience:
\begin{prop}
For a positive, linear map $T:\mathcal{M}_d\to\mathcal{M}_d$ the following properties are equivalent:
\begin{enumerate}
	\item $T$ is irreducible,
	\item if $P\in\mathcal{M}_d$ is a Hermitian projector such that $T(P\mathcal{M}_dP)\subset P\mathcal{M}_dP$ then $P\in\{0,\id\}$,
	\item for every nonzero $A\geq 0$ we have $(\mathrm{id}+T)^{d-1}(A)>0$,
	\item for every nonzero $A\geq 0$ and every strictly positive $t\in\mathbb{R}$ we have $\operatorname{exp}(tT)(A)>0$,
	\item There does not exist a nontrivial orthogonal projection $P$ s.th. $\tr(T(P)(\id-P))=0$.
\end{enumerate}
\end{prop}
Most of these properties are well-known. A proof can be found in \cite{wol12}.

As with matrices in the original Sinkhorn theorem, irreducibility is not the right characterization to work with, since given an irreducible map $\mathcal{E}$ and two $X,Y>0$, $Y\mathcal{E}(X.X^{\dagger})Y^{\dagger}$ is not necessarily irreducible. Before giving a characterization of fully indecomposable maps, we will study rank non-decreasing maps. 
\begin{dfn}[\cite{gur04}]
To every positive map $\mathcal{E}:\mathcal{M}_n\to\mathcal{M}_n$ and any unitary $U\in U(n)$, we associate the \emph{decoherence operator} $\mathcal{E}_U$ via:
\begin{align}
	\mathcal{E}_U(X):=\sum_i \mathcal{E}(u_iu_i^{\dagger})\tr(Xu_iu_i^{\dagger})
\end{align}
where $u_i$ is the $i$-th row of $U$. Furthermore, we associate to every decoherence operator the tuple
\begin{align}
	\mathbf{A}_{\mathcal{E},U}:=(\mathcal{E}(u_1u_1^{\dagger}),\ldots,\mathcal{E}(u_nu_n^{\dagger}))
\end{align}
\end{dfn}
This will be important during the proof of the Sinkhorn scaling, because every map $\mathcal{E}$ will be associated to the mixed discriminants of its decoherence operators:
\begin{dfn}
Let $(A_i)_i$ be an $n$-tuple with $A_i\in\mathcal{M}_n$, then 
\begin{align}
	M(A_1,\ldots,A_n):=\frac{\partial^n}{\partial x_1\ldots\partial x_n} \detm (x_1A_1+\ldots +x_nA_n)|_{x_1,\ldots,x_n=0}
\end{align}
is called the \emph{mixed discriminant}.
\end{dfn}

Then we have the following characterization of rank non-decreasing maps, which is essentially due to \cite{gur04}:
\begin{prop} \label{prop:ranknon}
Let $\mathcal{E}:\mathcal{M}_n\to\mathcal{M}_n$ be a positive, linear map. Then the following expressions are equivalent:\begin{enumerate}
	\item[(i)] $\mathcal{E}$ is rank non-decreasing.
	\item[(ii)] $\mathcal{E}_U$ is rank non-decreasing for any unitary $U\in U(n)$.
	\item[(iii)] For any $U\in U(n)$, if $(A_i)_i:=\mathbf{A}_{\mathcal{E},U}$, then
	\begin{align*}
		\rank\left(\sum_{i\in \mathcal{S}} A_i\right)\geq |\mathcal{S}| \qquad \forall \mathcal{S}\subseteq\{1,\ldots,n\}
	\end{align*}
	\item[(iv)] For any $U\in U(n)$, $M(\mathbf{A}_{\mathcal{E},U})>0$.
	\item[(v)] $\mathcal{E}^{\prime}(\cdot):=Y^{\dagger}\mathcal{E}(X\cdot X^{\dagger})Y$ is rank non-decreasing for any $X,Y$ of full rank.
\end{enumerate}
\end{prop}
The proofs that (i), (ii), (iii) and (v) are equivalent are essentially the same as for the fully indecomposable case in \ref{prop:fully}. It remains to show the equivalence of (v) with (i). This was done in \cite{pan85}.

We can define what will turn out as a measure of being indecomposable for a tuple of matrices:
\begin{dfn}
Let $A:=(A_i)_i$ be an $n$-tuple of matrices $A_i\in \mathcal{M}_n$ and denote by $A^{ij}$ the tuple where $A_i$ is substituted by $A_j$. Then define:
\begin{align}
	\overline{M}(A):=\min_{i\neq j} M(A^{ij})
\end{align}
the minimal mixed discriminant.
\end{dfn}

For fully decomposable maps, we have the following characterization (part of which is already present in \cite{gur04}):
\begin{prop} \label{prop:fully}
Let $\mathcal{E}:\mathcal{M}_n\to\mathcal{M}_n$ be a positive, linear map. Then the following expressions are equivalent:
\begin{itemize}
	\item[(i)] $\mathcal{E}$ is fully indecomposable
	\item[(ii)] $\mathcal{E}^*$ is fully indecomposable
	\item[(iii)] For all singular, but nonzero $A\geq 0$, $\rank(\mathcal{E}(A))>\rank A$.
	\item[(iv)] Property (iii) holds for $Y\mathcal{E}(X\cdot X^{\dagger})Y^{\dagger}$ for every $X,Y>0$. 
	\item[(v)] There do not exist nontrivial orthogonal projections $P, Q$ of the same rank such that $\tr(\mathcal{E}(P)(\id-Q))=0$.
	\item[(vi)] $\mathcal{E}_U$ is fully indecomposable for all $U\in U(n)$.
	\item[(vii)] For any $U\in U(n)$, if $(A_i)_i:=\mathbf{A}_{\mathcal{E},U}$, then
	\begin{align*}
		\rank\left(\sum_{i\in \mathcal{S}} A_i\right)> |\mathcal{S}| \qquad \forall \mathcal{S}\subset\{1,\ldots,n\},0<|S|<n
	\end{align*} 
	\item[(viii)] $\overline{M}(A_{\mathcal{E},U})>0$ for all $U\in U(n)$.
\end{itemize}
Furthermore, when this is satisfied, $\mathcal{E}$ and via $(v)$ also $\mathcal{E}^*$ map the open cone $\mathcal{C}^n$ into itself. \\
Note also that the properties (vi) to (viii) are also equivalent for any fixed unitary.
\end{prop}
\begin{proof}
(i) $\to$ (iii): let $\mathcal{E}$ be fully indecomposable and assume there was a nonzero $A\geq 0$, with $\rank( \mathcal{E}(A))\leq \rank A$. Since the kernels are vector spaces, this implies we can find a unitary matrix $U$ transforming the basis such that $\kernel(\mathcal{E}(A))\supseteq U\cdot\kernel A$. Thus:
\begin{align*}
	\kernel(\mathcal{E}(A))&\supseteq U\cdot\kernel A \\
	\Leftrightarrow \quad \kernel(\mathcal{E}(A))&\supseteq \kernel UAU^{\dagger} \\
	\Leftrightarrow \quad \kernel(U^{\dagger}\mathcal{E}(A)U)&\supseteq \kernel A
\end{align*}
The latter implies $\supp(U^{\dagger}\mathcal{E}(A)U)\subseteq \supp A$. Let $P$ be the projection onto the support of $A$. By assumption, $P\neq \{0,\id\}$ since $A$ is nonzero and singular. For any positive matrix $B$ with $BP=B$, we have $\supp{B}\subseteq \supp A$. Hence, there exists a constant $r>0$ such that $A\geq rB$. Then $\mathcal{E}(A)\geq r\mathcal{E}(B)$ and $\supp(\mathcal{E}(B))\subseteq\supp(\mathcal{E}(A))$. But this implies via linearity $\supp(Q\mathcal{E}(\mathcal{M}_n)Q)\subseteq \supp (P\mathcal{M}_nP)$, where $Q:=UPU^{\dagger}$ is an orthogonal projection. \\
(iii) $\leftrightarrow$ (iv): Given (iii), the claim follows immediately from the fact that since $X,Y> 0$, the matrix ranks are not changed. For any nonzero and singular $A$ we have $\rank (A)=\rank( XAX^{\dagger})$. By assumption, for any nonzero, singular $A\geq 0$ we have $\rank(\mathcal{E}(A))>\rank (A)$, $\rank(\mathcal{E}(XAX^{\dagger}))>\rank(XAX^{\dagger})$ and hence 
\begin{align*}
	\rank(Y\mathcal{E}(XAX^{\dagger})Y^{\dagger})>\rank(A)
\end{align*}

(iii) $\to$ (i): Note that given $P,Q$ of the same rank such that $\mathcal{E}(P\mathcal{M}_nP)\subseteq Q\mathcal{M}_nQ$, we have in particular $\mathcal{E}(P)=QAQ$ for some $A\in\mathcal{M}_n$. Since $Q$ is of the same rank as $P$, $\rank (\mathcal{E}(P))=\rank(QAQ) \leq \rank P$, which is a contradiction. 

(v) $\to$ (i): Note that if $\mathcal{E}(P\mathcal{M}_nP)\subseteq Q\mathcal{M}_nQ$, then in particular $\tr(\mathcal{E}(P)(\id-Q))=0$ since $(\id-Q)$ is the orthogonal complement of $Q$. 

(i) $\to$ (v): Let $A\geq 0$. By positivity of $\mathcal{E}$, we have 
\begin{align*}
	0\leq \tr(\mathcal{E}(PAP)(\id-Q))&=\tr(AP\mathcal{E}^{*}(\id-Q)P) \\
	&\leq \|A\|_{\infty} \tr(P\id P\mathcal{E}^{*}(\id-Q))=\|A\|_{\infty}\tr(\mathcal{E}(P)(\id-Q))=0
\end{align*}
Hence in particular $\supp (\mathcal{E}(PAP))\subseteq \supp(Q)$ and $\mathcal{E}(P\mathcal{M}_nP)\subseteq Q\mathcal{M}_nQ$. 

(i) $\leftrightarrow$ (ii): This equivalence follows directly from (iv) by expressing $Q$ and $P$ in terms of the projections onto the orthogonal complements.

The remaining equivalences (i) $\leftrightarrow$ (vi),(vii),(viii) can be found in \cite{gur04} (with proofs scattered throughout the earlier papers by the same author). We just repeat them here using our notation for the reader's convenience.

(iii) $\to$ (vi): By definition, $\mathcal{E}_U=\mathcal{E}\circ \mathcal{U}$ where $\mathcal{U}(X)=\sum_i \tr(Xu_iu_i^{\dagger})$. Obviously, $\mathcal{U}$ is doubly stochastic, hence rank non-decreasing. Therefore, if (iii) holds for $\mathcal{E}$, it must also hold for $\mathcal{E}\circ \mathcal{U}$. 

(vi) $\to$ (iii): Let $\mathcal{E}_U$ be fully indecomposable for all unitaries $U$ and assume that $\rank(\mathcal{E}(X))\leq \rank (X)$ for some $X\geq 0$ with $0<\rank(X)<n$. Let $U$ be the unitary that diagonalizes $X$, then $\mathcal{E}_U(X)=\mathcal{E}(X)$, hence $\mathcal{E}_U$ is not fully indecomposable. This is a contradiction.

(vi) $\Leftrightarrow$ (vii): Let $\mathcal{T}_U$ be fully indecomposable. Then
\begin{align*}
	\rank(X)&< \rank (\mathcal{E}_U(X)) \\
	&=\rank \left(\sum_{i=1}^n \mathcal{E}(u_iu_i^{\dagger})\tr(Xu_iu_i^{\dagger})\right) \\
		&=\rank \left( \sum_{\substack{1\leq i\leq n \\ \tr(Xu_iu_i^{\dagger})\neq 0}}\mathcal{E}(u_iu_i^{\dagger})\right) \\
\end{align*}
Note that if $S:=\{i|\tr(Xu_iu_i^{\dagger})\}$, then $\rank(X)\leq |S|$ and hence follows the claim. For the other direction, we can use the same idea.

(vii) $\leftrightarrow$ (viii): Let $A:=\mathbf{A}_{\mathcal{E},U}$ for all unitary $U$ fulfill (vii). Define $A^{ij}$ as the tuple where the $i$-th element is replaced by the $j$-th. Note that
\begin{align*}
	\rank \left(\sum_{k\in \mathcal{S}} A^{ij}_k\right)\geq 
		\rank\left(\sum_{k\in\mathcal{S}\setminus \{j\}} A_k\right) \geq |\mathcal{S}|
\end{align*}
for any $\mathcal{S}\subset \{1,\ldots,n\}$, where the last inequality follows from the fact that $\mathcal{E}$ is fully indecomposable by assumption. Hence, from the proposition \ref{prop:ranknon} we know that the mixed discriminant of $A^{ij}$ cannot vanish, i.e.~$M(A^{ij})>0$. Minimizing over $i\neq j$ and the compact $U(n)$ gives $\tilde{M}(\mathbf{A}_{\mathcal{E},U})>0$. 

Conversely, let $A:=\mathbf{A}_{\mathcal{E},U}$ not fulfill (vii) for some unitary $U$, i.e.~for some $\mathcal{S}\subset \{1,\ldots,n\}$ with $0<\mathcal{S}<n$ we have $\rank\left(\sum_{k\in\mathcal{S}}A_k\right)\leq |\mathcal{S}|$. Let $i\in\mathcal{S},j\notin \mathcal{S}$, then for the tuple $A^{(ij)}$ as before, we have:
\begin{align*}
	\rank \left(\sum_{k\in \mathcal{S}\cup\{j\}} A^{ij}_k\right) &=
	 \rank \left(\sum_{k\in\mathcal{S}} A_k\right) < |\mathcal{S}|+1=|\mathcal{S}\cup \{j\}|
\end{align*}
But then, by proposition \ref{prop:ranknon}, $M(A^{ij})=0$ and hence also $\overline{M}=0$.
\end{proof}
This proposition shows in particular that any fully indecomposable map is \emph{primitive}: For any unit trace $\rho\geq 0$, $\mathcal{E}^d(\rho)>0$. Note that the converse might not be true. By the characterization of primitive maps (\cite{san10}, Theorem 6.7), this implies that each fully indecomposable map has only one fixed point.

\begin{lem} \label{lem:kernelred}
If $\mathcal{T}$ is a doubly-stochastic positive linear map, then there exists a unitary matrix $U$ such that $U\mathcal{T}(\cdot)U^{\dagger}$ admits a set of orthogonal projections $\{P_i\}_i$ such that $\sum_i P_i=\id$, $P_iP_j=\delta_{ij}P_i$ and $U\mathcal{T}(P_i\mathcal{M}_d P_i)U^{\dagger}\subseteq P_i\mathcal{M}_d P_i$. Furthermore, the restriction of $U\mathcal{T}(\cdot)U^{\dagger}$ to $P_i\mathcal{M}_d P_i$ is fully indecomposable for every $i$.
\end{lem}
\begin{proof}
Note that for an arbitrary unitary $U>0$ the maps $\mathcal{T}(U(\cdot)U^{\dagger})$ and $U\mathcal{T}(\cdot)U^{\dagger}$ are still doubly-stochastic. Let $P,Q$ be a nontrivial Hermitian projector decomposing $\mathcal{T}$, i.e.~$\tr(\mathcal{T}(P)(\id-Q))=0$ by proposition \ref{prop:fully} (if no such projector exists, we are finished). Then we have:
\begin{align*}
	0&=\tr(\mathcal{T}(P)(\id-Q))=\tr(P\mathcal{T}^*(\id-Q)) \\
	&=\tr(P)-\tr(P\mathcal{T}^*(Q))=\tr(Q)-\tr(Q\mathcal{T}(P)) \\
	&=\tr(Q\mathcal{T}(\id-P))
\end{align*}
where we used that $\mathcal{T}$ is doubly-stochastic in the second and last equality and in between we only used the cyclicity and linearity of the trace as well as the fact that $P$ and $Q$ have equal rank and thus their traces equal. This means that if $P$ reduces $\mathcal{T}$, then also $\id-Q$ reduces $\mathcal{T}$, i.e.
\begin{align*}
	&\mathcal{T}(P\mathcal{M}_nP)\subset Q\mathcal{M}_nQ \\
	\Rightarrow \quad &\mathcal{T}((\id-P)\mathcal{M}_n(\id-P))\subset (\id-Q)\mathcal{M}_n(\id-Q)
\end{align*}
Since the two projections $P,Q$ are of the same rank, there exists a unitary matrix $U$ such that $Q=UPU^{\dagger}$. This implies that $\mathcal{T}^{\prime}(\cdot)=U\mathcal{T}(\cdot)U^{\dagger}$ is reducible by $P$ and $(\id-P)$, which implies that $\mathcal{T}^{\prime}$ is a direct sum of maps defined on $P\mathcal{M}_nP$ and $(\id-P)\mathcal{M}_n(\id-P)$. 

We obtain these maps by setting 
\begin{align*}
	\mathcal{T}^{\prime}_1:=\mathcal{T}^{\prime}(P\cdot P)|_{P\mathcal{M}_nP} \\
	\mathcal{T}^{\prime}_2:=\mathcal{T}((\id-P)\cdot (\id-P))|_{(\id-P)\mathcal{M}_n(\id-P)}.
\end{align*}
By construction, $P,(\id-P)$ are the identities on the respective subspaces and the maps are therefore doubly stochastic, i.e.~$\mathcal{T}^{\prime}_1(\id_{P\mathcal{M}_nP})=\mathcal{T}(P)=P=\id_{P\mathcal{M}_nP}$ (and for $\mathcal{T}^{\prime}_2$ equivalently).

If the restricted maps are not fully indecomposable, we can iterate the procedure, thereby going over to $\mathcal{T}^{\prime\prime}(\cdot)=(U_1\oplus U_2)\mathcal{T}^{\prime}(\cdot)(U_1\oplus U_2)$ and so forth, which will terminate after finitely many steps, since the ranks of the projections involved have to decrease, thus giving a map $\tilde{\mathcal{T}}(\cdot)=\tilde{U}\mathcal{T}(\cdot)\tilde{U}^{\dagger}$, which admits the stated decomposition.
\end{proof}

\section{Gurvits' proof of scaling and approximate scaling} \label{app:gurvitsproof}
This appendix provides the details of Gurvits' approach, hence it does not contain original material. For easier readability, we repeat all Lemmata.

\subsection{Approximate scalability}
We need a way to study scalability:
\begin{dfn}
Let $C_1,C_2\in\mathcal{M}_n$ and $\mathcal{E}:\mathcal{M}_n\to\mathcal{M}_n$ a positive, linear map. Then we define a \emph{locally scalable functional} to be a map $\varphi\in\overline{\mathcal{C}^d}^*$ such that
\begin{align}
	\varphi(C_1\mathcal{E}(C_2^{\dagger}\cdot C_2)C_1^{\dagger})=\detm(C_1C_1^{\dagger})\detm(C_2C_2^{\dagger})\varphi(\mathcal{E}) \label{eqn:locscale}
\end{align}
A locally scalable functional will be called \emph{bounded}, if $|\varphi(\mathcal{E})|\leq f(\tr(\mathcal{E}(\id)))$ for some function $f$.
\end{dfn}

Locally bounded functionals are the right tools to study scalability:
\begin{prop} \label{prop:locbdd}
Let $\mathcal{E}:\mathcal{M}_n\to\mathcal{M}_n$ be a positive, linear map. Given a bounded locally scalable functional $\varphi$ such that $\varphi(\mathcal{E})\neq 0$, the Sinkhorn-iteration procedure converges:
\begin{align}
	\operatorname{DS}(\mathcal{E}_n)\to 0 \quad n\to \infty
\end{align}
\end{prop}
\begin{proof} We follow \cite{gur04}. Recall the definitions of the Sinkhorn iteration in equations (\ref{eqn:iterate1})-(\ref{eqn:iterate2}). Because of property (\ref{eqn:locscale}), we have 
\begin{align*}
	\varphi(\mathcal{E}_{i+1})&=a(i)\varphi(\mathcal{E}_i) \\
	a(i)&=\begin{cases} 
		\detm (\mathcal{E}_i^*(\id))^{-1} & \mathrm{if~} i\mathrm{~odd} \\
		\detm (\mathcal{E}_i(\id))^{-1} & \mathrm{if~} i\mathrm{~even} 
	\end{cases}
\end{align*}
Let $i$ be even. Note that $\mathcal{E}_i$ is trace-preserving for $i$ even, hence $\tr(\mathcal{E}_i(\id))=n$. Let $s_j^{(i)}$ be the singular values of $\mathcal{E}_i(\id)$ and observe:
\begin{align}
	|\detm (\mathcal{E}_i(\id))|=\prod_{j=1}^ns_j^{(i)}\leq \frac{1}{n} \sum_{j=1}^n s_j^{(i)} = \frac{1}{n}\tr(\mathcal{E}_i(\id))=1
\end{align}
using the arithmetic-geometric mean inequality (AGM). 
Similarly, for $i$ odd, $\mathcal{E}_i$ is unital, hence $\tr(\mathcal{E}^*_i(\id))=n$ and we can use the AGM inequality again to obtain that $a(i)\geq 0$ for all $i\geq 0$ and therefore
\begin{align*}
	|\varphi(\mathcal{E}_{i+1})|\geq |\varphi(\mathcal{E}_i)|
\end{align*}
and thus, as $\varphi$ was assumed to be bounded, $|\varphi(\mathcal{E}_i)|$ converges to some value $c\leq f(\tr(\mathcal{E}(\id)))$. 

It remains to prove that for $|\varphi(\mathcal{E})|\neq 0$, $\mathrm{DS}(\mathcal{E}_i)$ converges to zero for $i\to\infty$. The idea is of course that if $|\varphi(\mathcal{E})|\neq 0$, then $|a(i)|$ converges to one and thus $\mathcal{E}_i(\id)$ and $\mathcal{E}_i^*(\id)$ converge to $\id$. 

To make this more formal, since $|\varphi(\mathcal{E})|$ converges, for all $\varepsilon>0$ there exists $N\in\mathbb{N}$ such that for all $d\geq N$:
\begin{align*}
	||\varphi(\mathcal{E}_d)|-|\varphi(\mathcal{E}_{d+1})|&\leq \varepsilon \\
	\Leftrightarrow\quad \left| |\varphi(\mathcal{E}_d)|-\frac{1}{|a(i)|}|\varphi(\mathcal{E}_d)|\right|&\leq \varepsilon \\
	\Leftrightarrow\quad |a_i|&\geq \frac{1}{1+ \varepsilon |\varphi(\mathcal{E}_d)|^{-1}}\geq \frac{1}{1+ \varepsilon |\varphi(\mathcal{E})|^{-1}}
\end{align*}
where we used that $|\varphi(\mathcal{E}_d)|$ increases monotonically in the last inequality. 

Let us now only consider $i$ even. Then we have just seen that
\begin{align*}
	\frac{1}{1+ \varepsilon |\varphi(\mathcal{E})|^{-1}}\leq \detm(T_i(\id))\leq 1
\end{align*}
hence, for $i\geq N$ even, we have:
\begin{align*}
	\mathrm{DS}(\mathcal{E}_i)=\tr((\mathcal{E}_i(\id)-\id)^2)=\sum_{j=1}^n (s_j^{(i)}-1)^2
\end{align*}
where the $s_j^{(i)}$ are the singular values of $\mathcal{E}_i(\id)$. If we can upper bound the last quantity by $\tilde{\varepsilon}(\varepsilon)$, we are done. This is an exercise in using logarithms:

Since $\mathcal{E}_i$ is trace-preserving as $i$ is even, $s_j^{(i)}\leq d$ for all $i$. If we set $\alpha:=\frac{(n-1)-\ln (n)}{(n-1)^2}$, then by strict concavity of the logarithm, 
\begin{align*}
	\ln(x)\leq (x-1)-\alpha(x-1)^2 \qquad x\leq d
\end{align*}
since $\ln(d)=(x-1)-\alpha(x-1)^2$ and $\ln(1)=0$. But then:
\begin{align*}
	0&\leq \sum_{j=1}^n (s_j^{(i)}-1)^2\leq \sum_{j=1}^n \left( \frac{s_j^{(i)}-1}{\alpha}-\frac{\ln(s_j^{(i)})}{\alpha}\right) \\
	&=-\sum_{i=1}^n \frac{\ln(s_j^{(i)})}{\alpha} \\
	&=-\frac{1}{\alpha}\ln(\prod_{i=1}^n s_j^{(i)}) \leq -\frac{1}{\alpha}\ln(1-\varepsilon)\\
	&\leq \frac{\varepsilon}{\alpha}
\end{align*}
where we used that $\sum_{j=1}^n s_j^{(i)}=\tr(\mathcal{E}_i(\id))=n$. But $\frac{\varepsilon}{\alpha}\to 0$ for $i\to\infty$.

Exchanging $\mathcal{E}_i$ with $\mathcal{E}_i^*$ gives the same reasoning for odd $i$. In total, we get that for any $\varepsilon>0$ exists an $N\in \mathbb{N}$ such that for all $d\geq n$
\begin{align*}
	0\leq \mathrm{DS}(\mathcal{E}_d)\leq \frac{\varepsilon}{\alpha}
\end{align*}
hence $\mathcal{DS}(\mathcal{E}_i)\to 0$ for $i\to \infty$.
\end{proof}

\begin{lem} \label{lem:caplocbdd}
$\operatorname{Cap}$ is a bounded locally scalable functional.
\end{lem}
\begin{proof}
Note that for
\begin{align*}
 \inf&\{\detm (C_2^{\dagger}\mathcal{E}(C_1XC_1^{\dagger})C_2)|X>0, \detm (X)=1\} \\
	&=\inf\{\detm(C_2^{\dagger})\detm(C_2)\detm(\mathcal{E}(C_1XC_1^{\dagger}))|X>0, \detm(X)=1\} \\
	&=\detm(C_2^{\dagger}C_2)\inf\{\detm(\mathcal{E}(C_1XC_1^{\dagger}))|X>0, \detm(X)=1\} \\
	&=\detm(C_2^{\dagger}C_2)\inf\{\detm(\mathcal{E}(\tilde{X}))|X>0, \detm(\tilde{X})=\detm(C_1)\detm(C_1^{\dagger})\detm(X), \detm(X)=1\} \\
	&=\detm(C_2^{\dagger}C_2)\detm(C_1^{\dagger}C_1)\inf\{\detm(\mathcal{E}(\tilde{X}))|\tilde{X}>0, \detm(\tilde{X})=1\}
\end{align*}
hence $\operatorname{Cap}$ is a locally scalable functional. Via the AGM inequality, we have
\begin{align*}
	0\leq \operatorname{Cap}(\mathcal{E})\leq \detm(\mathcal{E}(\id))\leq \left(\frac{\tr(\mathcal{E}(\id))}{n}\right)^{\frac{1}{n}}
\end{align*}
hence $\operatorname{Cap}$ is bounded.
\end{proof}

This gives a proof of Lemma \ref{lem:approxscaling}.

\begin{lem}[Lemma \ref{lem:tuplecap} of the main text]
Let $\mathcal{E}:\mathcal{M}_n\to\mathcal{M}_n$ be a positive, linear map and $U\in U(n)$ a fixed unitary. Then defining
\begin{align*}
	\operatorname{Cap}(\mathbf{A}_{\mathcal{E},U}):=\inf \left\{\detm \left(\sum_i \mathcal{E}(u_iu_i^{\dagger})\gamma_i\right) | \gamma_i>0, \prod_{i=1}^n\gamma_i=1\right\}
\end{align*}
where $u_i$ are once again the rows of $U$, we have the following properties:
\begin{enumerate}
	\item Using the mixed discriminant $M$, we have
	\begin{align*}
		M(\mathbf{A}_{\mathcal{E},U})\leq \operatorname{Cap}(\mathbf{A}_{\mathcal{E},U})\leq \frac{n^n}{n!}M(\mathbf{A}_{\mathcal{E},U})
	\end{align*}
	\item $\inf\limits_{U\in U(n)} \operatorname{Cap}(\mathbf{A}_{\mathcal{E},U})=\operatorname{Cap}(\mathcal{E})$
\end{enumerate}
\end{lem}
\begin{proof}
The first part of the lemma is one of the main results of \cite{gur02a}. Since the proof is long due to many technicalities, we leave it out here.

The second part gives the relation between the two capacities.  Let $\{X_d\}_d$ with $\detm (X_d)=1$ and $X_d>0$ be such that $\detm (\mathcal{E}(X_d))\to \operatorname{Cap}(\mathcal{E})$, $d\to\infty$. Then there exist unitaries $U_d\in U(n)$ such that $U_dX_nU_d^{\dagger}$ is diagonal with diagonal entries $\lambda_i^{(d)}$. By construction,
\begin{align*}
	\detm(\sum_{1\leq i\leq n} \mathcal{E}((u_d)_i(u_d)_i^{\dagger})\lambda_i^{(d)})=\detm (\mathcal{E}(X_d))
\end{align*}
where $(u_d)$ are again the columns of $U_d$. Hence
\begin{align*}
	\inf_{U\in U(n)} \operatorname{Cap}(\mathbf{A}_{\mathcal{E},U})\leq \operatorname{Cap}(\mathcal{E})
\end{align*}

Likewise, we can construct a sequence of $U_d$ such that $\operatorname{Cap}(\mathbf{A}_{\mathcal{E},U_d})$ converges to the infimum and we can construct a sequence $(\gamma_{(k)}^{(d)})_k$ with $(\gamma_{(k)}^{(d)})_i>0$ for each $U_d$ such that 
\begin{align*}
	\detm\left(\sum_{i=1}^n \mathcal{E}((u_d)_i(u_d)_i^{\dagger})(\gamma_{(k)}^{(d)})_i\right)\to \operatorname{Cap}(\mathbf{A}_{\mathcal{E},U_d}) \qquad \mathrm{for~}k\to \infty
\end{align*}
Taking the diagonal sequence $\gamma_{(d)}^{(d)}$ we obtain a sequence converging to $\inf \operatorname{Cap}(\mathbf{A}_{\mathcal{E},U})$. Finally, define $X_k=U_k\diag(\gamma^{(k)}_{(k)})U_k^{\dagger}$, then $X_k>0$ and 
\begin{align*}
	\detm(\mathcal{E}(X_k))=	\detm\left(\sum_{i=1}^n \mathcal{E}((u_d)_i(u_d)_i^{\dagger})(\gamma_{(k)}^{(d)})_i\right)
\end{align*}
and hence
\begin{align*}
	\operatorname{Cap}(\mathcal{E})\leq \inf_{U\in U(n)} \operatorname{Cap}(\mathbf{A}_{\mathcal{E},U})
\end{align*}
after taking the limit $k\to \infty$.
\end{proof}

Finally, we can write down the Operator Sinkhorn theorem (Theorem \ref{thm:approxsinkquant} in the main text):
\begin{thm}[Approximate Operator Sinkhorn Theorem, \cite{gur04} Theorem 4.6]
Let $\mathcal{E}:\mathcal{M}_n\to\mathcal{M}_n$ be a positive, linear map. Then $\mathcal{E}$ is $\varepsilon$-scalable for all $\varepsilon>0$ iff $\mathcal{E}$ is rank non-decreasing. 
\end{thm}
\begin{proof}
We mostly need to combine the lemmas. By lemma \ref{lem:caplocbdd}, the capacity is a bounded, locally scalable functional, which implies by proposition \ref{prop:locbdd} that $\operatorname{DS}(\mathcal{E}_i)$ converges, if $\operatorname{Cap}(\mathcal{E})>0$. Now, by lemma \ref{lem:tuplecap}, 
\begin{align*}
	\operatorname{Cap}(\mathcal{E})=\inf\{\operatorname{Cap}(\mathbf{A}_{\mathcal{E},U})|U\in U(n)\}
\end{align*}
Since $U(n)$ is compact, it suffices to show that for every $U$, $\operatorname{Cap}(\mathbf{A}_{\mathcal{E},U})>0$. Again, by lemma \ref{lem:tuplecap}, 
\begin{align*}
	\operatorname{Cap}(\mathbf{A}_{\mathcal{E},U})\geq M(\mathbf{A}_{\mathcal{E},U})
\end{align*}
but $M(\mathbf{A}_{\mathcal{E},U})>0$ for every $U$ if and only if $\mathcal{E}$ is rank non-decreasing by proposition \ref{prop:ranknon}. Hence, $\operatorname{DS}(\mathcal{E}_i)$ converges for rank non-decreasing maps.

Now suppose that $\mathcal{E}$ is a positive map such that in the Sinkhorn iteration, $\mathrm{DS}(\mathcal{E}_i)$ converges. Then, for some $i\in\mathbb{N}$, $\mathrm{DS}(\mathcal{E}_i)<\frac{1}{n}$. We claim that then $\mathcal{E}_i$ is rank non-decreasing and by consequence, also $\mathcal{E}$ is rank non-decreasing via proposition \ref{prop:ranknon}.

To see this, assume $\mathcal{E}(\id)=\id$ and $\mathcal{E}^*(\id)=\id+E$, where $E$ is Hermitian and $\tr(E^2)\leq 1/n$. We can do this, because this is exactly what $\mathcal{E}_i$ looks like for $i$ big enough such that $\mathrm{DS}(\mathcal{E}_i)<\frac{1}{n}$ and $i$ is odd. Given a matrix $U\in U(n)$ and the corresponding $A:=\mathbf{A}_{\mathcal{E},U}$, we have that
\begin{align*}
	\sum_{i=1}^n A_i=\mathcal{E}(\id)=\id
\end{align*}
Likewise, for every $i$:
\begin{align}
	\tr(A_i)=\tr(A_i\id)=\tr(u_iu_i^{\dagger}T^*(\id))=1+\tr(u_iu_i^{\dagger}E)=:1+\delta_i \label{eqn:proof5}
\end{align}
But by assumption,
\begin{align}
	\sum_{i=1}^n |\delta_i|^2 & \leq \sum_{i,j=1}^n |\tr(u_iu_j^{\dagger}E)|^2 \nonumber \\
	&=\sum_{i,j=1}^n \langle u_i|E|u_j\rangle\langle u_j|E^{\dagger}|u_i\rangle \nonumber \\
	&=\tr(E^2)\leq \frac{1}{n} \label{eqn:proof6}
\end{align}
Now, suppose that $\mathcal{E}$ is not rank non-decreasing. Then, by proposition \ref{prop:fullyindec} (vii), there is a $U$ such that $\mathbf{A}_{\mathcal{E},U}$ fulfills
\begin{align*}
	\rank \left(\sum_{i=1}^k A_i\right)<k
\end{align*}
for some $0<k<n$. Note that, since $\mathcal{E}$ is positive, $A_i\geq 0$, hence $H:=\sum_{i=1}^k A_i$ fulfills $0\leq H\leq \id$. As $\rank(H)\leq k-1$, we have $\tr(H)\leq k-1$. From equation (\ref{eqn:proof5}), we obtain 
\begin{align*}
	\tr(H)=\sum_{i=1}^k A_i=k+\sum_{i=1} k \delta_i
\end{align*}
Using equation (\ref{eqn:proof6}), by the Cauchy Schwarz inequality, 
\begin{align*}
	\sum_{i=1}^k|\delta_i|\leq \sqrt{k/n}<1
\end{align*}
which contradicts $\tr(H)\leq k-1$, hence $\mathcal{E}$ must be rank non-decreasing.
\end{proof}

\subsection{Exact scalability}
\begin{lem}[Lemma \ref{lem:gurscaleequiv} of the main text]
Let $\mathcal{E}:\mathcal{M}_d\to\mathcal{M}_d$ be a positive map. Then $\mathcal{E}$ is scalable to a doubly-stochastic map if and only if $\operatorname{Cap}(\mathcal{E})>0$ and the capacity can be achieved.
\end{lem}
\begin{proof}
Suppose there exists $C>0$ with $\det(\mathcal{E}(C))=\operatorname{Cap}(\mathcal{E})$. The Lagrangian of the capacity is
\begin{align*}
	\mathcal{L}(X):=\ln(\detm(\mathcal{E}(X)))+\lambda \ln(\detm(X))
\end{align*}
with the Lagrangian multiplier $\lambda\in\mathbb{R}$. Therefore, the minimum fulfills
\begin{align}
	\nabla \ln(\detm(\mathcal{E}(X)))|_{X=C}=(-\lambda)\nabla \ln(\det(X))|_{X=C} \label{eqn:minimumcond} 
\end{align}
We claim that the conditions are equivalent to
\begin{align}
	\mathcal{E}^*((\mathcal{E}(C))^{-1})^{-1}=C^{-1} \label{eqn:differentiate}
\end{align}
Let $E_{ij}$ be the usual matrix unit, then
\begin{align*}
	(\nabla \ln(\detm(\mathcal{E}(X)))|_{X=C})_{jk}&=\frac{\partial}{\partial E_{jk}} \ln \left(\sum_{\sigma\in S_n} \operatorname{sgn}(\sigma) \prod_{i=1}^n \mathcal{E}(C)_{i\sigma(i)}\right) \\
	&=\frac{1}{\detm(\mathcal{E}(C))}\sum_{\sigma\in S_n} \operatorname{sgn}(\sigma) \frac{\partial}{\partial E_{jk}}\prod_{i=1}^n\mathcal{E}(C)_{i\sigma(i)}.
\end{align*}
Noting that
\begin{align*}
	\partial E_{jk}\mathcal{E}(C)_{i\sigma(i)} = \tr(E_{\sigma(i)i}\frac{\partial \mathcal{E}(C)}{\partial E_{jk}}) =\tr(\mathcal{E}^*(E_{\sigma(i)i})E_{jk})=\mathcal{E}^*(E_{\sigma(i)i})_{jk}
\end{align*}
we have
\begin{align*}
	(\nabla \ln(&\detm(\mathcal{E}(X)))|_{X=C})_{jk}=\frac{1}{\detm(\mathcal{E}(C))}\sum_{\sigma\in S_n} \operatorname{sgn}(\sigma)\sum_{l=1}^n\mathcal{E}^*(E_{\sigma(l)l})_{jk}\prod_{i\neq l} \mathcal{E}(C)_{i\sigma(i)} \\
	&=\mathcal{E}^*\left(\frac{1}{\detm(\mathcal{E}(C))}\sum_{\sigma\in S_n} \operatorname{sgn}(\sigma)\sum_{l=1}^nE_{\sigma(l)l}\prod_{i\neq l} \mathcal{E}(C)_{i\sigma(i)}\right)_{jk} \\
		&=\mathcal{E}^*\left(\frac{1}{\detm(\mathcal{E}(C))}\sum_{m,n=1}^n\left(\sum_{\sigma(m)=n\in S_n} \operatorname{sgn}(\sigma)(-1)^{n-m}\prod_{i\neq m} \mathcal{E}(C)_{i\sigma(i)}\right)E_{mn}\right)_{jk} \\
		&=\mathcal{E}^*(\mathcal{E}(C)^{-1})_{jk}
\end{align*}
where in the last step we use Cramer's rule. For $\mathcal{E}=\operatorname{id}$ we obtain the right hand side of equation (\ref{eqn:differentiate}) from equation (\ref{eqn:minimumcond}), hence follows the claim. It now follows from Lemma \ref{lem:sinkhornderivative} that any $C$ fulfilling Equation (\ref{eqn:differentiate}) defines a scaling. 

Conversely, suppose $\tilde{\mathcal{E}}(\cdot)=C_1\mathcal{E}(C_2^{\dagger}C_2)C_1^{\dagger}$ is a doubly stochastic map. Since $\tilde{\mathcal{E}}$ is unital, the eigenvalues of $\tilde{\mathcal{E}}(X)$ are majorized by the eigenvalues of $X$ (cf. \cite{wol12} Theorem 8.8). Majorization stays invariant under strictly increasing functions (cf. \cite{bha96}, Chapter 1), hence we have ($\lambda_i$ being the eigenvalues of $X$ and $\lambda_i^{\tilde{\mathcal{E}}}$ the eigenvalues of $\tilde{\mathcal{E}}(X)$):
\begin{align*}
	\sum_i -\ln(\lambda_i^{\tilde{\mathcal{E}}})\leq \sum_i -\ln(\lambda_i) 
\end{align*}
which is equivalent to $\detm(\tilde{\mathcal{E}}(X))\geq \detm(X)$. Hence, a doubly stochastic map is in particular determinant increasing. Obviously, equality is attained at $X=\id$. But then:
\begin{align}
	\detm (\mathcal{E}(X))&=|\detm (C_1)|^{-2}|\detm(C_2)|^{-2}\detm (\tilde{\mathcal{E}}(X)) \\
		&\geq |\detm (C_1)|^{-2}|\detm(C_2)|^{-2}\detm (X). \label{eqn:proof1}
\end{align}
A quick calculation shows that $X=C_2^{\dagger}C_2/\detm(C_2^{\dagger}C_2)^{1/n}$ attains equality in Equation (\ref{eqn:proof1}). This then necessarily minimises the capacity.
\end{proof}

\begin{lem}[Lemma \ref{lem:convexgurvits} of the main text]
Let $\mathcal{E}:\mathcal{M}_n\to\mathcal{M}_n$ be a positive, linear map and given $U\in U(n)$, let $A=\mathbf{A}_{\mathcal{E},U}$. Then 
\begin{enumerate}
	\item $f_A$ is convex on $\mathbb{R}^n$.
	\item If $\mathcal{E}$ is fully indecomposable, then $f_A$ is strictly convex on $\{\xi=(\xi_1,\ldots, \xi_n)\in\mathbb{R}^n|\sum_i \xi_i=0\}$.
\end{enumerate}
\end{lem}
\begin{proof}
We follow the proof of \cite{gur02a}. Given a tuple $A$ of positive definite matrices, one can show (\cite{bap89b}):
\begin{align}
	\detm (e^{\xi_1}A_1+\ldots+e^{\xi_n}A_n)=\sum_{r\in P_n} t_r e^{(\xi,r)} \label{eqn:proof3}
\end{align}
where $(\cdot,\cdot)$ denotes the general inner product, $P_n$ is the set of $n$-tuples of integers $r_i\geq 0$ such that $\sum_i r_i=n$ and 
\begin{align}
	t_r:=\frac{1}{r_1!\ldots r_n!}M(\overbrace{A_1,\ldots,A_1}^{r_1},\ldots,\overbrace{A_n,\ldots,A_n}^{r_n}) \label{eqn:proof4}
\end{align}

This implies that we can rewrite $f_A$:
\begin{align*}
	f_A(\xi_1,\ldots, \xi_n)=\ln \detm(e^{\xi_1}A_1+\ldots+e^{\xi_n})=\ln\left(\sum_{r\in P_n} t_re^{(\xi,r)}\right)
\end{align*}
It is well known that for positive matrices this is a convex function, but let us follow the proof of \cite{gur02a} here. 

Let $f=:\ln g$. We need to prove that $\nabla^2 f$, the Hessian, is positive (semi)definite. By definition, $\nabla^2f=\frac{1}{g^2}(g(\nabla^2g)-(\nabla g)(\nabla g)^{tr})$, hence it is sufficient that $g(\nabla^2g)\geq (\nabla g)(\nabla g)^{tr}$. 

Note that for any $v\in\mathbb{R}^n$ we have $\nabla e^{(\xi,v)} =e^{(\xi,v)}\cdot v$ and $\nabla^2e^{(\xi,v)}=e^{(\xi,v)}vv^{tr}$, where $vv^{\tr}$ is positive definite. Therefore:
\begin{align*}
	g(\nabla^2g)-(\nabla g)(\nabla g)^{tr}&=\sum_{r\in P_n} t_re^{(\xi,r)}\cdot \sum_{s\in P_n} t_re^{(\xi,s)}ss^{tr}-\sum_{r,s\in P_n} t_rt_se^{(\xi,r+s)}rs^{tr} \\
	&=\frac{1}{2} \sum_{r,s\in P_n} t_rt_se^{(\xi,r+s)}(r-s)(r-s)^{tr} \geq 0
\end{align*}
hence the Hessian of $f$ is positive semi-definite and therefore $f$ is convex. 

Now, assume that $\mathcal{E}$ is fully indecomposable, hence the tuple $A:=\mathbf{A}_{\mathcal{E},U}$ fulfills proposition \ref{prop:fullyindec} (vii) and (viii) for all $U\in U(n)$. In particular, if $A^{ij}$ is the tuple $A$ with the $j$-th entry being replaced by the $i$-th. entry. Then $M(A^{ij})>0$ in particular. Note that $M(A^{ij})=2t_{r_{ij}}$ by equation (\ref{eqn:proof4}), where
\begin{align*}
	(r_{ij})_k:=\begin{cases} 2 & k=i \\ 0 &k=j \\ 1\\ \mathrm{else} \end{cases}.
\end{align*}
Then, 
\begin{align*}
	\nabla^2f &\geq \frac{1}{g^2} \sum_{r,s\in P_n} t_rt_se^{(\xi,r+s)}(r-s)(r-s)^{tr} \\
	&\geq \frac{1}{8g^2} \sum_{i\neq j, k\neq l} M(A^{ij})M(A^{kl})e^{(\xi,r_{ij}+s_{kl})}(r_{ij}-r_{kl})(r_{ij}-r_{kl})^{tr} \\
	&\geq \frac{cM^2}{8g^2} \sum_{i\neq j, k\neq l} (r_{ij}-r_{kl})(r_{ij}-r_{kl})^{tr}
\end{align*}
where $c:=\min_{i\neq j\neq k\neq l} e^{(\xi,r_{ij}-r_{kl})}$ and $M:=\min_{i\neq j} M(A^{ij})>0$ by proposition \ref{prop:fullyindec} (viii). 

We only need to consider $\sum_{i\neq j, k\neq l} (r_{ij}-r_{kl})(r_{ij}-r_{kl})^{tr}=:S$ and show that this is a positive definite matrix on the hyperplane $H$. Using the usual matrix units $E_{mn}$ we can write:
\begin{align*}
	S:=\sum_{i\neq j\neq k\neq l} (E_{ii}+E_{jj}+E_{kk}+E_{ll}+2(E_{il}+E_{jk}-E_{ik}-E_{jl}-E_{kl}))
\end{align*}
We find that $S_{ii}=(n-1)(n-2)(n-3)$, since only the first four summands contribute to the diagonal terms. For the off-diagonal terms, note that in $2(E_{il}+E_{jk}-E_{ik}-E_{jl}-E_{kl})$, all unordererd combinations of $i,j,k,l$ occur, twice with a positive sign and four times with a negative. Hence we obtain $(n-2)\cdot (n-3)$ terms with either $E_{ij}$ or $E_{ji}$ that are not cancelled by other terms and therefore $S_{ij}=-(n-2)(n-3)$. In short:
\begin{align*}
	S=(n-1)(n-2)(n-3)\begin{pmatrix}{} 
		1 & \frac{1}{n-1} &\ldots &\frac{1}{n-1} \\
		\frac{1}{n-1} & 1 &\ldots &\frac{1}{n-1} \\
		\vdots &  & \ddots & \vdots \\
		\frac{1}{n-1} & \frac{1}{n-1} & \ldots & 1 
		\end{pmatrix}
\end{align*}
Note that the image of $S$ is just the hyperplane $H$ and it is easy to see that $S$ is a multiple of the projection onto the hyperplane $S$. Therefore, $\nabla^2f$ is strictly convex on $H$.
\end{proof}

Finally, we obtain the theorem:
\begin{lem}[Lemma \ref{lem:quantumgurvits} of the main text]
Let $\mathcal{E}:\mathcal{M}_n\to\mathcal{M}_n$ be a positive, linear map. If $\mathcal{E}$ is fully indecomposable, there exists a unique scaling of $\mathcal{E}$ to a doubly stochastic map.
\end{lem}
\begin{proof}
Recall that one can show (\cite{bap89b}):
\begin{align}
	\detm (e^{\xi_1}A_1+\ldots+e^{\xi_n}A_n)=\sum_{r\in P_n} t_r e^{(\xi,r)} \label{eqn:proof2}
\end{align}
where $(\cdot,\cdot)$ denotes the general inner product, $P_n$ is the set of $n$-tuples of integers $r_i\geq 0$ such that $\sum_i r_i=n$ and $t_r:=\frac{1}{r_1!\ldots r_n!}M(\overbrace{A_1,\ldots,A_1}^{r_1},\ldots,\overbrace{A_n,\ldots,A_n}^{r_n})$. 

Suppose $X\geq 0$, $\detm(X)=1$ and $\mathcal{E}$ is fully indecomposable. Let $U\in U(n)$ diagonalize $X$ with eigenvalues $\gamma_i=e^{\xi_i}$. Assume the eigenvalues are ordered $\gamma_1\geq \ldots\geq \gamma_n$. Observe that then $\detm(\mathcal{E}(X))\leq \detm(\mathcal{E}(\id))$ is equivalent to say that $f_A(\xi)\leq f_A(0)$, where $A=\mathbf{A}_{\mathcal{E},U}$. We know:
\begin{align*}
	\detm (A_1+\ldots + A_n)&\geq \detm(\gamma_1A_1+\ldots +\gamma_n A_n) \\
		&=\sum_{r\in P_n}t_re^{(\xi,r)}
		\geq \frac{1}{2} \sum_{i\neq j} M_{ij} e^{(\xi,r_{ij})}
\end{align*}
where we use that certainly for all $i\neq j\in\{1,\ldots,n\}$, $r_{ij}:=$ with $r_k=1$ for all $k\neq i,j$ and $r_i=2$, $r_j=0$ is a valid $n$-tuple where the coefficient $t_r=\frac{1}{2}M^{ij}$. Since all the terms in the sum of equation \ref{eqn:proof2} are positive, we can just leave out all other $r$. By definition, $\overline{M}(\mathcal{E})\leq M^{ij}$ for every $A$, hence:
\begin{align*}
	\detm (A_1+\ldots + A_n)&\geq \frac{1}{2}\overline{M}(\mathcal{E})\sum_{i\neq j} e^{(\xi,r_{ij})} \\
	&\geq \frac{1}{2} \overline{M}(\mathcal{E})e^{\max_{i\neq j} (\xi_i-\xi_j)} \\
	&\geq \frac{1}{2} \overline{M}\frac{\gamma_1}{\gamma_N}
\end{align*}
where we used that $(\xi,r_{ij})=\sum_{k\neq j} \xi_k + \xi_i=\xi_i-\xi_j$ since $\sum_i \xi_i=0$. 
Since $\detm(A_1+\ldots+\ldots)=\detm(\mathcal{E}(\id))$, we have
\begin{align*}
	\frac{\gamma_1}{\gamma_n}\leq \frac{2\detm (T(\id))}{M(A)}\leq \frac{2\detm(\mathcal{E}(\id))}{\overline{M}}<\infty
\end{align*}
from the lemma above. But then, the infimum must be attained on the compact subset $\{\detm(X)=1 | \gamma_1\leq \frac{2\detm(\mathcal{E}(X))}{\overline{M}}\}$. Therefore, also for the capacity $\operatorname{Cap}(\mathcal{E})$ the infimum can be considered on a compact subset of $\{\detm(X)=1\}$ and is then attained. Uniqueness is ensured by the strict convexity of $f_A$. 
\end{proof}
\end{document}